\documentclass[11pt, twoside]{article}
\usepackage{amssymb}
\usepackage{amsfonts}
\usepackage{amsmath}
\usepackage{amsthm}
\usepackage{color}
\usepackage{mathrsfs}
\usepackage{txfonts}
\usepackage{enumerate}
\usepackage{indentfirst}
\usepackage{graphicx}

\usepackage{anysize}

\allowdisplaybreaks

\newtheorem{theorem}{Theorem}[section]
\newtheorem{lemma}[theorem]{Lemma}
\newtheorem{corollary}[theorem]{Corollary}
\newtheorem{proposition}[theorem]{Proposition}

\theoremstyle{definition}

\newtheorem{remark}[theorem]{Remark}
\newtheorem{definition}[theorem]{Definition}

\renewcommand{\appendix}{\par
\setcounter{section}{0}%
\setcounter{subsection}{0}%
\setcounter{subsubsection}{0}%
\gdef\thesection{\@Alph\c@section}%
\gdef\thesubsection{\@Alph\c@section.\@arabic\c@subsection}%
\gdef\theHsection{\@Alph\c@section.}%
\gdef\theHsubsection{\@Alph\c@section.\@arabic\c@subsection}%
\csname appendixmore\endcsname
}

\numberwithin{equation}{section}

\begin{document}

\arraycolsep=1pt

\title{\bf\Large
Sharp Brezis--Seeger--Van Schaftingen--Yung Formulae
for Higher-Order Gradients in Ball Banach Function Spaces
\footnotetext{\hspace{-0.35cm} 2020
\emph{Mathematics Subject Classification}. Primary 46E35; Secondary
26D10, 35A23, 42B25, 42B35.
\endgraf \emph{Key words and phrases}.
Brezis--Seeger--Van Schaftingen--Yung formula, ball Banach function
space, higher-order difference,
fractional Gagliardo--Nirenberg-type inequality.
\endgraf
This project is partially
supported by the National Key Research and Development Program of China
(Grant No. 2020YFA0712900), the National Natural Science Foundation of
China (Grant Nos. 12431006, 12371093, and 124B2004), the Fundamental Research Funds
for the Central Universities (Grant No. 2233300008).
}}
\author{Pingxu Hu, Yinqin Li\footnote{Corresponding
author,
E-mail: \texttt{yinqli@mail.bnu.edu.cn} /{\color{red}\today}/Final version.},
\ Dachun Yang, Wen Yuan and Yangyang Zhang
}
\date{}
\maketitle

\vspace{-0.8cm}

\begin{center}
\begin{minipage}{13cm}
{\small {\bf Abstract}\quad
Let $X$ be a ball Banach function space on $\mathbb{R}^n$,
$k\in\mathbb{N}$,
$h\in\mathbb{R}^n$, and $\Delta^k_h$ denote the $k${\rm th} order
difference.
In this article, under some mild extra assumptions about $X$,
the authors prove that, for both parameters $q$ and $\gamma$ in
\emph{sharp}
ranges which are related to $X$ and
for any locally integrable function $f$ on ${\mathbb{R}^n}$
satisfying $|\nabla^k f|\in X$,
$$
\sup_{\lambda\in(0,\infty)}\lambda
\left\|\left[\int_{\{h\in\mathbb{R}^n:\
|\Delta_h^k f(\cdot)|>\lambda|h|^{k+\frac{\gamma}{q}}\}}
\left|h\right|^{\gamma-n}\,dh\right]^\frac{1}{q}\right\|_X
\sim
\left\|\,\left|\nabla^k f\right|\,\right\|_{X}
$$
with the positive equivalence constants independent of $f$.
As applications, the authors establish the
Brezis--Seeger--Van Schaftingen--Yung (for short, BSVY)
characterization of higher-order homogeneous ball Banach Sobolev spaces and
higher-order fractional Gagliardo--Nirenberg and
Sobolev type inequalities in critical cases.
All these results are of quite wide generality
and can be applied
to various specific function spaces; moreover,
even when $X:= L^{q}$,
these results when $k=1$ coincide with the best known results
and when $k\ge 2$ are completely new.
The first novelty is
 to establish a sparse characterization of dyadic cubes
in level sets related to the higher-order local
approximation, which, together with
the well-known Whitney inequality in approximation theory,
further induces a higher-order
weighted variant of the remarkable inequality obtained
by A. Cohen, W. Dahmen, I. Daubechies, and R. DeVore;
the second novelty is to combine
this weighted inequality neatly with a variant
higher-order Poincar\'e inequality
to establish the desired upper estimate of BSVY formulae in weighted Lebesgue spaces.
}
\end{minipage}
\end{center}

\vspace{0.2cm}

\tableofcontents

\vspace{0.2cm}

\section{Introduction}

Throughout this article, we work in $\mathbb{R}^n$
and, unless necessary, we will not explicitly mention 
this underlying space.
Recall that, for any given $s\in (0,1)$ and $p\in [1,\infty)$,
the \emph{homogeneous fractional
Sobolev space} $\dot{W}^{s,p}$
is
defined to be the set of all $f\in \mathscr{M}$
having the following finite \emph{Gagliardo semi-norm}
\begin{align}\label{eq-GSN}
\|f\|_{\dot{W}^{s,p}}
:=\left[\int_{\mathbb{R}^n}
\int_{\mathbb{R}^n}\frac{|f(x+h)-f(x)|^p}{|h|^{n+sp}}\,dh\,dx\right]^{\frac{1}{p}};
\end{align}
here and thereafter,
$\mathscr{M}$ denotes the set of all measurable
functions $f$
on $\mathbb{R}^n$.
A well-known defect of the Gagliardo semi-norm is that it can not
recover the first-order homogeneous Sobolev semi-norm,
that is,
if $p\in [1,\infty)$ and $f$ is a measurable function on
$\mathbb{R}^n$ such that
\begin{align}\label{e-fcc}
\left[\int_{\mathbb{R}^n}
\int_{\mathbb{R}^n}\frac{|f(x+h)-f(x)|^p}{|h|^{n+p}}\,dh\,dx\right]^{\frac{1}{p}}
<\infty,
\end{align}
then $f$ is a constant almost everywhere (see \cite[Proposition
2]{b2002}).
Brezis et al.\ \cite{bsvy24}
repaired the above defect by replacing the $L^p$ norm
with the weak $L^p$ quasi-norm
and
established a one-parameter equivalent
formula of norms of first-order Sobolev spaces
via the size of level sets of suitable
difference quotients,
which is now called the Brezis--Seeger--Van Schaftingen--Yung
(for short, BSVY) formula.
Using this formula, Brezis et al.
gave new characterizations of first-order
homogeneous Sobolev spaces in \cite{bsvy24} and
established critical fractional
Gagliardo--Nirenberg and Sobolev type inequalities in
\cite{bsvy22,bvy21}.
For more developments of BSVY formulae and their applications,
we refer to
\cite{dgpyyz24,dlyyz2022,dm22,dm23,dssvy,dt.arXiv,dt23,f22,f24,lxy24,m24,p22}.
Very recently, Zhu et al.\ \cite{ZYY23}
and Li et al.\ \cite{lyyzz24}
extended the BSVY formula to
a framework of function spaces, called the ball
Banach function space (see Definition \ref{def-X}).
Recall that the concept of ball
(quasi-)Banach function spaces was
introduced by Sawano et al.\ \cite{shyy2017}
to unify the study of several different important function spaces
(see Remark \ref{rm-bqbf}).
Due to its wide generality, ball
(quasi-)Banach function spaces have recently attracted a lot of
attention
and yielded many applications; we refer to
\cite{hcy2021,is2017,tyyz2021,wyy2020,wyyz2021,zwyy2021}
for the boundedness of operators and to
\cite{lyh2320,yhyy2022-1,
yhyy2022-2,yyy2020,zhyy2022} for the real-variable
theory of function spaces.

In the article, we always let $\mathbb{N}:=\{1,2,\ldots\}$.
For any given $k\in \mathbb{N}$, $s\in (0,k)\setminus \mathbb{N}$,
and $p\in [1,\infty)$, the \emph{higher-order homogeneous
fractional Sobolev space} $\dot{W}^{s,p}$
is defined to be the set of all $f\in \mathscr{M}$
such that
\begin{align*}
\|f\|_{\dot{W}^{s,p}}
:=
\left[\int_{\mathbb{R}^n}
\int_{\mathbb{R}^n}
\frac{|\Delta_h^k
f(x)|^p}{|h|^{n+sp}}\,dh\,dx\right]^{\frac{1}{p}}<\infty
\end{align*}
(see \cite[Section 3.5.3]{Tre92}
and \cite[Proposition 2.4]{bm2018});
here and thereafter,
for any given $k\in \mathbb{N}$ and $h\in \mathbb{R}^n$,
the $k${\rm th} order
difference $\Delta^k_h f$ of
$f\in \mathscr{M}$
is defined by setting
\begin{align}\label{eq-sho-hd}
\Delta_{h}^k
f(\cdot):=\sum_{j=0}^{k}(-1)^{k-j}\binom{k}{j}f(\cdot+jh),
\end{align}
where $\binom{k}{j}:=\frac{k!}{j!(k-j)!}$.
In the higher-order case,
we also have a defect similar to \eqref{e-fcc}, that is,
for any given
$k\in \mathbb{N}$ and $p\in [1,\infty)$,
if $f\in \mathscr{M}$
satisfying
\begin{align*}
\left[\int_{\mathbb{R}^n}
\int_{\mathbb{R}^n}
\frac{|\Delta_h^k f(x)|^p}{|h|^{n+kp}}\,dh\,dx\right]^{\frac{1}{p}}
<
\infty,
\end{align*}
then $f$
coincides with a polynomial of degree
at most $k-1$ almost everywhere (see \cite[Proposition 1.3]{FKR15}).
Motivated by the aforementioned works (mainly \cite{bsvy24}), in this article,
we repair this defect by replacing the strong type norm
with the corresponding weak quasi-norm
and establish the higher-order BSVY formulae in ball
Banach function spaces with parameters
in \emph{sharp} ranges (see Theorem \ref{thm-main} and Remark \ref{re42}).
As applications, we establish the
BSVY characterization of the higher-order homogeneous ball Banach
Sobolev space (see Theorem \ref{t-chaWkX})
and higher-order fractional Gagliardo--Nirenberg and
Sobolev type inequalities in critical cases (see Theorem \ref{thm-fsGN}
and Remark \ref{rem-fix}).
All these results are of quite wide generality
and can be applied
to various specific function spaces (see Section \ref{ssec-fract-s}
for details).
The first novelty of these results is
to establish a sparse characterization (Lemma \ref{lem-Qx})
of dyadic cubes in level sets related to the higher-order
local approximation, which, together with
the well-known Whitney inequality \eqref{eq-de1}
in approximation theory, further induces a higher-order
weighted variant (Theorem \ref{esti-omega})
of the remarkable inequality obtained
by Cohen et al. \cite[Theorems 3.1 and 4.1]{cddd03};
the second novelty is to combine
this weighted inequality neatly with a variant higher-order Poincar\'e inequality
(Lemma \ref{Poin})
to establish the desired upper estimate of BSVY formulae
in weighted Lebesgue spaces
(Theorem \ref{thm-up-ap}).
We should also mention that some ideas from
the proof of Frank \cite[Theorem 1.1]{f22}
help us  to circumvent the unknown
Bourgain--Brezis--Mironescu formula in
higher-order ball Banach Sobolev spaces.

Throughout this article,
let $L^1_{\rm loc}$
denote the set of all locally integrable functions $f$ on
${\mathbb{R}^n}$.
For any multi-index
$\alpha=(\alpha_1, \dots, \alpha_n)\in \mathbb{Z}_{+}^n$ and
for any $x=(x_1,\dots,x_n)\in {\mathbb{R}^n}$ and $f\in L^1_{\rm
loc}$,
we denote by $x^{\alpha}:={x_1}^{\alpha_1}\cdots{x_n}^{\alpha_n}$
the monomial of degree $|\alpha|:=\sum_{i=1}^{n}\alpha_i$ and by $\partial^{\alpha}f \in L^1_{\rm loc}$
the \emph{$\alpha${\rm th} weak
partial derivative} of $f$, that is,
for any $\phi \in C_{\rm c}^{\infty}$,
\begin{align*}
\int_{{\mathbb{R}^n} }f(x) \partial^{\alpha}\phi(x)\,dx
=
(-1)^{|\alpha|}
\int_{{\mathbb{R}^n} }\partial^{\alpha}f(x) \phi(x)\,dx.
\end{align*}
In what follows,
for any $p\in [1,\infty)$
and $q\in (0,\infty)$, let
\begin{align}\label{eq-GAMMA}
\Gamma_{p,q}:=
\begin{cases}
(-\infty,-q)\cup (0,\infty)& \text{if } p=1,\\
\mathbb{R}\setminus \{0\}& \text{if } p\in (1,\infty).
\end{cases}
\end{align}
For any $\lambda\in (0,\infty)$, $b\in \mathbb{R}$, $k\in \mathbb{N}$,
and $f\in \mathscr{M}$, let
\begin{align}\label{e-EE}
E_{\lambda,b,k}[f]
:=\left\{(x,h)\in {\mathbb{R}^n} \times
\left[{\mathbb{R}^n}\setminus\{{\bf 0}\}\right]:\
\left|\Delta_{h}^{k}f (x)\right|>\lambda |h|^{k+b}\right\}.
\end{align}
The main result of this article
is the following BSVY formula with
higher-order derivatives in ball Banach function spaces,
where all the symbols are precisely defined in
Subsection \ref{ssub-pre}.

\begin{theorem}\label{thm-main}
Let $X$ be a ball Banach function space and $k\in\mathbb{N}$.
\begin{enumerate}[{\rm (I)}]
\item Assume that there exists some $p\in [1,\infty)$
such that $X^{\frac{1}{p}}$ is a ball Banach function space
and the Hardy--Littlewood maximal operator
$\mathcal{M}$ is bounded on $(X^{\frac{1}{p}})'$.
Let $q\in(0,\infty)$ satisfy
$n(\frac1p-\frac1q)<k$ and $\gamma\in\Gamma_{p,q}$.
Then the following two statements hold:
\begin{enumerate}[{\rm (i)}]
\item for any
$f\in \dot{W}^{k,X}$,
\begin{align}\label{eq-main-01}
	\sup_{\lambda\in (0,\infty)}\lambda\left\|\left[
	\int_{{\mathbb{R}^n}}{\bf
		1}_{E_{\lambda,\frac{\gamma}{q},k}[f]}(\cdot,h)
	|h|^{\gamma-n}\, dh\right]^{\frac{1}{q}}\right\|_X
	\sim
	\left\|\,\left|\nabla^k f\right|\,\right\|_{X},
\end{align}
where the positive equivalence constants are independent of $f$ and,
for any $\lambda\in (0,\infty)$,
$E_{\lambda,\frac{\gamma}{q},k}[f]$ is the same as in
\eqref{e-EE} with $b:= \frac{\gamma}{q}$;
here and thereafter,
$\nabla^k f:=\{\partial^{\alpha}f:\ \alpha \in \mathbb{Z}_{+}^n ,\
|\alpha|=k\}$
and
\begin{align}\label{eq-Nk}
	\left|\nabla^k f\right|
	:= \left[\sum_{\alpha\in \mathbb{Z}_{+}^n,\,
		|\alpha|=k}\left|\partial^{\alpha}
	f\right|^2\right]^{\frac{1}{2}},
\end{align}
where $\partial^{\alpha}f$ is the $\alpha${\rm th} weak
partial derivative of $f$;

\item assume further that $X$ has an absolutely continuous norm,
then, for any $f\in \dot{W}^{k,X}$,
\begin{align}\label{eq-main-02}
	&\lim_{\lambda\to L_\gamma}\lambda\left\|\left[
	\int_{{\mathbb{R}^n}}{\bf
		1}_{E_{\lambda,\frac{\gamma}{q},k}[f]}(\cdot,h)
	|h|^{\gamma-n}\, dh\right]^{\frac{1}{q}}\right\|_X\notag\\
	&\quad=
	|\gamma|^{-\frac{1}{q}}\left\|\left[ \int_{\mathbb{S}^{n-1}}
	\left|\sum_{\alpha\in\mathbb{Z}_{+}^n,\,|\alpha|=k}
	\partial^{\alpha}f(\cdot)\xi^{\alpha}\right|^q\,d\mathcal{H}^{n-1}(\xi)
	\right]^{\frac{1}{q}}\right\|_X
	\sim
	\left\|\,\left|\nabla^k f\right|\,
	\right\|_X,
\end{align}
where the positive equivalence constants depend only on $n$, $k$,
and $q$;
here and thereafter,
\begin{align}\label{e-Lgamma}
	\lambda \to L_\gamma
\end{align}
means that $\lambda \to \infty$ whenever $\gamma\in (0,\infty)$
or that $\lambda\to 0^{+}$ whenever $\gamma\in (-\infty,0)$.
\end{enumerate}

\item Assume that the Hardy--Littlewood maximal
operator $\mathcal{M}$
is endpoint bounded on $X'$, the centered
ball average operators are uniformly bounded
on $X$, and $X$ has an absolutely continuous norm.
Let $q\in(0,\infty)$ satisfy
$n(1-\frac1q)<k$ and $\gamma\in\Gamma_{1,q}$.
Then both \eqref{eq-main-01} and \eqref{eq-main-02} hold
for any
$f\in \dot{W}^{k,X}$.
\end{enumerate}
\end{theorem}

\begin{remark}\label{re42}
\begin{enumerate}[{\rm (i)}]
\item
It is clear that $L^p$ with $p\in [1,\infty)$ satisfies
all the assumptions of
Theorem \ref{thm-main}. In the case where $X:=L^p$,
Theorem \ref{thm-main}(i) with $k=1$, $p\in [1,\infty)$, and $q=p$
coincides with the BSVY formulae established in \cite[Theorems 1.3 and
1.4]{bsvy24},
Theorem \ref{thm-main}(ii) with
$k=1$, $p\in [1,\infty)$, and $q=p$ coincides with
the limiting formula
of Brezis et al. \cite[Theorem 1.1]{bsvy24}, and
Theorem \ref{thm-main}
with $k\in\mathbb{N}\cap[2,\infty)$ and $p\in [1,\infty)$
is completely new.

\item	When $k=1$, Theorem \ref{thm-main}(i) coincides with
\cite[Theorem 1.4]{lyyzz24}
and Theorem \ref{thm-main}(ii) improves \cite[Theorem 4.3]{lyyzz24}
via removing some extra assumptions
in \cite[Theorem 4.3]{lyyzz24}.
For instance, let $k=1$, $n\in \mathbb{N}\cap [2,\infty)$,
$X:=L^{p}$ with $p\in [n,\infty)$, and $\gamma\in
(-\infty,0)$.
In this case, Theorem \ref{thm-main}(ii)
widens the range of $q\in (0,\frac{n-\gamma}{n}p)$
in \cite[Theorem 4.3]{lyyzz24}
into $q\in (0,\infty)$.
Obviously, Theorem \ref{thm-main}(ii) even when $k=1$
is better than \cite[Theorem 4.3]{lyyzz24}.
Moreover,
when $k\in\mathbb{N}\cap[2,\infty)$,
Theorem \ref{thm-main}
is completely new.

\item
By Proposition \ref{pro-sharp}, we find that the assumption
$n(\frac{1}{p}-\frac{1}{q})<k$ in
Theorem \ref{thm-main} is \emph{sharp}.
As pointed out in \cite[Remark 4.4(iv)]{lyyzz24}, the assumption
$\gamma\in\Gamma_{p,q}$
in Theorem \ref{thm-main} is also \emph{sharp}.
\end{enumerate}
\end{remark}

Applying Theorem \ref{thm-main}, we obtain the
following BSVY
characterization of $\dot{W}^{k,X}$.

\begin{theorem}\label{t-chaWkX}
Let $X$ be a ball Banach function space
satisfying that there exists
$p\in(1,\infty)$
such that $X^{\frac1p}$ is also a ball Banach function space
and the Hardy--Littlewood maximal operator $\mathcal{M}$
is bounded on $(X^{\frac1p})'$,
$k\in \mathbb{N}$, $q\in (1,\infty)$ satisfy
$n(\frac{1}{p}-\frac{1}{q})<k$,
and $\gamma\in \mathbb{R}\setminus\{0\}$.
Assume that both $X$ and $X'$ have absolutely continuous
norms.
Then $f\in \dot{W}^{k,X}$ if and only if
$f\in L^{1}_{\rm loc}$ and
\begin{align}\label{e-c-wkx}
\sup_{\lambda\in (0,\infty)}\lambda\left\|\left[
\int_{{\mathbb{R}^n}}{\bf
1}_{E_{\lambda,\frac{\gamma}{q},k}[f]}(\cdot,h)
|h|^{\gamma-n}\, dh\right]^{\frac{1}{q}}\right\|_X<\infty,
\end{align}
where, for any $\lambda\in (0,\infty)$,
$E_{\lambda,\frac{\gamma}{q},k}[f]$ is the same as in
\eqref{e-EE} with $b:= \frac{\gamma}{q}$;
moreover,
\eqref{eq-main-01} holds for any $f\in L^{1}_{\rm loc}$.
\end{theorem}

\begin{remark}
When $X:=L^p$,
Theorem \ref{t-chaWkX} with $k=1$, $p\in (1,\infty)$,
and $q=p$ coincides with \cite[Theorem 1.3]{bsvy24},
that is,
the BSVY characterization of
the Sobolev space $\dot{W}^{1,p}$.
Other cases of Theorem \ref{t-chaWkX} are
new.
\end{remark}

As applications of Theorem \ref{thm-main}, we establish higher-order
fractional Gagliardo--Nirenberg and Sobolev type inequalities.
In what follows, for any $f\in \mathscr{M}$, let
$\nabla^0 f:= f$.

\begin{theorem}\label{thm-fsGN}
Let $X$ be a ball Banach function space.
Assume that there exists some $p\in [1,\infty)$
such that $X^{\frac{1}{p}}$ is a ball Banach function space
and the Hardy--Littlewood maximal operator
$\mathcal{M}$ is bounded on $(X^{\frac{1}{p}})'$, or that $p=1$
and all the assumptions of Theorem \ref{thm-main}(II) hold.
Let $k\in {\mathbb{N}}$ and $\gamma\in \Gamma_{p,1}$.
\begin{enumerate}[{\rm (i)}]
\item
Let $s\in (0,1)$,
${q_{0}} \in [1,\infty]$,
and
$q\in [1,{q_{0}}]$
satisfy $\frac{1}{q}=\frac{1-s}{{q_{0}}}+s$.
If ${q_{0}}\in [1,\infty)$ and $X^{q_0}$
is a ball Banach function space,
then,
for any
$f\in \dot{W}^{k,X}$,
\begin{align}\label{eq-in-kks}
\sup_{\lambda\in (0,\infty)}\lambda\left\|\left[
\int_{{\mathbb{R}^n}}{\bf
1}_{E_{\lambda,\frac{\gamma}{q}+s-1,k}[f]}(\cdot,h)
|h|^{\gamma-n}\, dh\right]^{\frac{1}{q}}\right\|_{X^q}
\lesssim
\left\|\,\left|\nabla^{k-1} f\right|\,\right\|^{1-s}_{X^{q_{0}}}
\left\|\,\left|\nabla^k f\right|\,\right\|^s_X;
\end{align}
if ${q_{0}}=\infty$, then,
for any
$f\in \dot{W}^{k,X}$,
\begin{align}\label{eq-in-kks1}
\sup_{\lambda\in (0,\infty)}\lambda\left\|\left[
\int_{{\mathbb{R}^n}}{\bf
1}_{E_{\lambda,\frac{\gamma}{q}+s-1,k}[f]}(\cdot,h)
|h|^{\gamma-n}\, dh\right]^{\frac{1}{q}}\right\|_{X^q}
\lesssim
\left\|\,\left|\nabla^{k-1}f\right|\,\right\|^{1-s}_{L^\infty}
\left\|\,\left|\nabla^k f\right|\,\right\|^s_X,
\end{align}
where the implicit positive constants are independent of $f$ and,
for any $\lambda\in (0,\infty)$,
the set
$E_{\lambda,\frac{\gamma}{q}+s-1,k}[f]$ is the same as in
\eqref{e-EE} with $b:= \frac{\gamma}{q}+s-1$.

\item Let $\eta\in (0,1)$
and
$0\leq s_0<s<1<q<q_0<\infty$ satisfy
\begin{align}\label{eq-app-ss}
s=(1-\eta)s_0+\eta
\ \ \text{and}\ \
\frac{1}{q}=\frac{1-\eta}{q_0}+\eta.
\end{align}
Then, for any $f\in\dot{W}^{k,X}$,
\begin{align}\label{eq-app-dd}
&\sup_{\lambda\in(0,\infty)}\lambda
\left\|\int_{\mathbb{R}^n}
\mathbf{1}_{E_{\lambda,\frac{\gamma}{q}+s-1,k}[f]}(\cdot,h)
\left|h\right|^{\gamma-n}\,dh\right\|_{X}^{\frac{1}{q}}\notag\\
&\quad
\lesssim
\sup_{\lambda\in(0,\infty)}\lambda
\left\|\int_{\mathbb{R}^n}
\mathbf{1}_{E_{\lambda,\frac{\gamma}{q_0}+s_0-1,k}[f]}(\cdot,h)
\left|h\right|^{\gamma-n}\,dh\right\|_{X}^{\frac{1-\eta}{q_0}}
\left\|\,\left|\nabla^k f\right|\,\right\|_X^\eta,
\end{align}
where the implicit positive constant is independent of $f$ and,
for any $\lambda\in (0,\infty)$,
the set $E_{\lambda,\frac{\gamma}{q_0}+s_0 -1,k}[f]$ is the same as in
\eqref{e-EE} with $b:= \frac{\gamma}{q_0}+s_0 -1$.
\end{enumerate}
\end{theorem}

\begin{remark}\label{rem-fix}
\begin{enumerate}[{\rm (i)}]
\item Theorem \ref{thm-fsGN} when $k=1$ coincides with
\cite[Corollaries 4.5 and 4.6]{lyyzz24}
and, in other cases, Theorem \ref{thm-fsGN} is new.

\item
Let $k\in \mathbb{N}$, $q\in (1,\infty)$,
and $s\in (0,1-\frac{1}{q})$. We recall that the higher-order
Gagliardo--Nirenberg type inequality is that
there exists a positive constant $C$ depending only on $k,q$,
and $s$ such that,
for any $f\in C_{\rm c}^{\infty}$,
\begin{align*}
\left\|
\frac{\Delta^k_h f(x)}{|h|^{k-1+\frac{1}{q}+\frac{n}{q}+s}}
\right\|_{L^{q}(\mathbb{R}^n\times\mathbb{R}^n)}
&\le C
\left[\left\|\,\left|\nabla^{k-1}f\right|\,\right\|
_{L^\infty}+\sup_{|h|\neq 0}\frac{|\,|\nabla^k
f(x+h)|-|\nabla^k f(x)|\,|}{|h|^{s}}\right]^{1-{\frac{1}{q}}}\\
&\quad\times
\left[\left\|\,\left|\nabla^k f\right|\,\right\|_{L^1}
+\sum_{\alpha\in\mathbb{Z}_+^n,\,|\alpha|=k}\left\|\partial^\alpha f\right\|_{\dot{W}^{s,1}}\right]^{\frac{1}{q}},
\end{align*}
where $\|\cdot\|_{\dot{W}^{s,1}}$ is the same as in
\eqref{eq-GSN} (see \cite[Theorem 1 (A)]{bm2018}).
In the critical case $s=0$,
applying \cite[Theorem 1 (B)]{bm2018},
we conclude that there
does \emph{not} exist a positive constant $C$ such that,
for any $f\in C_{\rm c}^{\infty}$,
\begin{align}\label{rem-fixe1}
\left\|
\frac{\Delta^k_h f(x)}{|h|^{k-1+\frac{1}{q}+\frac{n}{q}}}
\right\|_{L^{q}(\mathbb{R}^n\times\mathbb{R}^n)}
\le C
\left\|\,\left|\nabla^{k-1}f\right|\,\right\|^{1-{\frac{1}{q}}}_{L^\infty}
\left\|\,\left|\nabla^k f\right|\,\right\|^{\frac{1}{q}}_{L^1}.
\end{align}
Fortunately,
by \eqref{eq-in-kks1} with $X:=L^{1}$, $\gamma:=n$, and
$s:=\frac{1}{q}$,
we find that \eqref{rem-fixe1} holds
if we replace $\|\cdot\|_{L^q(\mathbb{R}^n\times\mathbb{R}^n)}$
by
$\|\cdot\|_{L^{q,\infty}(\mathbb{R}^n\times\mathbb{R}^n)}$.
This fixes the above ``defect'' of the Gagliardo--Nirenberg type
inequality in this critical case $\sigma=0$.
\end{enumerate}
\end{remark}

All these results are of quite wide applications and can be applied
to various specific function spaces,
including (mixed-norm, or variable, or weighted) Lebesgue,
(Bourgain--)Morrey-type,
Lorentz, and Orlicz (or Orlicz-slice) spaces (see Section
\ref{ssec-fract-s}).

We emphasize that the core difficulty in proving the BSVY formula
\cite{bsvy24} lies in the critical case $p=1$ for the upper
estimate. In that proof, leveraging the rotation
invariance of $L^p$, the problem can be
reduced to the one-dimensional case where the proof
employs the Vitali covering theorem and a stopping time
argument (see \cite[Proposition 2.3]{bsvy24} for more details).
However, since these methods crucially rely on the
single-direction property of $\mathbb{R}$
and the rotation invariance of $L^1$,
the approach used in \cite{bsvy24} cannot be extended to ball
Banach function spaces in higher dimensions where
the function spaces under consideration lack rotation invariance.

To overcome this difficulty,
we first establish a higher-order weighted variant
of Cohen et al. \cite[Theorems 3.1 and 4.1]{cddd03}
with respect to the higher-order local
approximation (see Theorem
\ref{esti-omega}),
which is of independent interest and may be useful
in tackling other related problems.
Recall that, for $k=1$,
Cohen et al. \cite[Theorems 3.1 and 4.1]{cddd03} applied
the coarea formula, the isoperimetric inequality, and
a nuanced classification of the dyadic
cubes of $\mathbb{R}^n$ to establish a quite remarkable
inequality related to both level sets of renormalized
averaged moduli of continuity and gradients,
which have proved very useful in
\cite{BN11,bm2018,dt23,DHT20,dssvy,lyyzz24,zyyy24}.
In this article, for $k\in \mathbb{N}\cap [2,\infty)$,
we utilize the well-known Whitney inequality in approximation
to obtain a delicate relation between the higher-order local
approximation and the first-order
renormalized averaged modulus of continuity.
This approach allows us to reduce the desired higher-order estimate
to the first-order one and then establish the
the desired inequality.
Subsequently, by first establishing a sparse
characterization about dyadic cubes
in the level set of the higher-order local approximation
(Lemma \ref{lem-Qx}) and
a variant higher-order Poincar\'e inequality
(Lemma \ref{Poin}) and
employing an argument originating from
the pigeonhole principle,
we show that the higher-order difference has a pointwise
upper bound determined by
the higher-order local approximation [see \eqref{pointe2},
\eqref{eq-claim-02}, and \eqref{eq-dff}].
Finally, we apply the above higher-order variant
of Cohen et al. (Theorem \ref{esti-omega})
to derive
the higher-order upper estimate of the BSVY formula
in weighted Lebesgue spaces (see Theorem \ref{thm-up-ap})
and then utilize a standard method
of extrapolation to obtain the desired upper estimate
in \eqref{eq-main-01}. Moreover, we show the above two weighted
estimates are
\emph{sharp}
and use them to characterize Muckenhoupt weights when
$n=1$ (see Corollary \ref{cor-ap}).
However, extending Corollary \ref{cor-ap} to the case
$n\in\mathbb{N}\cap[2,\infty)$
is still unknown (see Remark \ref{rem-ques}).

On the other hand,
recall that the proof of the BSVY characterization of
first-order Sobolev spaces
in \cite[Theorem 1.3]{bsvy24} essentially depends on another
famous formula established by
Bourgain et al. \cite{BBM01},
which is called the BBM formula now and
is still
unknown in higher-order ball Banach Sobolev spaces.
To overcome this deficiency, we
borrow some ideas from the proof of
Frank \cite[Theorem 1.1]{f22}.
To be precise, we
first establish
a Young inequality of convolutions in ball Banach
function spaces (see Proposition \ref{p-young})
via a Marcinkiewicz-type interpolation theorem
for mixed-norm Lebesgue spaces (see Lemma \ref{lem-sw-ww})
and then borrow some ideas from the
proof of \cite[Theorem 1.1]{f22}
on the characterization of Sobolev
spaces, whose advantage lies in that it depends only on
the Banach--Alaoglu theorem but independent of the BBM formula.

The organization of the remainder of this article is as follows.

In Section \ref{sec-hohbbs}, we give some necessary
preliminaries and prove the density properties
of the higher-order homogeneous ball Banach Sobolev space.
Precisely, in Subsection \ref{ssub-pre},
we recall the ball (quasi-)Banach function
space $X$ (see Definition \ref{def-X}) and
the Muckenhoupt weight class $A_p$.
We then introduce the higher-order homogeneous ball Banach Sobolev
space
$\dot{W}^{k,X}$ and give some preliminary properties on $X$,
which are frequently used throughout this article.
In Subsection \ref{ssub-density}, by the Poincar\'{e} inequality
on ball Banach function spaces (see Lemma \ref{lem-poin}), we prove
the density properties of ``good" functions in
 $\dot{W}^{k,X}$ (see Theorem \ref{thm-e}).

The main purpose of Section \ref{sec-AP-es}
is to obtain a higher-order weighed extension of
the inequality \cite[Theorems 3.1 and 4.1]{cddd03}
of Cohen et al.\ as well as the
upper estimate of the higher-order BSVY formula in
weighted Lebesgue spaces (see Theorems \ref{esti-omega} and
\ref{thm-up-ap}).
Furthermore, we prove the above two weighted estimates are
sharp and use them to characterize Muckenhoupt weights when
$n=1$ (see Corollary \ref{cor-ap}).

In Section \ref{sec-Formulae},
via using the density properties given in Section
\ref{sec-hohbbs}, the weighed upper estimate given
in Section \ref{sec-AP-es}, and the method of extrapolation,
we give the proof of Theorem \ref{thm-main}.
Then the main target of Subsection \ref{ss-cha}
is to show the BSVY characterization of
higher-order ball Banach Sobolev spaces (see Theorem \ref{t-chaWkX}).
In Subsection \ref{ssec-fract}, as an application of the
higher-order BSVY formula, we
obtain higher-order fractional Gagliardo--Nirenberg and Sobolev
type
inequalities (see Theorem \ref{thm-fsGN}).

Section \ref{ssec-fract-s} is devoted to applying
the above main theorems
to various specific examples of ball Banach function spaces, including
Lebesgue, weighted Lebesgue, (Bourgain--)Morrey-type, mixed-norm
Lebesgue,
variable Lebesgue, Lorentz, Orlicz, and Orlicz-slice spaces.

Finally, we make some conventions on notation.
We always let
$\mathbb{Z}_+:=\mathbb{N}\cup\{0\}$.
If $E$ is a subset of ${\mathbb{R}^n}$, we denote by
${\bf 1}_E$ its \emph{characteristic function},
by $E^{\complement}$ the set ${\mathbb{R}^n} \setminus E$,
by $|E|$ its \emph{Lebesgue measure},
and by $\mathcal{H}^{n-1}(E)$ its $(n-1)$-dimensional \emph{Hausdorff
measure}.
For any set $E,F\subset \mathbb{R}^n$, let
\begin{align*}
E-F:=\left\{x-y:\ x\in E\ \text{ and}\ y\in F\right\}.
\end{align*}
Moreover, we use ${\bf 0}$ to denote the origin of ${\mathbb{R}^n}$
and $\mathbb{S}^{n-1}$ the unit sphere of ${\mathbb{R}^n}$.
The symbol $\mathcal{Q}$ denotes the set
of all cubes with edges parallel to the coordinate axes.
For any $x\in {\mathbb{R}^n}$ and $r\in (0,\infty)$, let
$B(x,r):=\{y\in {\mathbb{R}^n}:\ |x-y|<r\}$ and
$B_r:= B({\bf 0},r)$.
For any $\lambda\in (0,\infty)$ and any ball $B:=B(x_B, r_B)$
in ${\mathbb{R}^n} $ with both center $x_B \in {\mathbb{R}^n}$
and the radius $r_B\in (0,\infty)$,
let $\lambda B:= B(x_B, \lambda r_B)$;
for any cube $Q\in \mathcal{Q}$, $\lambda Q$ means a cube
with the same center as $Q$ and $\lambda$ times the edge length of
$Q$.
Throughout this article, let
$\mathcal{P}_{s}$ denote the set of
all polynomials of degree not greater than $s\in \mathbb{Z}_{+}$ on
$\mathbb{R}^n$.
For any $f\in \mathscr{M}$, its support
$\mathrm{supp\,} (f)$
is defined by setting
$\mathrm{supp\,} (f):=
\{x\in{\mathbb{R}^n}:\ f(x)\neq 0\}$.
Let
$C_{\rm c}$ denote
the set of all
continuous functions with compact support,
and denote $C^{\infty}$
[resp. $C_{\rm c}^{\infty}$] the set of all infinitely
differentiable
functions on $\mathbb{R}^n$ (resp. with compact support).
For any $p\in (0,\infty]$ and any measurable
subset $\Omega\subset\mathbb{R}^n$, we use
$L^p_{\rm loc}(\Omega)$ to denote the set of all locally
$p$-integrable
functions on $\Omega$.
For any $f\in L^{1}_{\rm loc}$ and $E\subset
{\mathbb{R}^n}$
with $|E|<\infty$,
let
\begin{align*}
f(E):=\int_{E}f(x)\, dx
\ \ \text{and}\ \
f_E:=\fint_Ef(x)\,dx
:=\frac{1}{|E|}\int_Ef(x)\,dx.
\end{align*}
For any $\alpha=(\alpha_1,\ldots,\alpha_n),\beta
=(\beta_1,\ldots,\beta_n)\in \mathbb{Z}_{+}^n$, $\alpha\le\beta$
means that, for any $i\in \mathbb{N}\cap [1,n]$, $\alpha_i \le
\beta_i$.
Moreover, when $n=1$, for any $k\in \mathbb{N}$,
we denote simply by $f^{(k)}$ the
$k${\rm th} order derivative of $f$ on $\mathbb{R}$.
For any index $q\in [1,\infty]$, we denote by $q'$ its
\emph{conjugate index}, that is, $\frac{1}{q}+\frac{1}{q'}=1$.
We denote by $C$ a \emph{positive constant} which is independent of
the main parameters involved, but may vary from line to line.
We use $C_{(\alpha,\dots)}$
to denote a positive constant depending on the indicated parameters
$\alpha,\dots$.
The symbol $f\lesssim g$ means $f\leq C g$ and, if
$f\lesssim g\lesssim f$, then we write $f\sim g$.
If $f\leq C g$ and $g=h$ or $g\leq h$, we
then write $f\lesssim g=h$ or $f\lesssim g \leq h$.
In addition, $\varepsilon\to 0^+$ means that
there exists $c_0\in(0,\infty)$ such that
$\varepsilon\in(0,c_0)$ and $\varepsilon\to0$.
Finally, when we prove a theorem (and the like),
in its proof we always use the same symbols as in the statement itself
of that theorem (and the like).

\section{Higher-Order Homogeneous Ball Banach Sobolev
Spaces}\label{sec-hohbbs}

In this section, we give some preliminaries of this article
(Subsection \ref{ssub-pre}) and
the density properties of higher-order homogeneous ball Banach Sobolev
spaces (Subsection \ref{ssub-density}).

\subsection{Preliminaries}\label{ssub-pre}

We begin with the concept of ball quasi-Banach function
spaces introduced in \cite{shyy2017}.

\begin{definition}\label{def-X}
A quasi-Banach space $X\subset \mathscr{M}$, equipped
with a quasi-norm
$\|\cdot\|_X$ which makes sense for all functions in
$\mathscr{M}$,
is called a \emph{ball quasi-Banach function space}
(for short, {\rm BQBF} space)
if $X$ satisfies that
\begin{enumerate}[{\rm (i)}]
\item for any $f\in \mathscr{M}$, if $\|f\|_X =0$,
then $f=0$ almost
everywhere;
\item if $f,g\in\mathscr{M}$ satisfy that $|g|\leq
|f|$
almost everywhere, then $\|g\|_X \leq \|f\|_X$;
\item if a sequence
$\{f_m\}_{m\in{\mathbb{N}}}$ in $\mathscr{M}$
satisfies that
$0\leq f_m \uparrow f$ almost everywhere as $m\to\infty$,
then $\|f_m\|_X \uparrow \|f\|_X$ as $m\to\infty$;
\item for any ball $B:=B(x,r)$ with both $x\in{\mathbb{R}^n}$
and $r\in (0,\infty)$, ${\bf 1}_{B}\in X$.
\end{enumerate}
Moreover, a {\rm BQBF} space $X$ is called a \emph{ball Banach function
space} (for short, {\rm BBF} space) if $X$ satisfies the following
extra conditions:
\begin{enumerate}
\item[${\rm (v)}$] for any $f,g\in X$,
$\|f+g\|_X\leq \|f\|_X +\|g\|_X;$
\item[${\rm (vi)}$]
for any ball $B$, there exists a positive constant $C_{(B)}$,
depending on $B$,
such that, for any $f\in X$,
\begin{align*}
\int_{B}|f(x)|\,dx \leq C_{(B)}\|f\|_{X}.
\end{align*}
\end{enumerate}
\end{definition}

\begin{remark}\label{rm-bqbf}
\begin{enumerate}
\item[{\rm(i)}] Observe that,
in Definition \ref{def-X}(iv),
if we replace any ball $B$
by any bounded measurable set
$E$, we obtain an equivalent formulation
of {\rm BQBF} spaces.
\item[\textup{(ii)}]
Let $X$ be a {\rm BQBF} space.
Then, by the definition, we can easily conclude
that, for any $f\in\mathscr{M}$,
$\|f\|_X=0$ if and only if $f=0$ almost everywhere
(see also \cite[Proposition 1.2.16]{lyh2320}).
\item[{\rm(iii)}]Applying both (ii) and (iii)
of Definition \ref{def-X}, we find that any {\rm BQBF} space $X$
has the Fatou property, that is,
for any $\{f_k\}_{k\in\mathbb{N}}\subset X$,
\begin{align*}
\left\|\liminf_{k\to\infty}\left|f_k\right|\right\|_X
\leq\liminf_{k\to\infty}\left\|f_k\right\|_X
\end{align*}
(see also \cite[Lemma 2.4]{wyy23}).
\item[{\rm(iv)}] From \cite[Proposition 1.2.36]{lyh2320}
(see also \cite[Theorem 2]{dfmn2021}),
we infer that every {\rm BQBF} space
is complete.
\item[{\rm(v)}] Recall that a quasi-Banach space
$X\subset\mathscr{M}$
is called a \emph{quasi-Banach function space}
if it is a {\rm BQBF} space
and it satisfies Definition \ref{def-X}(iv)
with ball therein replaced by any
measurable set of \emph{finite measure}.
Moreover, a \emph{Banach function space}
is a quasi-Banach function space satisfying
(v) and (vi) of Definition \ref{def-X}
with ball therein replaced by any measurable set
of \emph{finite measure}, which was originally
introduced in
\cite[Chapter 1, Definitions 1.1 and 1.3]{bs1988}.
It is easy to show that every quasi-Banach
function space (resp. Banach function space)
is a ball quasi-Banach
function space (resp. ball Banach function space),
and the converse is not necessary to be true.
Several examples about ball (quasi-)Banach function spaces
are given in Section \ref{ssec-fract-s} below.

\item[${\rm (vi)}$]
In Definition \ref{def-X},
if we replace (iv)
by the following \emph{saturation property}:
\begin{enumerate}
\item[\rm(a)]
for any measurable set $E\subset\mathbb{R}^n$
of positive measure, there exists a measurable set $F\subset E$
of positive measure satisfying that $\mathbf{1}_F\in X$,
\end{enumerate}
then we obtain the definition of quasi-Banach function spaces
in Lorist and Nieraeth
\cite{ln23}.
Moreover, by \cite[Proposition 2.5]{zyy2023bbm0}
(see also \cite[Proposition 4.22]{n23}),
we find that, if the
quasi-normed vector space $X$ satisfies
the extra assumption that
the Hardy--Littlewood maximal operator is weakly bounded on
$X$ or its convexification,
then the definition of quasi-Banach function spaces in \cite{ln23}
coincides with
the definition of ball quasi-Banach function spaces. Thus,
under this extra assumption,
working with ball quasi-Banach function
spaces in the sense of Definition \ref{def-X}
or quasi-Banach function spaces in
the sense of \cite{ln23} would yield
exactly the same results.
\end{enumerate}
\end{remark}

The following definition of the $p$-convexification of
a {\rm BQBF} space can be
found in \cite[Definition 2.6]{shyy2017}.

\begin{definition}
Let $X$ be a {\rm BQBF} space and $p\in (0,\infty)$. The
\emph{ $p$-convexification}
$X^p$ of $X$ is defined by setting
$X^p:=\{f\in \mathscr{M}:\ |f|^p \in X\}$
and is equipped with the quasi-norm
$\left\|f\right\|_{X^p}:=\left\|\,|f|^p\right\|_{X}^{1/p}$
for any $f\in X^p$.
\end{definition}

The following concept of the associate space of a
{\rm BBF} space can be found in
\cite[p.\,9]{shyy2017}; see \cite[Chapter 1, Section 2]{bs1988}
for more details.

\begin{definition}\label{def-X'}
For any {\rm BBF} space $X$,
the \emph{associate space} (also called the
\emph{K\"{o}the dual}) $X'$ is defined by setting
\begin{align*}
X':=\left\{f\in \mathscr{M}:\ \|f\|_{X'}:=
\sup_{\{g\in X:\ \|g\|_X =1\}}\|fg\|_{L^1
}<\infty\right\},
\end{align*}
where $\|\cdot\|_{X'}$ is called the \emph{associate norm}
of $\|\cdot\|_{X}$.
\end{definition}

\begin{remark}\label{rem-x-a}
Let $X$ be a {\rm BBF} space.
\begin{enumerate}[{\rm (i)}]
\item From \cite[Proposition 2.3]{shyy2017}, we infer that the
associate space $X'$
is also a {\rm BBF} space.

\item Using \cite[Theorem 2.4]{bs1988}, we find that, if $f\in X$
and $g\in X'$, then $fg$ is integrable and
\begin{align*}
\int_{{\mathbb{R}^n}}\left|f(x)g(x)\right|\,dx\leq
\|f\|_{X}\|g\|_{X'}.
\end{align*}

\item From \cite[Lemma 2.6]{zwyy2021}, it follows that
$X$
coincides with its second associate space $X''$.
In other words, a function $f\in X$ if and only if
$f\in X''$ and, in that case,
$\|f\|_X = \|f\|_{X''}.$
\end{enumerate}
\end{remark}

In what follows, for any $f\in L^{1}_{\rm loc}$, its
\emph{Hardy--Littlewood maximal function}
$\mathcal{M}(f)$ is defined by setting, for any $x\in {\mathbb{R}^n}$,
\begin{align*}
\mathcal{M}(f)(x):=\sup_{B\ni x}\fint_{B}|f(y)|\,dy,
\end{align*}
where the supremum is taken over all balls
$B\subset \mathbb{R}^n$ containing $x$. We denote by
$\|{\mathcal M}\|_{X \to Y}$
the norm of $\mathcal{M}$ from a {\rm BQBF} space $X$ to another {\rm BQBF} space
$Y$.

As pointed out in \cite[Remark 4.11(ii)]{dlyyz-bvy},
for some specific examples of $X$ such as $X: =
L^{r(\cdot)}$
with $\widetilde{r}_- = 1$ (see Subsection \ref{ssec-vari-lp}
for the precise definitions of both $L^{r(\cdot)}$
and $\widetilde{r}_-$), in the endpoint case $p=1$ of Theorem
\ref{thm-main}
it is still unknown whether or not $\mathcal{M}$ is bounded on $X'$.
Therefore, we need the following
definition in this article, which
was introduced in \cite[Definition
2.14]{pyyz2023}.

\begin{definition}\label{ass-endpoint}
Let $X$ be a {\rm BBF} space. The Hardy--Littlewood maximal
operator $\mathcal{M}$
is said to be \emph{endpoint bounded} on $X'$
if there exists a sequence
$\{\theta_m\}_{m\in\mathbb{N}}\subset(0,1)$
satisfying $\lim_{m\to\infty}\theta_m=1$ such that,
for any $m\in\mathbb{N}$,
$X^\frac{1}{\theta_m}$ is a {\rm BBF} space,
$\mathcal{M}$
is bounded on $(X^\frac{1}{\theta_m})'$, and
\begin{align*}
\lim_{m\to\infty}\left\|\mathcal{M}\right\|_{(X^\frac{1}{\theta_m})'
\to(X^\frac{1}{\theta_m})'}<\infty.
\end{align*}
\end{definition}

Recall that, for any given $r\in(0,\infty)$,
the \emph{centered ball average operator}
$\mathcal{B}_r$ is defined by setting,
for any $f\in L_{{\mathrm{loc}}}^1$ and
$x\in\mathbb{R}^n$,
\begin{align*}
\mathcal{B}_r(f)(x):=\frac{1}{|B(x,r)|}\int_{B(x,r)}\left|f(y)\right|\,dy.
\end{align*}

The following key lemma can be found in \cite[Lemma 3.11]{dgpyyz24}
and \cite[Lemma 3.7 and Remark 3.8]{zyy2023bbm0},
which is frequently used in this article.

\begin{lemma}\label{rem-weak}
If $X$ is a {\rm BBF} space such that there exists $p\in[1,\infty)$
satisfying that $X^{\frac1p}$ is also a {\rm BBF} space
and $\mathcal{M}$ is bounded on $(X^{\frac1p})'$,
then the centered ball average operators
$\{\mathcal{B}_r\}_{r\in (0,\infty)}$
are uniformly bounded on $X$, that is, there exists
a positive constant $C$, independent of $r$, such that, for any $f\in
X$,
$
\|\mathcal{B}_r (f)\|_X
\leq
C\| f\|_X .$
Furthermore, if $p\in(1,\infty)$, then $\mathcal{M}$
is bounded on $X$.
\end{lemma}

The following concept can be found in \cite[Definition 3.1]{bs1988}
or
\cite[Definition 3.2]{wyy2020}.

\begin{definition}\label{ass-abs-norm}
A BQBF space $X$ is said to have an \emph{absolutely continuous norm}
if, for any $f\in X$ and any sequence $\{E_{j}\}_{j\in{\mathbb{N}}}$
of measurable sets satisfying that ${\bf 1}_{E_j}\to 0$ almost
everywhere as $j\to \infty$,
then $\|f{\bf 1}_{E_j}\|_X \to 0$ as $j\to \infty$.
\end{definition}

\begin{remark}\label{r-dual}
Let $X$ be a {\rm BBF} space having an absolutely continuous norm.
Then, by \cite[Lemma 4.7]{zyyy24} and
\cite[Remark 2.8(iii)]{zyy2023bbm0},
we conclude that $X$ is separable and
$X'$ coincides with $X^*$. Here and thereafter,
we denote the \emph{dual space} of $X$ by $X^*$.
\end{remark}

We next recall the concepts of $A_p $-weights and
weighted function spaces as follows
(see, for instance, \cite{g2014}).

\begin{definition}\label{def-Ap}
\begin{enumerate}[{\rm (i)}]
\item Let $p\in[1,\infty)$.
An \emph{$A_p $-weight} $\upsilon$ is
a nonnegative locally integrable function
on ${\mathbb{R}^n}$ such that,
when $p\in(1,\infty)$,
\begin{align*}
[\upsilon]_{A_p }:=
\sup_{Q\subset {\mathbb{R}^n}}
\left[\frac{1}{|Q|}\int_{Q}\upsilon(x)\,dx\right]
\left\{\frac{1}{|Q|}\int_{Q}[\upsilon(x)]^{1-p'}\,dx\right\}^{p-1}
<
\infty
\end{align*}
and, when $p=1$,
\begin{align*}
[\upsilon]_{A_1}:=\sup_{Q\subset {\mathbb{R}^n}}
\left[\frac{1}{|Q|}
\int_{Q}\upsilon(x)\,dx\right]
\left\|\upsilon^{-1}\right\|_{L^\infty (Q)}
<
\infty,
\end{align*}
where
the suprema are
taken over all cubes $Q\subset {\mathbb{R}^n}$.

\item Let $p\in (0,\infty]$ and $\upsilon$ be a nonnegative locally
integrable function on ${\mathbb{R}^n}$.
The \emph{weighted Lebesgue space}
$L^{p}_{\upsilon}$
is defined to be the set of all
$f\in \mathscr{M}$
such that
\begin{align*}
\|f\|_{L^p_{\upsilon}}:=
\left[\int_{{\mathbb{R}^n}}|f(x)|^p \upsilon(x)\,
dx\right]^{\frac{1}{p}}
<
\infty.
\end{align*}
\end{enumerate}
\end{definition}

Now, we present the following definition
of ball Banach Sobolev
spaces, which was originally introduced in
\cite[Definition 2.4]{dlyyz-bvy} when $k=1$.

\begin{definition}\label{def-WKX}
Let $X$ be a {\rm BBF} space and $k\in {\mathbb{N}}$. The
\emph{$k${\rm th} order homogeneous ball Banach Sobolev space}
$\dot{W}^{k,X}$
is defined to be the set of all $f\in L^{1}_{\rm loc}$
such that $|\nabla^k f|\in X$, which is equipped
with the \emph{semi-norm}
\begin{align*}
\|f\|_{\dot{W}^{k,X}}:=\left\|\nabla^k f\right\|_{X}:=
\left\|\,\left|\nabla^k f\right|\,\right\|_{X},
\end{align*}
where $|\nabla^k f|$ is the same as in \eqref{eq-Nk}.
Moreover, if $X:=L^p_{\upsilon}$
with $p\in [1,\infty]$ and $\upsilon\in L^1_{\rm loc}$
being nonnegative, then we denote $\dot{W}^{k,X}$
simply by $\dot{W}^{k,p}_{\upsilon}$,
which is called the \emph{$k${\rm th} order homogeneous weighted
Sobolev space}.
\end{definition}

For any given $k\in {\mathbb{N}}$ and any given open set $U\subset
{\mathbb{R}^n}$, the
\emph{Sobolev space} $W^{k,1}(U)$ is defined to be the set of
all integrable functions $f$ on $U$ such that
\begin{align*}
\|f\|_{W^{k,1}(U)}:=\|f\|_{L^1(U)}+
\left\|\,\left|\nabla^k f\right|\,\right\|_{L^1(U)}<\infty.
\end{align*}
Furthermore, we denote by $W^{k,1}_{\rm loc}$
the set of all $f\in L^{1}_{\rm loc}$ such that, for any
bounded open set $U\subset {\mathbb{R}^n}$,
$f\in W^{k,1}(U)$ .

\begin{proposition}\label{pro-wklo}
Let $k\in {\mathbb{N}}$ and $X$ be a {\rm BBF} space. Then
$\dot{W}^{k,X}\subset W^{k,1}_{\rm
loc}$.
\end{proposition}

\begin{proof}
Let $f\in \dot{W}^{k,X}$.
Then $|\nabla^k f|\in X$. By
this and
Definition \ref{def-X}(vi),
we find that $|\nabla^k f|\in L^1_{\rm loc}$.
This, combined with \cite[Sect. 1.1.2, Theorem]{ma2011},
further implies that,
for any $\alpha\in \mathbb{Z}_{+}^n$ with
$0\leq |\alpha|\leq k-1 $,
$\partial^{\alpha}f\in L^1_{\rm loc}$.
Thus, $f\in W^{k,1}_{\rm loc}$.
This finishes the proof of Proposition \ref{pro-wklo}.
\end{proof}

For Muckenhoupt $A_p $-weights, we have the
following basic properties,
which are frequently used in this article; see, for instance,
\cite[Lemma 3.15]{dgpyyz24}, \cite[Section 7.1]{shyy2017}, \cite[(7.3)
and (7.5)]{D2001},
\cite[Proposition 7.1.5 and Theorem 7.1.9]{g2014}, and \cite[Theorem
2.7.4]{dhhr2011}.

\begin{lemma}\label{lem-apwight}
Let $p\in [1,\infty)$ and $\upsilon\in A_p $. Then the
following
statements hold:
\begin{enumerate}[{\rm (i)}]
\item $L^p_{\upsilon}$
is a {\rm BBF} space having an absolutely continuous norm
and the centered ball average operators
are uniformly bounded on $L^p_{\upsilon}$.

\item For any cubes $Q,S\subset {\mathbb{R}^n}$ with $Q\subset S$,
$
\upsilon(S)\leq [\upsilon]_{A_p }
({|S|}/{|Q|})^p
\upsilon(Q).
$

\item \begin{align*}
[\upsilon]_{A_p}=\sup_{Q\subset {\mathbb{R}^n} }
\sup_{\|f{\bf 1}_{Q}\|_{L^p_{\upsilon}}
\in (0,\infty)}
\frac{[\frac{1}{Q}\int_{Q}|f(x)|\,dx]^p}{\frac{1}{\upsilon(Q)}\int_{Q}|f(x)|^p\upsilon(x)\,dx},
\end{align*}
where the first supremum is taken over all cubes $Q\subset
{\mathbb{R}^n}$
and the second supremum is taken
over all $f\in L^{1}_{\rm loc}$ such that
$\|f{\bf 1}_{Q}\|_{L^p_{\upsilon}}
\in (0,\infty)$.

\item $\upsilon\in A_{q}$ for any $q\in [p,\infty)$
and, moreover,
$[\upsilon]_{A_q }\leq [\upsilon]_{A_p
}.$
\item If $p\in (1,\infty)$ and $\mu:= \upsilon^{1-p'}$, then $\mu \in
A_{p'}$,
$[\mu]_{A_{p'}}^{p-1}=[\upsilon]_{A_{p}}$, and
$[L_{\upsilon}^p ]'=L_{\mu
}^{p'}$,
where $[L_{\upsilon}^p ]'$ denotes the associate space
of
$L_{\upsilon}^p $.

\item If $p\in (1,\infty)$, then $\mathcal{M}$
is bounded on $L_{\upsilon}^p $ and, moreover, there
exists a positive constant $C$,
independent of $\upsilon$, such that
$\|{\mathcal M}\|_{L_{\upsilon}^p  \to
L_{\upsilon}^p }
\leq C[\upsilon]^{p'-1}_{A_p }$.
\end{enumerate}
\end{lemma}

At the end of this subsection, we show that the ``drawback"
of the strong-type norm.

\begin{proposition}\label{pro-defect}
Let $X$ be a {\rm BQBF} space, $k\in \mathbb{N}$, $s,q\in (0,\infty)$, and
$f\in \mathscr{M}$.
Assume that
$X^{\frac{1}{q}}$ is a {\rm BBF} space.
If $s\min\{1,q\}\in [k,\infty)$ and
\begin{align}\label{e-sfp}
\left\|\left[\int_{\mathbb{R}^n}
\frac{|\Delta^k_h f(\cdot)|^q}{|h|^{n+sq}}\,dh\right]^{\frac{1}{q}}\right\|_X
<\infty,
\end{align}
then $f$
coincides almost everywhere with a polynomial of degree at most $k-1$.
\end{proposition}

\begin{remark}
Proposition \ref{pro-defect} when $k=1$
coincides with \cite[Theorem 4.3]{dlyyz-bvy}.
Proposition \ref{pro-defect} when $X:=L^q$ and $s=k$
coincides with \cite[Proposition 1.3]{FKR15}.
The other cases of Proposition \ref{pro-defect} are new.
\end{remark}

\begin{proof}[Proof of Proposition \ref{pro-defect}]
By the assumption that $X^{\frac{1}{q}}$ is a {\rm BBF} space, Definition
\ref{def-X'}, Remark \ref{rem-x-a}(iii), and \eqref{e-sfp}, we have
\begin{align}\label{eq-defect}
\sup_{\|g\|_{(X^{\frac{1}{q}})'}=1}
\int_{\mathbb{R}^n}\int_{\mathbb{R}^n}\frac{|\Delta^k_h
f(x)|^q}{|h|^{n+sq}}\,dh g(x)\,dx
&=	\left\|\int_{\mathbb{R}^n}\frac{|\Delta^k_h
f(\cdot)|^q}{|h|^{n+sq}}\,dh\right\|_{(X^{\frac{1}{q}})''}\notag\\
&=\left\|\int_{\mathbb{R}^n}\frac{|\Delta^k_h
f(\cdot)|^q}{|h|^{n+sq}}\,dh\right\|_{X^{\frac{1}{q}}}
<
\infty.
\end{align}
For any $N\in (0,\infty)$, let $g:= {\bf 1}_{B({\bf 0},kN)}/
\|{\bf 1}_{B({\bf 0},kN)}\|_{(X^{\frac{1}{q}})'}$.
Then, from this and \eqref{eq-defect},
it follows that, for any $N\in (0,\infty)$ and $r\in (0,N)$,
\begin{align}\label{eq-ss}
&\sum_{j=0}^{\infty}
2^{j(n+sq)}r^{-(n+sq)}
\int_{2^{-(j+1)}r\le |h|<2^{-j}r}
\int_{|x|<kN}\left|\Delta_h^k f(x)\right|^q\,dx \,dh\notag\\
&\quad \le
\int_{|x|<kN}
\int_{\mathbb{R}^n}\frac{|\Delta_h^k f(x)|^q}{|h|^{n+sq}}\,dh\,dx
< \infty.
\end{align}
We next consider the following two cases on $q$ and $s$.

\emph{Case 1)} $q\in [1,\infty)$ and $s\in [k,\infty)$.
In this case, note that, for any $j\in \mathbb{Z}_{+}$ and $x,h\in \mathbb{R}^n$,
\begin{align*}
\Delta_{2^j h}f(x)=\sum_{i=0}^{2^{j}-1}\Delta_h f(x+ih)
\ \ \text{and}\ \
\Delta_{2^j h}\Delta_{h}=
\Delta_{h}\Delta_{2^j h}.
\end{align*}
From this and the discrete H\"{o}lder inequality,
we deduce that, for any $j\in \mathbb{Z}_{+}$, $N\in (0,\infty)$, and
$r\in (0,N)$,
\begin{align*}
&\int_{2^{-(j+1)}r\le |h|<2^{-j}r}
\int_{|x|<kN-kr}\left|\Delta_{2^j h}^k f(x)\right|^q\,dx \,dh\\
&\quad=
\int_{2^{-(j+1)}r\le |h|<2^{-j}r}
\int_{|x|<kN-kr}\left|\sum_{i=0}^{2^{j}-1}\Delta_h\left(\Delta_{2^j
h}^{k-1} f\right)(x+ih)\right|^q\,dx \,dh\\
&\quad\le
2^{j(q-1)}\sum_{i=0}^{2^{j}-1}\int_{2^{-(j+1)}r\le |h|<2^{-j}r}
\int_{|x|<kN-kr}\left|\Delta_h\left(\Delta_{2^j h}^{k-1}
f\right)(x+ih)\right|^q\,dx \,dh\\
&\quad\le
2^{j(q-1)}\sum_{i=0}^{2^{j}-1}\int_{2^{-(j+1)}r\le |h|<2^{-j}r}
\int_{|x|<kN-kr+i2^{-j}r}\left|\Delta_h\left(\Delta_{2^j h}^{k-1}
f\right)(x)\right|^q\,dx \,dh\\
&\quad\lesssim
2^{jq}\int_{2^{-(j+1)}r\le |h|<2^{-j}r}
\int_{|x|<kN-(k-1)r}\left|\Delta_{2^j h}^{k-1}\left(
\Delta_hf\right)(x)\right|^q\,dx \,dh.
\end{align*}
Using this $k$ times and a change of variables,
we find that, for any $j\in \mathbb{Z}_{+}$, $N\in (0,\infty)$, and
$r\in (0,N)$,
\begin{align*}
&\int_{\frac{r}{2}\le |h|<r}
\int_{|x|<kN-kr}\left|\Delta_{ h}^k f(x)\right|^q\,dx \,dh\\
&\quad = 2^{jn}\int_{2^{-(j+1)}r\le |h|<2^{-j}r}
\int_{|x|<kN-kr}\left|\Delta_{ 2^j h}^k f(x)\right|^q\,dx \,dh\\
&\quad\lesssim 2^{j(n+kq)}
\int_{2^{-(j+1)}r\le |h|<2^{-j}r}
\int_{|x|<kN}\left|\Delta_{ h}^{k}f(x)\right|^q\,dx \,dh,
\end{align*}
which, together with \eqref{eq-ss}, further implies that
\begin{align*}
\sum_{j=0}^{\infty}
2^{j(s-k)q}
\int_{\frac{r}{2}\le |h|<r}
\int_{|x|<kN}\left|\Delta_h^k f(x)\right|^q\,dx \,dh<\infty.
\end{align*}
By this, $s\in [k,\infty)$, and the arbitrariness
of both $N$ and $r$, we obtain
\begin{align*}
\int_{\mathbb{R}^n}
\int_{\mathbb{R}^n}\left|\Delta_h^k f(x)\right|^q\,dx \,dh=0,
\end{align*}
which, combined with \cite[Lemma 3.6]{dssvy}, further implies that $f$
coincides almost everywhere with a polynomial of degree at most $k-1$.

\emph{Case 2)} $q\in (0,1)$ and $sq\in [k,\infty) $.
In this case, repeating the proof of
Case 1) with the discrete H\"{o}lder inequality replaced by the Jensen
inequality, we obtain the desired result.
This finishes the proof of Proposition \ref{pro-defect}.
\end{proof}

\subsection{Density Properties of Higher-Order Homogeneous Ball Banach
Sobolev Spaces}\label{ssub-density}
The main target of this subsection is to obtain the following density
properties of $\dot{W}^{k,X}$.
\begin{theorem}\label{thm-e}
Let $k\in \mathbb{N}$ and $X$ be a {\rm BBF} space
having an absolutely continuous norm such that
the centered ball average operators are uniformly
bounded on $X$.
\begin{enumerate}[{\rm (i)}]
\item For any $f\in \dot{W}^{k,X}$, there exists a
sequence
$\{f_m\}_{m\in {\mathbb{N}}}$ in $C^{\infty}$
with $|\nabla^k f_m|\in C_{\rm c}$
for any $m\in \mathbb{N}$
such that, for any $R\in (0,\infty)$,
\begin{align}\label{eq-extension}
\lim_{m\to\infty}\|f-f_m\|_{\dot{W}^{k,X}}=0
\ \ \text{and}\ \
\lim_{m\to\infty}\|(f-f_m){\bf 1}_{B({\bf 0},R)}\|_{X}=0.
\end{align}
\item If the Hardy--Littlewood maximal operator $\mathcal{M}$ is
bounded on $X$ or $n\ge 2$, then,
for any $f\in \dot{W}^{k,X}$, there exists a sequence
$\{f_m\}_{m\in {\mathbb{N}}}$ in $C_{\rm c}^{\infty}$
such that
\begin{align*}
\lim_{m\to\infty}\|f-f_m\|_{\dot{W}^{k,X}}=0
\ \ \text{and}\ \
\lim_{m\to\infty}\|(f-f_m){\bf 1}_{B({\bf 0},R)}\|_{X}=0.
\end{align*}
\end{enumerate}
\end{theorem}

\begin{remark}
\begin{enumerate}[{\rm (i)}]
\item Theorem \ref{thm-e}(i) when $k=1$ coincides with
\cite[Theorem 2.6]{dlyyz-bvy}.
\item Let $k\in \mathbb{N}$ and $p\in [1,\infty)$. In \cite[Theorem
4]{hk95},
Haj\l asz and
Ka\l amajska
proved that $C_{\rm c}^{\infty}$
is dense in $\dot{W}^{k,p}$
if and only if either $p>1$ or $n\ge 2$.
Clearly,
Theorem \ref{thm-e}(ii) when $X:=L^p$ coincides with
\cite[Theorem 4]{hk95}.
\end{enumerate}
\end{remark}

To show Theorem \ref{thm-e}, we require the following Poincar\'{e}
inequality on {\rm BBF} spaces.
\begin{lemma}\label{lem-poin}
Let $k\in{\mathbb{N}}$, $R\in (0,\infty)$,
and
$X$ be a {\rm BBF} space satisfying that the centered ball average
operators are uniformly bounded on $X$.
Assume that $\Omega$ is a ball with the radius $R$ or a cube
with the edge length $R$ or that $n\ge 2$ and
$\Omega:=\{x\in \mathbb{R}^n:\ R< |x|< 2R\}$ is an
annulus.
Then there exists a positive constant
$C$, independent of $\Omega$,
such that,
for any $f\in \dot{W}^{k,X}$,
there exists a polynomial $P\in \mathcal{P}_{k-1}$
such that, for any 	$j\in\mathbb{Z}_+ \cap [0,k-1]$,
\begin{align*}
\left\|\nabla^j(f-P){\bf 1}_{\Omega}\right\|_X
\le CR^{k-j}
\left\|\nabla^k f {\bf 1}_{\Omega}\right\|_X.
\end{align*}
\end{lemma}

\begin{proof}
Repeating
the proof of \cite[Section 1.1.11, Lemma]{ma2011} with $\|\cdot\|_{L^p}$
and $\mathcal{B}_i$
therein replaced, respectively, by $\|\cdot\|_X$ and the ball $B$ with
the radius $R$ here, we conclude the desired
conclusion of Lemma \ref{lem-poin}.
\end{proof}

Repeating
the proof of \cite[Proposition 2.15]{dlyyz-bvy}
with $\|\cdot\|_{{\dot W}^{1,X}}$ replaced by
$\|\cdot\|_{{\dot W}^{k,X}}$,
we obtain the following lemma; we omit the details here.

\begin{lemma}\label{lem-mod}
Let $k$ and $X$ be the same as in Theorem \ref{thm-e}.
Then, for any $f\in \dot{W}^{k,X}$, there exists a
sequence
$\{f_m\}_{m\in{\mathbb{N}}}\subset C^{\infty}\cap
\dot{W}^{k,X}$
such that, for any $R\in (0,\infty)$,
\begin{align*}
\lim_{m\to\infty}\|f-f_m\|_{\dot{W}^{k,X}}=0
\ \ \text{and}\ \
\lim_{m\to\infty}\|(f-f_m){\bf 1}_{B({\bf 0},R)}\|_{X}=0.
\end{align*}
\end{lemma}

In what follows,
we denote by $\mathcal{A}^k$ the $k$-fold iteration of an operator
$\mathcal{A}$.
Next, we give the proof of
Theorem \ref{thm-e}.

\begin{proof}[Proof of Theorem \ref{thm-e}]
By Lemma \ref{lem-mod}, we may assume
$f\in C^{\infty}\cap \dot{W}^{k,X}$.
Choose $\varphi\in C_{\rm c}^{\infty}$
such that $0\leq \varphi\leq 1$, $\mathrm{supp\,}(\varphi)\subset
B({\bf 0},2)$,
and $\varphi\equiv 1$ in $B({\bf 0},1)$.
For any $m\in{\mathbb{N}}$ and $x\in{\mathbb{R}^n}$, let
$\eta_{m}(x):=\varphi(x/m)$
and
$\Omega_m:=\{x\in \mathbb{R}^n: m<|x|<2m\}$.
To prove (i), we consider the following
two cases on $n$.

\emph{Case 1)} $n=1$.
In this case,
for any $g\in C^{\infty}(\mathbb{R})$ and
$x\in \mathbb{R}$, let
\begin{align}\label{eq-def-a}
\mathcal{A}(g)(x):=\int_{0}^{x}g(t)\,dt.
\end{align}
For any $m\in \mathbb{N}$ and $i\in \mathbb{N}\cap [1,k]$, let
\begin{align*}
T_{1,m} (f):=\mathcal{A}\left(f^{(k)}\eta_m\right)
+f^{(k-1)}(0),\
T_{i,m} (f):=\mathcal{A}\left(T_{i-1,m} (f)\right)+f^{(k-i)}(0),
\end{align*}
and $f_m:=T_{k,m}(f)$, where $f^{(0)}:=f$.
Notice that, for any $m\in \mathbb{N}$, $f_m^{(k)}=f^{(k)}\eta_m$.
From this and the assumption that $X$ has an absolutely
continuous norm, we deduce that
\begin{align*}
\lim_{m\to \infty}\left\|f^{(k)}_m-f^{(k)}\right\|_X
=
\lim_{m\to \infty}\left\|f^{(k)}\eta_m-f^{(k)}\right\|_X
=0.
\end{align*}
Also observe that, for any $m\in \mathbb{N}\cap [R,\infty)$ and $x\in
B({\bf 0},R)$,
$f_m (x)=f(x)$.
Hence,
\begin{align*}
\lim_{m\to \infty}\left\|(f-f_m){\bf 1}_{B({\bf 0},R)}\right\|_X
=
0.
\end{align*}
This finishes the proof of (i) in this case.

\emph{Case 2)} $n\ge 2$.
In this case, by Lemma \ref{lem-poin},
we find that
there exist
polynomials $\{P_m\}_{m\in \mathbb{N}}\subset
\mathcal{P}_{k-1}$ such that, for any $m\in \mathbb{N}$
and $j\in \mathbb{Z}_+$,
\begin{align}\label{eq-den01}
\left\|\nabla^j(f-P_m){\bf 1}_{\Omega_{m}}\right\|_X
\lesssim
m^{k-j}
\left\|\,\left|\nabla^k f\right| {\bf 1}_{\Omega_m}\right\|_X.
\end{align}
Then, for any $m\in \mathbb{N}$, let
$
f_m:= (f-P_m)\eta_m +P_m.
$

Next, we show that $f_m \to f$ in $\dot{W}^{k,X}$ as
$m\to \infty$.
To do this, we fix $\alpha\in \mathbb{Z}_{+}^n$ with $|\alpha|=k$.
Using the Newton--Leibniz formula, we obtain, for any
$m\in{\mathbb{N}}$,
\begin{align}\label{eq-newton}
\partial^{\alpha}(f_m -f)
=
\sum_{\beta\in \mathbb{Z}_{+}^n,\, \beta\leq
\alpha}\binom{\alpha}{\beta}
\partial^{\beta}\left(f-P_m\right)\partial^{\alpha-\beta}(\eta_{m}-1).
\end{align}
When $\beta=\alpha$, via the assumption that $X$ has an absolutely
continuous norm, we find that
\begin{align}\label{eq-den02}
\left\|\partial^{\beta}\left(f-P_m\right)\partial^{\alpha-\beta}(\eta_{m}-1)\right\|_X
=
\left\|\left(\partial^{\alpha}f\right)(\eta_{m}-1)\right\|_X \to 0
\end{align}
as $m\to \infty$.
On the other hand, when $\beta\in \mathbb{Z}_{+}^n$ with both
$\beta\leq \alpha$ and $\beta\neq \alpha$,
notice that, for any $m\in{\mathbb{N}}$,
$$\mathrm{supp\,}\left(\partial^{\alpha-\beta}(\eta_{m}-1)\right)\subset
\Omega_{m}$$
and, for any $x\in \mathbb{R}^n$,
\begin{align*}
\left|\partial^{\alpha-\beta}(\eta_{m}-1)(x)\right|
=
m^{|\beta|-k}\left|\partial^{\alpha-\beta}\varphi\left(\frac{x}{m}\right)\right|
\lesssim
m^{|\beta|-k},
\end{align*}
which, together with \eqref{eq-den01} and the assumption
that $X$ has an absolutely continuous norm again,
further imply that
\begin{align*}
\left\|\partial^{\beta}\left(f-P_m\right)\partial^{\alpha-\beta}(\eta_{m}-1)\right\|_{X}
&\lesssim
m^{|\beta|-k}
\left\|\partial^{\beta}\left(f-P_m\right){\bf
1}_{\Omega_m}\right\|_{X}\\
&\lesssim
\left\|\,\left|\nabla^k f\right|{\bf 1}_{\Omega_{m}}\right\|_X
\le
\left\|\,\left|\nabla^k f\right|{\bf 1}_{[B({\bf
0},m)]^{\complement}}\right\|_X
\to 0
\end{align*}
as $m\to \infty$.
From this, \eqref{eq-newton}, and \eqref{eq-den02},
it follows that
\begin{align}\label{eq-den05}
\lim_{m\to \infty}\|f-f_m\|_{\dot{W}^{k,X}}
\lesssim
\sum_{\alpha\in \mathbb{Z}_{+}^n, \,|\alpha|=k}
\lim_{m\to \infty}\left\|\partial^{\alpha}(f_m -f)\right\|_X=0.
\end{align}

Finally, observe that,
for any $m\in {\mathbb{N}}\cap [R,\infty)$ and $x\in B({\bf 0},R)$,
$\eta_{m}(x)=1$ and hence
\begin{align*}
f_m (x)=\left[f(x)-P_m (x)\right]\eta_m (x) +P_m (x)=f(x).
\end{align*}
It follows that
$\lim_{m\to \infty}\|(f-f_m){\bf 1}_{B({\bf 0},R)}\|_X=0,$
which, combined with \eqref{eq-den05}, further implies (i) in this
case and hence completes the proof of (i).

It remains to show (ii).
For this purpose,
we also consider the following
two cases.

\emph{Case I)} $n=1$
and the Hardy--Littlewood maximal operator $\mathcal{M}$ is bounded on
$X$. In this case,
let
$g_k :=\mathcal{A}^k(f^{(k)})$
and
$f_m:=\varphi(\frac{\cdot}{m})g_k$,
where $\mathcal{A}$ is the same as in \eqref{eq-def-a}. Clearly, for
any $m\in \mathbb{N}$, $f_{m}\in C_{\rm c}^{\infty}(\mathbb{R})$.

Next, we claim that
\begin{align}\label{eq-c-den}
f_m^{(k)}\to f^{(k)}
\end{align}
in $X$ as $m\to \infty$. Indeed, via the Newton--Leibniz formula, we
find that, for any $m\in \mathbb{N}$,
\begin{align*}
f_m^{(k)}	& =
\sum_{\ell=0}^{k}
\binom{k}{\ell}m^{-\ell}\varphi^{(\ell)}\left(\frac{\cdot}{m}\right)g^{(k-\ell)}_{k}(\cdot)=:\sum_{\ell=0}^{k}F_{\ell,m}.
\end{align*}
Notice that, for any $1\le\ell \le k $, $m\in\mathbb{N}$, and $x\in
\mathbb{R}$,
\begin{align*}
\left|F_{\ell,m}(x)\right|
&\lesssim \left||x|^{-\ell}g_{k}^{(k-\ell)}(x)\right|{\bf
1}_{\Omega_m}(x)
\leq
|x|^{-\ell}\mathcal{A}^{\ell}\left(\left|f^{(k)}\right|\right)(x){\bf
1}_{\Omega_m}(x)\\
&\lesssim
\mathcal{M}^{\ell}\left(f^{(k)}\right)(x){\bf 1}_{\Omega_m}(x),
\end{align*}
which, together with the boundedness of $\mathcal{M}$ and the assumption
that $X$ has an absolutely continuous norm, further implies
that $F_{\ell,m}\to 0$
in $X$ as $m\to \infty$.
On the other hand, it is easy to find that $F_{0,m}\to f^{(k)}$ in $X$
as $m\to \infty$. Therefore, the above claim \eqref{eq-c-den} holds.
Using
\eqref{eq-c-den}, we immediately obtain (ii) in this case.

\emph{Case II)} $n\geq 2$. In this case, repeating the proof
of Case 2) with	$f_m$ replaced by $(f-P_m)\eta_m$, we can find that
(ii) holds. This finishes the proof of (ii) and hence
Theorem \ref{thm-e}.
\end{proof}

\section{Upper Estimates in Weighted Sobolev Spaces}\label{sec-AP-es}

In this section, we first establish a
higher-order weighted variant of the inequality
of Cohen et al. \cite[Theorems 3.1 and 4.1]{cddd03}.
Applying this, we further obtain the
upper estimate of \eqref{eq-main-01}, which plays a
key role in the proof of all the main results
in Section \ref{sec-Formulae}. Finally,
we show the above two weighted estimates are sharp
and use them to characterize Muckenhoupt weights when
$n=1$.

To state these results, we first recall some basic concepts.
For any $f\in L^{1}_{\rm loc}$,
$s\in \mathbb{Z}_{+}$, and ball $B\subset \mathbb{R}^n$,
let
$P^{(s)}_{B}(f)$ denote the unique \emph{minimizing polynomial} in
$\mathcal{P}_{s}$ such that, for any
$\alpha\in\mathbb{Z}_{+}^n$ with $|\alpha|\leq s$,
\begin{align*}
\int_{B}\left[f(x)-P^{(s)}_{B}(f)(x)\right]x^{\alpha}\,dx=0;
\end{align*}
for any cube $Q\subset \mathbb{R}^n$, $P^{(s)}_{Q}(f)$ is defined
in a similar way.
For any $\alpha\in\{0,\frac{1}{3},\frac{2}{3}\}^n$, the \emph{shifted
dyadic grid} $\mathcal{D}^\alpha$ is defined by setting
\begin{align}\label{dyadic_grid}
\mathcal{D}^\alpha:=\left\{2^j\left[m+[0,1)^n+(-1)^j\alpha\right]:\
j\in\mathbb{Z},\ m\in\mathbb{Z}^n\right\}.
\end{align}
Let $k,\ell\in \mathbb{N}$ with $\ell\le k$
and let
$\alpha\in\{0,\frac13,\frac23\}^n$.
For any $f\in L^1_{\rm loc}$ and
$Q\in\mathcal{Q}$, the \emph{local approximation}
$E_k(f,Q)$ of $f$ of order $k$ on $Q$ is defined by setting
\begin{align*}
E_k(f,Q):=
\left\|f-P^{(k-1)}_{Q}(f)\right\|_{L^1(Q)}
\end{align*}
and, for any
$\beta\in \mathbb{R}$ and $\lambda\in(0,\infty)$,
let
\begin{align}\label{df-ab}
\mathcal{D}^{\alpha}_{\lambda,\beta,k,\ell}[f]
:=\left\{Q\in\mathcal{D}^{\alpha}:\
E_k(f,Q)>\lambda|Q|^{\beta+\frac{\ell}{n}}\right\};
\end{align}
in particular, when $\ell=k$,
let
$\mathcal{D}^{\alpha}_{\lambda,\beta,k}[f]:=\mathcal{D}^{\alpha}_{\lambda,\beta,k,k}[f]$.
Then we have the following estimate.

\begin{theorem}\label{esti-omega}
Let $k,\ell\in \mathbb{N}$ with $\ell\le k$, $p\in [1,\infty)$, and
$\beta\in (-\infty,1-\frac{1}{n})\cup (1,\infty)$ when $p=1$ or
$\beta\in \mathbb{R}\setminus\{1\}$ when $p\in (1,\infty)$.
Then there exist a positive constant $C$
and an increasing continuous function $\varphi$ on $[0,\infty)$
such that,
for any $\alpha\in\{0,\frac13,\frac23\}^n$,
$\upsilon\in A_p$,
and $f\in\dot{W}^{\ell,p}_{\upsilon}$,
\begin{align}\label{eq-est-m1}
\sup_{\lambda\in(0,\infty)}
\lambda^p \sum_{Q\in\mathcal{D}^{\alpha}_{\lambda,\beta,k,\ell}[f]}
|Q|^{p(\beta-1)}\upsilon(Q)
\le
C \varphi\left([\upsilon]_{A_p}\right)
\int_{\mathbb{R}^n}\left|\nabla^{\ell}f(x)\right|^p
\upsilon(x)
\,dx.
\end{align}
\end{theorem}

\begin{remark}
In Theorem \ref{esti-omega}, if
$k=\ell=1$, then Theorem \ref{esti-omega} exactly coincides with
\cite[Proposition 2.3]{lyyzz24}; in particular,
if further assume $p=1$,
$\upsilon\equiv1$, and $\alpha=(0,\ldots,0)$,
then Theorem \ref{esti-omega} exactly
coincides with Cohen et al.
\cite[Theorems 3.1 and 4.1]{cddd03}.
\end{remark}

In what follows,
for any $\lambda\in (0,\infty)$,
$b\in \mathbb{R}$, $k,\ell\in {\mathbb{N}}$ with $\ell\le k$,
and $f\in \mathscr{M}$,
let
\begin{align}\label{eq-E}
E_{\lambda,b,k,\ell}[f]
:=\left\{(x,h)\in {\mathbb{R}^n} \times
\left({\mathbb{R}^n}\setminus\{{\bf 0}\}\right):\
\left|\Delta_{h}^{k}f (x)\right|>\lambda |h|^{b+\ell}\right\}.
\end{align}

Now, we give the upper estimate of higher-order BSVY formula
in weighted Sobolev spaces as follows.
\begin{theorem}\label{thm-up-ap}
Let $\upsilon\in A_1$, $k,\ell\in \mathbb{N}$ with $\ell\le k$, $p\in[1,\infty)$, $q\in
(0,\infty)$ satisfy $n(\frac{1}{p}-\frac{1}{q})<\ell$,
$\Gamma_{p,q}$ be the same as in \eqref{eq-GAMMA}, and
$\gamma\in\Gamma_{p,q}$.
Then there exist a positive constant $C$,
independent of $\upsilon$,
and an increasing continous function $\psi$ on $[0,\infty)$
such that, for any $f\in \dot{W}^{\ell,p}_{\upsilon}$,
\begin{align}\label{e-up-ap}
&\sup_{\lambda\in(0,\infty)}
\lambda^p \int_{\mathbb{R}^n}
\left[\int_{\mathbb{R}^n}{\bf 1}_{E_{\lambda,
\frac{\gamma}{q},k,\ell}[f]}(x,h)|h|^{\gamma-n}\,dh
\right]^{\frac{p}{q}}\upsilon(x)\,dx\notag\\
&\quad\le
C\psi\left([\upsilon]_{A_1}\right)
\int_{\mathbb{R}^n}\left|\nabla^{\ell}f(x)\right|^p
\upsilon(x)\,dx,
\end{align}
where
$E_{\lambda,
\frac{\gamma}{q},k,\ell}[f]$ is the same as in \eqref{eq-E} with
$b$ repalced by $\frac{\gamma}{q}$.
\end{theorem}

\begin{remark}
\begin{enumerate}[{\rm (i)}]
\item In Theorem \ref{thm-up-ap}, if $k=\ell=1$, then, in this case,
Theorem \ref{thm-up-ap} coincides with
\cite[Proposition 3.10]{lyyzz24} with
$X$ therein replaced by $L^p_\upsilon$ here,
where $p\in[1,\infty)$ and $\upsilon\in A_1$;
in particular, if further assume $p=q\in[1,\infty)$ and
$\upsilon\equiv1$,
then Theorem \ref{thm-up-ap} in this case exactly
coincides with the upper estimates
\cite[(1-13) and (1-15)]{bsvy24}
of the BSVY formula.
Moreover, Theorem \ref{thm-up-ap}
when $k\in\mathbb{N}\cap[2,\infty)$ is completely new.

\item In Theorem \ref{thm-up-ap},
the assumption $n(\frac{1}{p}-\frac{1}{q})<\ell$ is sharp, which can be
deduced from Proposition \ref{pro-sharp}.
\end{enumerate}

\end{remark}

In particular, when $n=1$ and $k=\ell$, the following
theorem implies that the condition $\upsilon\in A_p (\mathbb{R})$
in Theorems \ref{esti-omega} and \ref{thm-up-ap}
is sharp.

\begin{corollary}\label{cor-ap}
Let $\upsilon \in L^{1}_{\rm loc}(\mathbb{R})$ be nonnegative, $k\in
\mathbb{N}$,
$p\in [1,\infty)$, and $q\in (0,\infty)$ satisfy $1-\frac{1}{q}<k$.
Then the following three statements are mutually equivalent.
\begin{enumerate}[{\rm (i)}]
\item $\upsilon\in A_p (\mathbb{R})$.
\item
There exist $\gamma\in \Gamma_{p,q}$ and $C\in (0,\infty)$ such that,
for any $f\in \dot{W}^{k,p}_{\upsilon} (\mathbb{R})$,
\begin{align}\label{eq-cor-ap}
&\sup_{\lambda\in(0,\infty)}
\lambda^p \int_{\mathbb{R}}
\left[\int_{\mathbb{R}}{\bf 1}_{E_{\lambda,
\frac{\gamma}{q},k}[f]}(x,h)|h|^{\gamma-1}\,dh
\right]^{\frac{p}{q}}\upsilon(x)\,dx
\le C \int_{\mathbb{R}^n}\left|\nabla^{k}f(x)\right|^p
\upsilon(x)\,dx,
\end{align}
where $\Gamma_{p,q}$ and
$E_{\lambda,\frac{\gamma}{q},k}[f]$ are the same as, respectively,
in \eqref{eq-GAMMA} and \eqref{e-EE}.
\item There exist
$\beta\in (-\infty, 0)\cup(1,\infty)$
when $p=1$ or $\beta\in \mathbb{R}\setminus\{1\}$
when $p\in (1,\infty)$ and $C\in (0,\infty)$
such that,
for any $\alpha\in \{0,\frac{1}{3},\frac{2}{3}\}$ and
$f\in \dot{W}^{k,p}_{\upsilon} (\mathbb{R})$,
\begin{align}\label{eq-n1-ap}
\sup_{\lambda\in(0,\infty)}
\lambda^p \sum_{Q\in\mathcal{D}^{\alpha}_{\lambda,\beta,k}[f]}
|Q|^{p(\beta-1)}\upsilon(Q)\le
C
\int_{\mathbb{R}^n}\left|\nabla^{k}f(x)\right|^p
\upsilon(x)\,dx.
\end{align}
\end{enumerate}
\end{corollary}

\begin{remark}\label{rem-ques}
Corollary \ref{cor-ap} when $k=1$
has been established in \cite[Theorems 1.1 and 1.3]{lyyzz24}
and, when $k\in\mathbb{N}\cap[2,\infty)$,
Corollary \ref{cor-ap} is completely new.
Recall that \cite[Theorems 1.1 and 1.3]{lyyzz24} also includes
the characterization
of $A_p$ with $p\in[1,\infty)$ and $n\in\mathbb{N}\cap[2,\infty)$.
However,
it seems that there exists no appropriate alternative
of $\mathcal{A}^k$ when $n\in\mathbb{N}\cap[2,\infty)$,
where $\mathcal{A}$ is the same as in \eqref{eq-aaaa}.
Thus,
it is still unknown how to generalize
Corollary \ref{cor-ap} to the case $n\in\mathbb{N}\cap[2,\infty)$.
\end{remark}

We provide the detailed proofs of
Theorems \ref{esti-omega} and \ref{thm-up-ap},
as well as Corollary \ref{cor-ap}, respectively, in Subsections \ref{ssub-3.1}, \ref{ssub-3.2}, and \ref{ssub-3.3}.

\subsection{Proof of Theorem \ref{esti-omega}}\label{ssub-3.1}

To prove Theorem \ref{esti-omega},
we need the following technique lemma about minimizing polynomials
and local approximations, which can
be easily deduced from the definition of minimizing polynomials;
see \cite[p.\,83]{TG80} and \cite[Lemma 4.1]{Lu95} for more details.

\begin{lemma}\label{uni-p}
Let $\Omega$ be a ball or a cube in $\mathbb{R}^n$ and
$s\in\mathbb{Z}_{+}$.
\begin{enumerate}[{\rm (i)}]
\item For any $P\in \mathcal{P}_{s}$,
$
P^{(s)}_{\Omega}(P)=P.
$
\item There exists a constant ${C_{(s)}}\in [1,{\infty })$,
depending only on $s$, such that, for any $f\in L^{1}_{\rm
loc}$ and $x\in \Omega$,
\begin{align*}
\left| P_{\Omega}^{(s)}(f)(x)\right| \leq
C_{(s)}\fint_{\Omega}|f(y)|\,dy.
\end{align*}
\item For any given $p\in
[1,\infty)$ and any $f\in L^{1}_{\rm loc}$,
\begin{align}\label{uni-sim}
\left\|f-P^{(s)}_{\Omega}(f)\right\|_{L^p(\Omega)}
\sim
\inf_{P\in \mathcal{P}_s}
\left\|f-P\right\|_{L^p(\Omega)},
\end{align}
where the implicit positive constants depend only on $n$ and $s$.

\item Let $k,\ell\in \mathbb{N}$ with
$\ell\le k$. Then there exists a positive constant
$C_{(n,k,\ell)}$, depending only on $n$, $k$,
and $\ell$, such that, for any $f\in L^{1}_{\rm loc}$,
\begin{align}\label{eq-Ekl}
E_{k}(f,\Omega)\le C_{(n,k,\ell)}E_{\ell}(f,\Omega).
\end{align}
\end{enumerate}
\end{lemma}

Also, the following exquisite geometrical properties
of shifted dyadic grids on Euclidean spaces play
key roles in this section (see, for instance, \cite[p.\,479]{mtt02}).
\begin{lemma}\label{lem-dy-cub}
For any $\alpha\in\{0,\frac{1}{3},\frac{2}{3}\}^n$,	let
$\mathcal{D}^\alpha$ be the same as in \eqref{dyadic_grid}.
\begin{enumerate}
\item[\textup{(i)}]
For any $Q,P\in\mathcal{D}^\alpha$ with
$\alpha\in\{0,\frac{1}{3},\frac{2}{3}\}^n$,
$Q\cap P\in\{\emptyset,Q,P\}$.
\item[\textup{(ii)}]
There exists a positive constant $C_{(n)}$,
depending only on $n$, such that,
for any ball $B\subset\mathbb{R}^n$, there exist
$\alpha\in\{0,\frac{1}{3},\frac{2}{3}\}^n$ and
$Q\in\mathcal{D}^\alpha$
such that $B\subset Q\subset C_{(n)}B$.
\end{enumerate}
\end{lemma}

In what follows,
for any $\alpha\in\{0,\frac{1}{3},\frac{2}{3}\}^n$,
$\beta\in \mathbb{R}$, $\lambda\in (0,\infty)$, $k,\ell\in \mathbb{N}$
with $\ell\le k$, and $f\in L^{1}_{\rm loc}$, let
\begin{align*}
{\Omega}^{\alpha}_{\lambda,\beta,k,\ell}[f]:=
\bigcup_{Q\in \mathcal{D}^{\alpha}_{\lambda,\beta,k.\ell}[f]} Q.
\end{align*}
Applying the higher-order Poincar\'{e} inequality,
we obtain the following key lemma,
which essentially characterizes the sparseness
of cubes in the level set
$\mathcal{D}^{\alpha}_{\lambda,\beta,k,\ell}[f]$.
Moreover, this lemma is
of independent interest and is frequently used in this article
(see \cite[Lemma 2.5]{lyyzz24} and \cite[Lemma 2.7]{zyyy24} for its more applications).

\begin{lemma}\label{lem-Qx}
Let $\alpha\in\{0,\frac{1}{3},\frac{2}{3}\}^n$,
$k,\ell\in \mathbb{N}$ with $\ell\le k$, $p,\lambda\in
(0,\infty)$,
$\beta\in \mathbb{R}\setminus \{1\}$, and $f\in
C^{\infty}$ with $|\nabla^{\ell}f|\in C_{\rm
c}$.
Then, for any $x\in {\Omega}^{\alpha}_{\lambda,\beta,k,\ell}[f]$,
there exists a unique cube $Q_x \in
\mathcal{D}^{\alpha}_{\lambda,\beta,k,\ell}[f]$ containing $x$
such that
\begin{align}\label{Qx-es}
\sum_{Q\in \mathcal{D}^{\alpha}_{\lambda,\beta,k,\ell}[f]}
|Q|^{p(\beta-1)}{\bf 1}_{Q}(x)
\sim
|Q_x|^{p(\beta-1)},
\end{align}
where the positive equivalence constants depend only on $n,p$, and
$\beta$.
\end{lemma}

\begin{proof}
From \eqref{df-ab}, \eqref{eq-Ekl}, Lemma \ref{lem-poin} with
$X:= L^{1}$, \eqref{uni-sim},
and the assumption $f\in C^{\infty}$ with
$|\nabla^{\ell}f|\in C_{\rm c}$, we infer that, for any
$Q\in\mathcal{D}^{\alpha}_{\lambda,\beta,k,\ell}[f]$,
\begin{align}\label{eq-bb00}
\lambda |Q|^{\beta}
&<
|Q|^{-\frac{\ell}{n}}E_{k}(f,Q)
\lesssim|Q|^{-\frac{\ell}{n}}E_{\ell}(f,Q)
\lesssim \int_{Q}\left|\nabla^{\ell} f(z)\right|\,dz\notag\\
&\le|Q|\left\|\nabla^\ell
f\right\|_{L^\infty}<\infty.
\end{align}
Now, we fix $x\in {\Omega}^{\alpha}_{\lambda,\beta,k,\ell}[f]$ and
consider the following two cases on $\beta\in \mathbb{R}\setminus
\{1\}$.

\emph{Case 1)} $\beta\in (-\infty,1)$. In this case, applying
\eqref{eq-bb00}, we find that
\begin{align*}
\inf_{Q\in\mathcal{D}^{\alpha}_{\lambda,\beta,k,\ell}[f]}|Q|
\gtrsim
\left[\lambda^{-1}\left\|\nabla^\ell
f\right\|_{L^\infty}\right]^{\frac{1}{\beta-1}}
>0.
\end{align*}
Therefore, from Lemma \ref{lem-dy-cub}(i), we deduce that there exists
a unique cube $Q_x \in \mathcal{D}^{\alpha}_{\lambda,\beta,k,\ell}[f]$
that is
minimum with respect to the set inclusion such that $x\in Q_x$;
moreover, for any $j\in \mathbb{Z}_{+}$,
\begin{align*}
\sharp\left\{Q\in\mathcal{D}^{\alpha}_{\lambda,\beta,k,\ell}[f]:\
Q\supset Q_x,\ |Q|=2^{jn}|Q_x| \right\}
\le 1,
\end{align*}
where $\sharp E$ denotes the \emph{cardinality} of the set $E$.
Using these and the assumption $\beta\in (-\infty,1)$
again, we conclude that
\begin{align}\label{Qx02}
\sum_{Q\in \mathcal{D}^{\alpha}_{\lambda,\beta,k,\ell}[f]}
|Q|^{p(\beta-1)}{\bf 1}_{Q}(x)
&=
\sum_{j\in\mathbb{Z}_{+}}\sum_{\{Q\in\mathcal{D}^{\alpha}_{\lambda,\beta,k,\ell}[f]:\
Q\supset Q_x,\ |Q|=2^{jn}|Q_x| \}}|Q|^{p(\beta-1)}\notag\\
&\leq
|Q_x|^{p(\beta-1)}\sum_{j\in \mathbb{Z}_{+}}2^{jnp(\beta-1)}
\sim |Q_x|^{p(\beta-1)}.
\end{align}
On the other hand, since $Q_x \in \mathcal{D}^{\alpha}_{\lambda,\beta,k,\ell}[f]$,
it follows that
\begin{align*}
|Q_x|^{p(\beta-1)}
\lesssim
\sum_{Q\in \mathcal{D}^{\alpha}_{\lambda,\beta,k,\ell}[f]}
|Q|^{p(\beta-1)}{\bf 1}_{Q}(x).
\end{align*}
This, combined with \eqref{Qx02}, further implies that \eqref{Qx-es}
holds in this case.

\emph{Case 2)} $\beta\in (1,\infty)$. In this case, by
\eqref{eq-bb00}, one has
\begin{align*}
\sup_{Q\in\mathcal{D}^{\alpha}_{\lambda,\beta,k,\ell}[f]}
|Q|
\lesssim
\left[\lambda^{-1} \left\|\nabla^k
f\right\|_{L^\infty}\right]^\frac{1}{\beta-1}<\infty.
\end{align*}
Thus, from Lemma \ref{lem-dy-cub}(i), it follows that there exists a
unique cube $Q_x \in \mathcal{D}^{\alpha}_{\lambda,\beta,k,\ell}[f]$
that is
maximum with respect to the set inclusion such that $x\in Q_x$;
furthermore, for any $j\in \mathbb{Z}_{+}$,
\begin{align*}
\sharp\left\{Q\in\mathcal{D}^{\alpha}_{\lambda,\beta,k,\ell}[f]:\
x\in Q\subset Q_x,\ |Q|=2^{-jn}|Q_x| \right\}\le 1.
\end{align*}
Applying these and
the assumption $\beta\in (1,\infty)$, we
find that
\begin{align*}
\sum_{Q\in \mathcal{D}^{\alpha}_{\lambda,\beta,k,\ell}[f]}
|Q|^{p(\beta-1)}{\bf 1}_{Q}(x)
&=
\sum_{j\in\mathbb{Z}_{+}}\sum_{\{Q\in\mathcal{D}^{\alpha}_{\lambda,\beta,k,\ell}[f]:\
Q\subset Q_x,\ |Q|=2^{-jn}|Q_x| \}}|Q|^{p(\beta-1)}\\
&\leq
|Q_x|^{p(\beta-1)}\sum_{j\in \mathbb{Z}_{+}}2^{-jnp(\beta-1)}
\sim |Q_x|^{p(\beta-1)}.
\end{align*}
From this and $Q_x \in \mathcal{D}^{\alpha}_{\lambda,\beta,k,\ell}[f]$,
we deduce that \eqref{Qx-es}
holds also in this case.
This finishes the proof of
Lemma \ref{lem-Qx}.
\end{proof}

Recall that, for any $Q\in\mathcal{Q}$
and $f\in L^1_{\rm loc}$,
the \emph{renormalized averaged modulus of continuity},
$\omega_Q(f)$, of $f$ is defined by setting
\begin{align*}
\omega_Q(f):=|Q|^{-1-\frac1n}
\int_{Q}\int_{Q}|f(x)-f(y)|\,dx\,dy.
\end{align*}
Via the above preparations, we now show Theorem \ref{esti-omega}.

\begin{proof}[Proof of Theorem \ref{esti-omega}]
Fix $\alpha\in\{0,\frac13,\frac23\}^n$
and $\upsilon\in A_p$.
From
Theorem \ref{thm-e}(i) and Lemma \ref{lem-apwight}(i),
we deduce that the set $\{f\in C^{\infty}(\mathbb{R}):\ |\nabla^\ell
f|\in C_{\rm c}(\mathbb{R})\}$ is
dense in $\dot{W}^{\ell,p}_{\upsilon} (\mathbb{R})$.
Applying this and
a density argument similar to that used in
the proof of \cite[(4.22)]{dlyyz-bvy}, we find that, to prove the
present theorem, it suffices to show that
\eqref{eq-est-m1} holds for any
$f\in C^{\infty}$ with $|\nabla^{\ell}f|\in C_{\rm
c}$.

Now, let $f\in C^{\infty}$ with $|\nabla^{\ell}f|\in
C_{\rm c}$.
Without loss of generality,
we may assume that $\lambda=1$; otherwise, we replace $f$ by $\frac{f}{\lambda} $
for any $\lambda\in(0,\infty)$.
To prove \eqref{eq-est-m1} for this $f$,
we consider the following four cases on $p$, $n$, and $\beta$.

\emph{Case 1)} both $p\in (1,\infty)$ and $\beta\in
\mathbb{R}\setminus \{1\}$
or both $p=1$ and $\beta\in (-\infty,0)\cup (1,\infty)$.
In this case, repeating the proof of \cite[Proposition 2.3]{lyyzz24}
with Lemma 2.5 and (2.3) therein replaced, respectively, by Lemma
\ref{lem-Qx} and \eqref{eq-bb00} here,
we obtain the desired result.

\emph{Case 2)}
$p=1$, $n\in \mathbb{N}\cap [2,\infty)$, and
$\beta\in [0,1-\frac{1}{n})$. In this case,
we fix $\sigma\in(\beta,1-\frac1n)$ and let
\begin{align*}
\mathcal{G}_{\sigma}:&=
\left\{Q\in\mathcal{D}^{\alpha}_{1,\beta,k,\ell}[f]:\
\text{for any collection}\ \mathcal{P}\subset
\mathcal{D}^{\alpha}_{1,\beta,k,\ell}[f]\right.\\
&\quad\text{of pairwise disjoint cubes strictly contained in}\
Q,\\
&\quad\left.\sum_{P\in\mathcal{P}}|P|^{\sigma-1}\upsilon(P)
\le |Q|^{\sigma-1}\upsilon(Q)\right\}
\end{align*}
and $\mathcal{B}_\sigma:=
\mathcal{D}^{\alpha}_{1,\beta,k,\ell}[f]\setminus\mathcal{G}_\sigma$.
Then, repeating the proof of \cite[Lemma 4.3]{cddd03}
with $|I|^\gamma$ therein replaced by $|Q|^{\beta-1}\upsilon(Q)$ here,
we obtain
\begin{align*}
\sum_{Q\in\mathcal{B}_\sigma}|Q|^{\beta-1}\upsilon(Q)
\lesssim\sum_{Q\in\mathcal{G}_\sigma}|Q|^{\beta-1}
\upsilon(Q),
\end{align*}
where the implicit positive constant depends only on
$n$, $\beta$, and $\sigma$.
From the Whitney inequality (see \cite[Theorem 1.15]{b09})
and \cite[p.\,46, (7.13)]{dl93},
we deduce that
\begin{align}\label{eq-de1}
E_k(f,Q)
\lesssim
\sup_{|h|\le \frac{\ell(Q)}{k}}\left\|\Delta_h^k
f\right\|_{L^1(Q(k,h))}
\lesssim
|Q|^{\frac{k-1}{n}}\sum_{\zeta\in \mathbb{Z}_{+}^n,\,|\zeta|=k-1}
\sup_{|h|\le \ell(Q)}\left\|\Delta_h
\left(\partial^{\zeta}f\right)\right\|_{L^1(Q(1,h))},
\end{align}
where $\ell(Q)$ denotes the edge length of $Q$ and,
for any $j\in\mathbb{N}$ and $h\in \mathbb{R}^n$,
$$Q(j,h):=\{x\in Q: x+jh\in Q\}.$$
Fix $\zeta\in \mathbb{Z}_{+}^n$ and $|\zeta|=k-1$.
Then, for any $h\in\mathbb{R}^n$ with
$|h|\le \ell(Q)$ and for any $\xi\in Q$ and $x\in Q(1,h)$, by the definition of
$Q(1,h)$,
we obtain $x,x+h\in Q$,
which, together with a change of variables, further implies that
\begin{align*}
\left\|	\Delta_h \left(\partial^{\zeta}f\right)\right\|_{L^1(Q(1,h))}
&=
\int_{Q(1,h)}\left|\partial^{\zeta}f(x+h)-\partial^{\zeta}f(x)\right|\,dx\\
&\le
\int_{Q(1,h)}\left|\partial^{\zeta}f(\xi)-\partial^{\zeta}f(x)\right|\,dx
+\int_{Q(1,h)}\left|\partial^{\zeta}f(\xi)-\partial^{\zeta}f(x+h)\right|\,dx\\
&\le
2\int_{Q}\left|\partial^{\zeta}f(\xi)-\partial^{\zeta}f(x)\right|\,dx.
\end{align*}
Integrating with respect to $\xi\in Q$,
we obtain
\begin{align*}
\sup_{|h|\le \ell(Q)}\left\|	\Delta_h
\left(\partial^{\zeta}f\right)\right\|_{L^1(Q(1,h))}\lesssim
|Q|^{\frac{1}{n}}\omega_Q\left(\partial^{\zeta}f\right).
\end{align*}
This, combined with $\mathcal{G}_{\sigma}
\subset\mathcal{D}^{\alpha}_{1,\beta,k,\ell}[f]$,
the definition of $\mathcal{D}^{\alpha}_{1,\beta,k,\ell}[f]$,
\eqref{eq-Ekl},
and \eqref{eq-de1},
further implies that
\begin{align}\label{eq-0027}
\sum_{Q\in\mathcal{D}^{\alpha}_{1,\beta,k,\ell}[f]}
|Q|^{\beta-1}\upsilon(Q)
&\lesssim
\sum_{Q\in\mathcal{G}_\sigma}|Q|^{\beta-1}
\upsilon(Q)
\le
\sum_{Q\in\mathcal{G}_\sigma}|Q|^{-\frac{\ell}{n}}E_{k}(f,Q)
\frac{\upsilon(Q)}{|Q|}\notag\\
&\lesssim
\sum_{Q\in\mathcal{G}_\sigma}|Q|^{-\frac{\ell}{n}}E_{\ell}(f,Q)
\frac{\upsilon(Q)}{|Q|}
\lesssim
\sum_{\zeta\in \mathbb{Z}_{+}^n,|\zeta|=\ell-1}
\sum_{Q\in\mathcal{G}_\sigma}\omega_Q\left(\partial^{\zeta}f\right)
\frac{\upsilon(Q)}{|Q|}.
\end{align}
Repeating the proof of Case 3) of
\cite[Proposition 2.3]{lyyzz24} with $f$ therein replaced by
$\partial^{\zeta}f$ here, we conclude that
\begin{align*}
\sum_{Q\in\mathcal{G}_\sigma}\omega_Q\left(\partial^{\zeta}f\right)
\frac{\upsilon(Q)}{|Q|}	
&\lesssim
[\upsilon]_{A_1}
\int_{\mathbb{R}^n}\left|\nabla \left(\partial^\zeta f\right)
(x)\right|\upsilon(x)\,dx,
\end{align*}
which, together with \eqref{eq-0027}, further implies
\eqref{eq-est-m1} with $\varphi(t):=t$ for any $t\in [0,\infty)$ in
this
case.
This finishes the proof of Theorem \ref{esti-omega}.
\end{proof}

\subsection{Proof of Theorem \ref{thm-up-ap}}\label{ssub-3.2}

In this subsection, we are devoted to
proving Theorem \ref{thm-up-ap} and the sharpness of the assumption
$n(\frac{1}{p}-\frac{1}{q})<k$ in Theorems \ref{thm-up-ap},
\ref{thm-main}, and \ref{t-chaWkX}.
We first establish the following
\emph{variant higher-order Poincar\'e inequality}.

\begin{lemma}\label{Poin}
Let $k\in \mathbb{N}$ and $f\in L^1_{\rm loc}$.
Then there exist a positive constant $C_{(n,k)}$, depending only on
$n$ and $k$, such that,
for almost every $x\in\mathbb{R}^n$ and for any $r\in(0,\infty)$
and any ball $B_1\subset B:=B(x,r)\subset 3B_1$,
\begin{align}\label{Poine0}
\left|f(x)-P^{(k-1)}_{B_1}(f)(x)\right|
\le C_{(n,k)}\sum_{j\in\mathbb{Z}_+}\fint_{2^{-j}B}
\left|f(y)-P^{(k-1)}_{2^{-j}B}(f)(x)\right|\,dy.
\end{align}
\end{lemma}

\begin{proof}
From $f\in L^1_{\rm loc}$
and
\cite[Lemma 4.1]{DS84} (see also \cite[Lemma 6.14]{ns2012}),
we infer that,
for almost every $x\in {\mathbb{R}^n}$
and for any ball $B:=B(x,r)$ with $r\in (0,\infty)$,
\begin{align*}
f(x)=\lim_{j\to \infty}P^{(k-1)}_{2^{-j}B}(f)(x).
\end{align*}
This, combined with both (i) and (ii) of Lemma \ref{uni-p},
further implies that, for almost every $x\in \mathbb{R}^n$
and for any ball $B:=B(x,r)$ with $r\in(0,\infty)$,
\begin{align}\label{Poine1}
\left|f(x)-P^{(k-1)}_B(f)(x)\right|
&=\lim_{j\to\infty}\left|P^{(k-1)}_{2^{-j}B}(f)(x)-P^{(k-1)}_B(f)(x)\right|\notag\\
&
\le\sum_{j\in\mathbb{Z}_+}\left|
P^{(k-1)}_{2^{-j-1}B}(f)(x)
-P^{(k-1)}_{2^{-j}B}(f)(x)\right|\notag\\
&=
\sum_{j\in\mathbb{Z}_+}\left|
P^{(k-1)}_{2^{-j-1}B}\left(f-P^{(k-1)}_{2^{-j}B}(f)\right)(x)\right|\notag\\
&\lesssim
\sum_{j\in\mathbb{Z}_+}\fint_{2^{-j}B}
\left|f(y)-P^{(k-1)}_{2^{-j}B}(f)(y)\right|\,dy.
\end{align}
In addition, for almost every $x\in \mathbb{R}^n$
and for any
$B:=B(x,r)$ with $r\in(0,\infty)$ and
$B_1$ satisfying $B_1\subset B\subset 3B_1$,
using both (i) and (ii) of Lemma \ref{uni-p} again,
we obtain
\begin{align*}
\left|P^{(k-1)}_B(f)(x)-P^{(k-1)}_{B_1}(f)(x)\right|
&=
\left|P^{(k-1)}_{B_1}\left(f-P^{(k-1)}_{B}(f)\right)(x)\right|\notag\\
&\lesssim
\fint_{B_1}\left|f(y)-P^{(k-1)}_{B}(f)(y)\right|\,dy\notag\\
&\sim \fint_{B}\left|f(y)-P^{(k-1)}_{B}(f)(y)\right|\,dy.
\end{align*}
Applying this and \eqref{Poine1},
we further find that \eqref{Poine0} holds, which
completes the proof of Lemma \ref{Poine0}.
\end{proof}

In what follows, for any $k,\ell\in \mathbb{N}$ with $\ell\le k$,
$i\in \mathbb{Z}_{+}\cap[0,k]$, $\beta\in \mathbb{R}$,
$\lambda\in (0,\infty)$,
and
$f\in L^1_{\rm loc}$,
let
\begin{align}\label{eq-Ei}
E^{(i)}_{\lambda,n(\beta-1),k,\ell}[f]:=\left\{(x,h)\in {\mathbb{R}^n}
\times({\mathbb{R}^n}\setminus\{ {\bf 0}\}):\
\left|\left(f-P^{(k-1)}_{B_{x,h,k}}(f)\right)(x+ih)\right|\geq
\lambda|h|^{n(\beta-1)+\ell}\right\},
\end{align}
where $B_{x,h,k}:=B(x+\frac{kh}{2},k|h|)$.

\begin{proposition}\label{pro-E0}
Let $k,\ell\in \mathbb{N}$ with $\ell\le k$, $p\in [1,\infty)$,
$q,\varepsilon,\lambda\in (0,\infty)$, $\beta\in
\mathbb{R}\setminus\{1\}$, and $\upsilon \in L^{1}_{\rm
loc}$ be nonnegative.
Then the following four statements hold.
\begin{enumerate}[{\rm (i)}]
\item If $q\in [p,\infty)$,
then there exists a positive constant $C_1$ such that, for any $f\in
L^{1}_{\rm loc}$,
\begin{align*}
&\int_{\mathbb{R}^n}\left[\int_{\mathbb{R}^n}{\bf
1}_{E^{(0)}_{\lambda,n(\beta-1),k,\ell}[f]}(x,h)|h|^{qn(\beta-1)-n}\,dh\right]^{\frac{p}{q}}\upsilon(x)\,dx\\
&\quad\le
C_1\sum_{j\in\mathbb{Z_{+}}}2^{jnp(\beta-1)}\sum_{\alpha\in\{0,\frac{1}{3},\frac{2}{3}\}^n}\sum_{Q\in
\mathcal{D}^{\alpha}_{\lambda(j),\beta,k,\ell}[f]}|Q|^{p(\beta-1)}\upsilon(Q).
\end{align*}
\item If $q\in [p,\infty)$ and $\upsilon\in A_1$, then
there exists a positive constant $C_2$ such that,
for any $i\in \mathbb{N}\cap [1,k]$ and
$f\in C^{\infty}$ with $|\nabla^{\ell}f|\in C_{\rm
c}$,
\begin{align*}
&\int_{\mathbb{R}^n}\left[\int_{\mathbb{R}^n}{\bf
1}_{E^{(i)}_{\lambda,n(\beta-1),k,\ell}[f]}(x,h)|h|^{qn(\beta-1)-n}\,dh\right]^{\frac{p}{q}}\upsilon(x)\,dx\\
&\quad\le
C_2[\upsilon]_{A_1}\sum_{j\in\mathbb{Z_{+}}}2^{jnp(\beta-1+\frac{1}{p}-\frac{1}{q})}\sum_{\alpha\in\{0,\frac{1}{3},\frac{2}{3}\}^n}\sum_{Q\in
\mathcal{D}^{\alpha}_{\lambda(j),\beta,k,\ell}[f]}|Q|^{p(\beta-1)}\upsilon(Q);
\end{align*}
\item If $q\in (0,p)$, then there exists a positive constant
$C_3$ such that, for any $f\in C^{\infty}$ with
$|\nabla^{\ell}f|\in C_{\rm c}$,
\begin{align*}
&\left\{\int_{\mathbb{R}^n}\left[\int_{\mathbb{R}^n}{\bf
1}_{E^{(0)}_{\lambda,n(\beta-1),k,\ell}[f]}(x,h)|h|^{qn(\beta-1)-n}\,dh\right]^{\frac{p}{q}}\upsilon(x)\,dx\right\}^{\frac{q}{p}}\\
&\quad\le
C_3
\sum_{j\in\mathbb{Z_{+}}}2^{jnq(\beta-1)}\sum_{\alpha\in\{0,\frac{1}{3},\frac{2}{3}\}^n}
\left[\sum_{Q\in
\mathcal{D}^{\alpha}_{\lambda(j),\beta,k,\ell}[f]}|Q|^{p(\beta-1)}\upsilon(Q)\right]^{\frac{q}{p}}.
\end{align*}
\item If $q\in (0,p)$ and $\upsilon\in A_p$, then there
exists a positive constant
$C_4$ such that, for any $f\in C^{\infty}$ with
$|\nabla^{\ell}f|\in C_{\rm c}$,
\begin{align*}
&\left\{\int_{\mathbb{R}^n}\left[\int_{\mathbb{R}^n}{\bf
1}_{E^{(i)}_{\lambda,n(\beta-1),k,\ell}[f]}(x,h)|h|^{qn(\beta-1)-n}\,dh\right]^{\frac{p}{q}}\upsilon(x)\,dx\right\}^{\frac{q}{p}}\\
&\quad\le
C_4 [\upsilon]^{\frac{q}{p}}_{A_p}
\sum_{j\in\mathbb{Z_{+}}}2^{jnq(\beta-1)}\sum_{\alpha\in\{0,\frac{1}{3},\frac{2}{3}\}^n}
\left[\sum_{Q\in
\mathcal{D}^{\alpha}_{\lambda(j),\beta,k,\ell}[f]}|Q|^{p(\beta-1)}\upsilon(Q)\right]^{\frac{q}{p}},
\end{align*}
\end{enumerate}	
where $C_1, C_2,C_3$, and $C_4$ depend only on $n,k,\ell,\beta$, and
$q$, and, for any $j\in \mathbb{Z_+}$,
$$\lambda(j):=c(n,k,\ell,\beta,\varepsilon)\lambda
2^{j[\ell+n(\beta-1)-\varepsilon]}$$ with a positive constant
$c(n,k,\ell,\beta,\varepsilon)$ depending only on $n,k,\ell,\beta,$
and $\varepsilon$.
\end{proposition}

\begin{proof}
Let $f\in L^1_{\rm loc}$
and fix $i\in \mathbb{Z}_{+}\cap [0,k]$.	We first estimate
\begin{align*}
\int_{\mathbb{R}^n}{\bf
1}_{E^{(i)}_{\lambda,n(\beta-1),k,\ell}[f]}(\cdot,h)|h|^{qn(\beta-1)-n}\,dh.
\end{align*}
To this end, for almost every $(x,h)\in
E^{(i)}_{\lambda,n(\beta-1),k,\ell}[f]$,
using Lemma \ref{Poin} with $B_1:=B_{x,h,k}:=B(x+\frac{kh}{2},k|h|)$
and
$B:=B(x+ih,2k|h|)$, we find that
\begin{align}\label{pointe2}
\lambda
|h|^{\ell+n(\beta-1)}
&<\left|f(x+ih)-P^{(k-1)}_{B_{x,h,k}}(f)(x+ih)\right|\notag\\
&\le C_{(n,k)}\sum_{j\in\mathbb{Z}_+}
\fint_{B(x+ih,2^{-j+1}k|h|)}
\left|f(z)-P^{(k-1)}_{B(x+ih,2^{-j+1}k|h|)}(f)(z)\right|\,dz,
\end{align}
where $C_{(n,k)}$ is the same constant as in Lemma \ref{Poin}.
In what follows, fix $\varepsilon\in(0,\infty)$
and let $c_0:=\frac{1-2^{-\varepsilon}}{ C_{(n,k)}}$.
We now claim that, for almost every $(x,h)\in
E^{(i)}_{\lambda,n(\beta-1),k,\ell}[f]$,
there exists $j_{x,h}\in\mathbb{Z}_+$ such that
\begin{align}\label{eq-claim-02}
c_0\lambda 2^{-j_{x,h}\varepsilon}|h|^{\ell+n(\beta-1)}
< \fint_{B(x+ih,2^{-j_{x,h}+1}k|h|)}
\left|f(z)-P^{(k-1)}_{B(x+ih,2^{-j_{x,h}+1}k|h|)}(f)(z)\right|\,dz.
\end{align}
Otherwise, it holds that
\begin{align*}
\frac{\lambda}{ C_{(n,k)}}|h|^{\ell+n(\beta-1)}
&=c_0\lambda|h|^{\ell+n(\beta-1)}
\sum_{j\in\mathbb{Z}_+}2^{-j\varepsilon}\\
&\ge
\sum_{j\in\mathbb{Z}_+}\fint_{B(x+ih,2^{-j+1}k|h|)}
\left|f(z)-P^{(k-1)}_{B(x+ih,2^{-j+1}k|h|)}(f)(z)\right|\,dz,
\end{align*}
which contradicts \eqref{pointe2}.
In addition, from Lemma \ref{lem-dy-cub},
we infer that there exists a positive constant $\widetilde{C}_{(n)}$,
depending only on $n$,
such that, for any $x,h\in\mathbb{R}^n$,
there exist $\alpha_{x,h}\in \{0,\frac13,\frac23\}^n$
and $Q_{x,h}\in \mathcal{D}^{\alpha_{x,h}}$ satisfying that
\begin{align*}
B\left(x,2^{-j_{x,h}+1}k|h|\right)\subset Q_{x,h}\subset
B\left(x,2^{-j_{x,h}+1}\widetilde{C}_{(n)}k|h|\right).
\end{align*}
Applying this, \eqref{eq-claim-02}, and \eqref{uni-sim},
we conclude that, for almost every $(x,h)\in
E^{(i)}_{\lambda,n(\beta-1),k,\ell}[f]$,
\begin{align}\label{eq-key}
&\lambda 2^{j_{x,h}[\ell+n(\beta-1)-\varepsilon]}
\left|Q_{x,h}\right|^{\frac{1}{n}(\ell+n(\beta-1))}\notag\\
&\quad=
\lambda 2^{-j_{x,h}\varepsilon}
\left|2^{j_{x,h}}Q_{x,h}\right|^{\frac{1}{n}[\ell+n(\beta-1)]}
\sim
\lambda 2^{-j_{x,h}\varepsilon}|h|^{\ell+n(\beta-1)}\notag\\
&\quad\lesssim\fint_{B(x+ih,2^{-j_{x,h}+1}k|h|)}
\left|f(z)-P^{(k-1)}_{B(x+ih,2^{-j_{x,h}+1}k|h|)}(f)(z)\right|\,dz\notag\\
&\quad\lesssim\inf_{P\in \mathcal{P}_{k-1}}
\fint_{B(x+ih,2^{-j_{x,h}+1}k|h|)}
\left|f(z)-P(z)\right|\,dz\notag\\
&\quad \sim
\inf_{P\in \mathcal{P}_{k-1}}
\fint_{Q_{x,h}}
\left|f(z)-P(z)\right|\,dz
\le
\fint_{Q_{x,h}}
\left|f(z)-P^{(k-1)}_{Q_{x,h}}(z)\right|\,dz,
\end{align}
which further implies that there exists a positive constant
$c_{(n,k,\ell,\beta,\varepsilon)}$,
depending only on $n$, $k$, $\ell$, $\beta$, and $\varepsilon$,
such that $Q_{x,h}\in
\mathcal{D}^{\alpha_{x,h}}_{\lambda(j_{x,h}),\beta,k,\ell}[f]$,
where $j_{x,h}\in \mathbb{Z}_{+}$
depends only on $x,h$,
$\lambda(j_{x,h}):=c(n,k,\beta,\ell,\varepsilon)\lambda
2^{j_{x,h}[\ell+n(\beta-1)-\varepsilon]}$,
and $\mathcal{D}^{\alpha_{x,h}}_{\lambda(j_{x,h}),\beta,k,\ell}[f]$ is
defined as in
\eqref{df-ab} with $\alpha$ and $\lambda$ replaced,
respectively, by $\alpha_{x,h}$ and $\lambda(j_{x,h})$.

If $i=0$,
then, for almost every $(x,h)\in
E^{(0)}_{\lambda,n(\beta-1),k,\ell}[f]$,
\begin{align*}
(x,h)\in\left[2^{-j_{x,h}}B(x,2k|h|)\right]
\times \left[B(x,2k|h|)-\{x\}\right]\subset
Q_{x,h}\times \left[2^{j_{x,h}}Q_{x,h}-\{x\}\right]
\end{align*}
and
\begin{align*}
|h|\sim\left|2^{j_{x,h}}Q_{x,h}\right|^{\frac1n}.
\end{align*}
By these, a change of variables, and \eqref{eq-key}, we find that, for
almost every $x\in \mathbb{R}^n$,
\begin{align}\label{pointe6}
&\int_{\mathbb{R}^n}{\bf
1}_{E^{(0)}_{\lambda,n(\beta-1),k,\ell}[f]}(x,h)|h|^{qn(\beta-1)-n}\,dh\notag\\
&\quad\lesssim\sum_{j\in\mathbb{Z}_+}
\sum_{\alpha\in\{0,\frac13,\frac23\}^n}
\sum_{Q\in\mathcal{D}^{\alpha}_{\lambda(j),\beta,k,\ell}[f]}
\int_{2^jQ}\left|2^jQ\right|^{q(\beta-1)-1}\,dy{\bf 1}_{Q}(x)\notag\\
&\quad=\sum_{j\in\mathbb{Z}_+}
\sum_{\alpha\in\{0,\frac13,\frac23\}^n}
\sum_{Q\in\mathcal{D}^{\alpha}_{\lambda(j),\beta,k,\ell}[f]}
\left|2^jQ\right|^{q(\beta-1)}{\bf 1}_{Q}(x).
\end{align}

If $i\in \mathbb{N}\cap [1,k]$,
then, for almost every $(x,h)\in
E^{(i)}_{\lambda,n(\beta-1),k,\ell}[f]$,
\begin{align*}
(x,h)\in B(x+ih,2k|h|)
\times i^{-1}\left[2^{-j_{x,h}}B(x+ih,2k|h|)-\{x\}\right]
\subset
2^{j_{x,h}}Q_{x,h}\times \left[Q_{x,h}-\{x\}\right]
\end{align*}
and $|h|\sim|2^{j_{x,h}}Q_{x,h}|^{\frac1n}$.
Using these, a change of variables, and \eqref{eq-key}, we obtain,
for almost every $x\in \mathbb{R}^n$,
\begin{align}\label{eq-ii3}
&\int_{\mathbb{R}^n}{\bf
1}_{E^{(i)}_{\lambda,n(\beta-1),k,\ell}[f]}(x,h)|h|^{qn(\beta-1)-n}\,dh\notag\\
&\quad\lesssim\sum_{j\in\mathbb{Z}_+}
\sum_{\alpha\in\{0,\frac13,\frac23\}^n}
\sum_{Q\in\mathcal{D}^{\alpha}_{\lambda(j),\beta,k,\ell}[f]}
\int_{Q}\left|2^j Q\right|^{q(\beta-1)-1}\,dy{\bf 1}_{2^jQ}(x)\notag\\
&\quad=\sum_{j\in\mathbb{Z}_+}
\sum_{\alpha\in\{0,\frac13,\frac23\}^n}
\sum_{Q\in\mathcal{D}^{\alpha}_{\lambda(j),\beta,k,\ell}[f]}
\left|2^jQ\right|^{q(\beta-1)-1}|Q|{\bf 1}_{2^j Q}(x).
\end{align}

Now, we show (i). Assume that $q\in [p,\infty)$.
From \eqref{pointe6} and the Tonelli theorem,
we deduce that
\begin{align*}
&\int_{\mathbb{R}^n}\left[\int_{\mathbb{R}^n}{\bf
1}_{E^{(0)}_{\lambda,n(\beta-1),k,\ell}[f]}(x,h)|h|^{qn(\beta-1)-n}\,dh\right]^{\frac{p}{q}}\upsilon(x)\,dx\\
&\quad\lesssim
\int_{\mathbb{R}^n}\left\{\sum_{j\in\mathbb{Z}_+}
\sum_{\alpha\in\{0,\frac13,\frac23\}^n}
\sum_{Q\in\mathcal{D}^{\alpha}_{\lambda(j),\beta,k,\ell}[f]}
\left|2^jQ\right|^{q(\beta-1)}{\bf
1}_{Q}(x)\right\}^{\frac{p}{q}}\upsilon(x)\,dx\\
&\quad
\le
\sum_{j\in\mathbb{Z_{+}}}2^{jnp(\beta-1)}\sum_{\alpha\in\{0,\frac{1}{3},\frac{2}{3}\}^n}\sum_{Q\in
\mathcal{D}^{\alpha}_{\lambda(j),\beta,k,\ell}[f]}|Q|^{p(\beta-1)}\upsilon(Q).
\end{align*}
This proves (i).

Next, we show (ii). Assume that $q\in [p,\infty)$ and $\upsilon \in
A_1$ and fix $i\in \mathbb{N}\cap [1,k]$. Using these,
\eqref{eq-ii3}, the Tonelli theorem, the assumption $\upsilon \in
A_1$, and
Lemma \ref{lem-apwight}(ii),
we find that
\begin{align*}
&
\int_{\mathbb{R}^n}\left[\int_{\mathbb{R}^n}{\bf
1}_{E^{(i)}_{\lambda,n(\beta-1),k,\ell}[f]}(x,h)|h|^{qn(\beta-1)-n}\,dh\right]^{\frac{p}{q}}\upsilon(x)\,dx\\
&\quad\lesssim
\int_{\mathbb{R}^n}\left\{\sum_{j\in\mathbb{Z}_+}
\sum_{\alpha\in\{0,\frac13,\frac23\}^n}
\sum_{Q\in\mathcal{D}^{\alpha}_{\lambda(j),\beta,k,\ell}[f]}
\left|2^jQ\right|^{q(\beta-1)-1}|Q|{\bf 1}_{2^j
Q}(x)\right\}^{\frac{p}{q}}\upsilon(x)\,dx\\
&\quad\le
\sum_{j\in\mathbb{Z}_+}
\sum_{\alpha\in\{0,\frac13,\frac23\}^n}
\sum_{Q\in\mathcal{D}^{\alpha}_{\lambda(j),\beta,k,\ell}[f]}
\left|2^jQ\right|^{p(\beta-1)-\frac{p}{q}}|Q|^{\frac{p}{q}}\upsilon\left(2^j
Q\right)\\
&\quad
\le
[\upsilon]_{A_1}\sum_{j\in\mathbb{Z_{+}}}2^{jnp(\beta-1+\frac{1}{p}-\frac{1}{q})}\sum_{\alpha\in\{0,\frac{1}{3},\frac{2}{3}\}^n}\sum_{Q\in
\mathcal{D}^{\alpha}_{\lambda(j),\beta,k,\ell}[f]}|Q|^{p(\beta-1)}\upsilon(Q).
\end{align*}
This proves (ii).

Next, we show (iii). Assume that $q\in (0,p)$.
By this, \eqref{pointe6}, and the Minkowski inequality, we obtain
\begin{align}\label{eq-0056}
&\left\{\int_{\mathbb{R}^n}\left[\int_{\mathbb{R}^n}{\bf
1}_{E^{(0)}_{\lambda,n(\beta-1),k,\ell}
[f]}(x,h)|h|^{qn(\beta-1)-n}\,dh\right]^{\frac{p}{q}}\upsilon(x)\,dx\right\}^\frac{q}{p}\notag\\
&\quad \lesssim
\sum_{j\in\mathbb{Z}_+}
\sum_{\alpha\in\{0,\frac13,\frac23\}^n}
\left\{\int_{\mathbb{R}^n}\left[\sum_{Q\in\mathcal{D}^{\alpha}_{\lambda(j),\beta,k,\ell}[f]}
\left|2^jQ\right|^{q(\beta-1)}{\bf
1}_{Q}(x)\right]^{\frac{p}{q}}\upsilon(x)\,dx\right\}^{\frac{q}{p}}.
\end{align}
From this and Lemma \ref{lem-Qx} with $p:=q$ and
$\lambda:=\lambda(j)$,
we deduce that, for any $\alpha\in \{0,\frac{1}{3},\frac{2}{3}\}^n$,
$j\in \mathbb{Z}_{+}$,
and $x\in \mathbb{R}^n$,
\begin{align}\label{e-QD}
\left[\sum_{Q\in\mathcal{D}^{\alpha}_{\lambda(j),\beta,k,\ell}[f]}
\left|2^jQ\right|^{q(\beta-1)}{\bf 1}_{Q}(x)\right]^{\frac{p}{q}}
&\lesssim
2^{jnp(\beta-1)}|Q_x|^{p(\beta-1)}\notag\\
&\le2^{jnp(\beta-1)}
\sum_{Q\in\mathcal{D}^{\alpha}_{\lambda(j),\beta,k,\ell}[f]}
\left|Q\right|^{p(\beta-1)}{\bf 1}_{Q}(x),
\end{align}
where $Q_x$ is the same as in Lemma \ref{lem-Qx}.
This, together with \eqref{eq-0056}, further implies that (iii) holds.

Finally, we show (iv). Assume that $q\in (0,p)$ and $\upsilon\in
A_p$ and fix $i\in \mathbb{N}\cap [1,k]$.
From this, \eqref{eq-ii3}, the Minkowski inequality,
a change of variables, \eqref{e-QD}, and
Lemma \ref{lem-apwight}(ii), it follows that
\begin{align*}
&\left\{\int_{\mathbb{R}^n}\left[\int_{\mathbb{R}^n}{\bf
1}_{E^{(i)}_{\lambda,n(\beta-1),k,\ell}[f]}(x,h)|h|^{qn(\beta-1)-n}\,dh\right]^{\frac{p}{q}}\upsilon(x)\,dx\right\}^\frac{q}{p}\\
& \quad \lesssim
\sum_{j\in\mathbb{Z}_+}
\sum_{\alpha\in\{0,\frac13,\frac23\}^n}
\left\{\int_{\mathbb{R}^n}\left[\sum_{Q\in\mathcal{D}^{\alpha}_{\lambda(j),\beta,k,\ell}[f]}
\left|2^jQ\right|^{q(\beta-1)-1}|Q|{\bf
1}_{2^jQ}(x)\right]^{\frac{p}{q}}\upsilon(x)\,dx\right\}^{\frac{q}{p}}\notag\\
&\quad=
\sum_{j\in\mathbb{Z}_+}2^{jn[q(\beta-1)-1]}
\sum_{\alpha\in\{0,\frac13,\frac23\}^n}
\left\{\int_{\mathbb{R}^n}\left[\sum_{Q\in\mathcal{D}^{\alpha}_{\lambda(j),\beta,k,\ell}[f]}
\left|Q\right|^{q(\beta-1)}{\bf
1}_{Q}(y)\right]^{\frac{p}{q}}\upsilon\left(2^{j}y\right)2^{jn}\,dy\right\}^{\frac{q}{p}}\notag\\
&\quad\lesssim
\sum_{j\in\mathbb{Z}_+}2^{jn[q(\beta-1)-1]}
\sum_{\alpha\in\{0,\frac13,\frac23\}^n}
\left\{\sum_{Q\in\mathcal{D}^{\alpha}_{\lambda(j),\beta,k,\ell}[f]}
\left|Q\right|^{p(\beta-1)}\upsilon\left(2^{j}Q\right)\right\}^{\frac{q}{p}}\\
&\quad\lesssim [\upsilon]^{\frac{q}{p}}_{A_p}
\sum_{j\in\mathbb{Z_{+}}}2^{jnq(\beta-1)}\sum_{\alpha\in\{0,\frac{1}{3},\frac{2}{3}\}^n}
\left[\sum_{Q\in
\mathcal{D}^{\alpha}_{\lambda(j),\beta,k,\ell}[f]}|Q|^{p(\beta-1)}\upsilon(Q)\right]^{\frac{q}{p}}\notag.
\end{align*}
This finishes the proof of
(iv)
and hence
Proposition \ref{pro-E0}.
\end{proof}

Next, we are ready to show Theorem \ref{thm-up-ap}.

\begin{proof}[Proof of Theorem \ref{thm-up-ap}]
Similarly to the density argument used in
the proof of Theorem \ref{esti-omega},
to prove the present theorem, it suffices to show
that
\eqref{e-up-ap} holds for any
$f\in C^{\infty}$ with $|\nabla^{\ell}f|\in C_{\rm
c}$.
Let $f\in C^{\infty}$ with $|\nabla^{\ell}f|\in C_{\rm
c}$ and $\beta:=1+\frac{\gamma}{qn}$ and choose
$\varepsilon\in (0,\ell-n(\frac{1}{p}-\frac{1}{q}))$. We only
consider the case $q\in [p,\infty)$ because the case $q\in (0,p)$ is
quite similar and hence we omit the details here.
Notice that, if $P\in \mathcal{P}_{k-1}$,
then $\Delta^k_h P(x)=0$ for any $x,h\in{\mathbb{R}^n}$.
Using this and \eqref{eq-sho-hd}, we find that,
for any $x,h\in {\mathbb{R}^n}$,
\begin{align}\label{eq-dff}
\left|\Delta_{h}^k f(x)\right|=
\left|\Delta_{h}^k \left(f-P^{(k-1)}_{B_{x,h,k}}(f)\right)(x)\right|
\leq\sum_{i=0}^{k}2^k
\left|\left(f-P^{(k-1)}_{B_{x,h,k}}(f)\right)(x+ih)\right|.
\end{align}
Using this, \eqref{eq-E}, \eqref{eq-Ei},
and $\beta=1+\frac{\gamma}{qn}$, we obtain
\begin{align}\label{e-k-I}
&\int_{\mathbb{R}^n}
\left[\int_{\mathbb{R}^n}{\bf 1}_{E_{\lambda,
\frac{\gamma}{q},k,\ell}[f]}(x,h)|h|^{\gamma-n}\,dh
\right]^{\frac{p}{q}}\upsilon(x)\,dx\notag\\
& \quad \lesssim
\sum_{i=0}^{k}\int_{\mathbb{R}^n}
\left[\int_{\mathbb{R}^n}{\bf 1}_{E^{(i)}_{2^{-k}\lambda,
n(\beta-1),k,\ell}[f]}(x,h)|h|^{qn(\beta-1)-n}\,dh
\right]^{\frac{p}{q}}\upsilon(x)\,dx =:
\sum_{i=0}^{k} {\rm I}_{i},
\end{align}
where $E^{(i)}_{2^{-k}\lambda,
n(\beta-1),k,\ell}[f]$ is the same as in \eqref{eq-Ei} with
$\lambda$ replaced by $2^{-k}\lambda$.

We first deal with ${\rm I}_{0}$.
From Proposition \ref{pro-E0}(i), Theorem \ref{esti-omega}
with the fact that $\varphi$ is increasing on $[0,\infty)$, the
definition of $\lambda (j)$, Lemma \ref{lem-apwight}(iv),
and the assumption $\varepsilon\in (0,\ell)$,
we infer that
\begin{align}\label{es-I0}
{\rm I}_0 &\lesssim
\sum_{j\in\mathbb{Z_{+}}}2^{jnp(\beta-1)}\sum_{\alpha\in\{0,\frac{1}{3},\frac{2}{3}\}^n}\sum_{Q\in
\mathcal{D}^{\alpha}_{2^{-k}\lambda(j),\beta,k,\ell}[f]}|Q|^{p(\beta-1)}\upsilon(Q)\notag\\
&\lesssim
\varphi\left([\upsilon]_{A_p}\right)\lambda^{-p}
\sum_{j\in\mathbb{Z_{+}}}2^{jnp(\beta-1)}2^{-jp[\ell+n(\beta-1)-\varepsilon]}
\|f\|^p_{\dot{W}^{\ell,p}_{\upsilon}}\notag\\
&\leq
\varphi\left([\upsilon]_{A_1}\right)\lambda^{-p}
\sum_{j\in\mathbb{Z_{+}}}2^{-jp(\ell-\varepsilon)}
\|f\|^p_{\dot{W}^{\ell,p}_{\upsilon}}\sim
\varphi\left([\upsilon]_{A_1}\right)\lambda^{-p}
\|f\|^p_{\dot{W}^{\ell,p}_{\upsilon}}.
\end{align}
This then finishes the estimation of ${\rm I}_0$.

Next, we fix $i\in \mathbb{Z}_+ \cap [1,k]$ and estimate ${\rm I}_i$.
By Proposition \ref{pro-E0}(ii), Theorem \ref{esti-omega}
with the fact that $\varphi$ is increasing on $[0,\infty)$, the
definition of $\lambda (j)$,
Lemma \ref{lem-apwight}(iv),
and the assumption $\varepsilon\in
(0,\ell-n(\frac{1}{p}-\frac{1}{q}))$,
\begin{align*}
{\rm I}_i
&\lesssim [\upsilon]_{A_1}
\sum_{j\in\mathbb{Z_{+}}}2^{jnp(\beta-1+\frac{1}{p}-\frac{1}{q})}\sum_{\alpha\in\{0,\frac{1}{3},\frac{2}{3}\}^n}\sum_{Q\in
\mathcal{D}^{\alpha}_{2^{-k}\lambda(j),\beta,k,\ell}[f]}|Q|^{p(\beta-1)}\upsilon(Q)\\
&\lesssim
\frac{[\upsilon]_{A_1}
\varphi([\upsilon]_{A_p})}{\lambda^{p}}
\sum_{j\in\mathbb{Z_{+}}}2^{jnp(\beta-1+\frac{1}{p}-\frac{1}{q})}2^{-jp[\ell+n(\beta-1)-\varepsilon]}
\|f\|^p_{\dot{W}^{\ell,p}_{\upsilon}}\\
&\le
\frac{[\upsilon]_{A_1}
\varphi([\upsilon]_{A_1})}{\lambda^{p}}
\sum_{j\in\mathbb{Z_{+}}}2^{-jp[\ell-n(\frac{1}{p}-\frac{1}{q})-\varepsilon]}
\|f\|^p_{\dot{W}^{\ell,p}_{\upsilon}}\\
&\sim
\frac{[\upsilon]_{A_1}
\varphi([\upsilon]_{A_1})}{\lambda^{p}}
\|f\|^p_{\dot{W}^{\ell,p}_{\upsilon}},
\end{align*}
which completes the estimation of ${\rm I}_i$.
Combining this, \eqref{es-I0}, and \eqref{e-k-I}, we find that
\eqref{e-up-ap}
holds with $\psi(t):=(1+t)\varphi(t)$ for any $t\in [0,\infty)$.
This finishes the proof of Theorem \ref{thm-up-ap}.
\end{proof}

At the end of this subsection, we show that the assumption
$n(\frac{1}{p}-\frac{1}{q})<\ell$ is sharp in Theorems \ref{thm-up-ap},
\ref{thm-main}, and \ref{t-chaWkX}.

\begin{proposition}\label{pro-sharp}
Let $k,\ell\in \mathbb{N}$ with $\ell\le k$,
$p\in [1,\infty)$, $q\in (0,\infty)$
satisfy $n(\frac{1}{p}-\frac{1}{q})\ge \ell$,
and $\gamma= -\ell q$.
Then
there exists $f\in C_{\rm c}^{\infty}$ such that
\begin{align}\label{e-sharp}
\sup_{\lambda\in (0,\infty)}
\lambda^p
\int_{\mathbb{R}^n}\left[
\int_{{\mathbb{R}^n}}{\bf 1}_{E_{\lambda,\frac{\gamma}{q},k,\ell}[f]}(x,h)
|h|^{\gamma-n}\, dh\right]^{\frac{p}{q}}
\,dx
=\infty.
\end{align}
\end{proposition}
\begin{proof}
For any $f\in \mathscr{M}$
and $x,y\in \mathbb{R}^n$, let
\begin{align}\label{e-sym}
\Delta^k_{x,y}f
:=
\sum_{j=0}^{k}(-1)^{k-j}\binom{k}{j}f\left(\frac{(k-j)x+jy}{k}\right).
\end{align}
By a change of variables, we find that, for any $f\in
\mathscr{M}$,
\begin{align}\label{e-shp-1}
&\sup_{\lambda\in(0,\infty)}
\lambda^p \int_{\mathbb{R}^n}
\left[\int_{\mathbb{R}^n}{\bf 1}_{E_{\lambda,
\frac{\gamma}{q},k,\ell}[f]}(x,h)|h|^{\gamma-n}\,dh
\right]^{\frac{p}{q}}
\,dx\notag\\
&\quad\sim
\sup_{\lambda\in(0,\infty)}
\lambda^p \int_{\mathbb{R}^n}
\left[\int_{\mathbb{R}^n}{\bf 1}_{E^{\star}_{\lambda,
\frac{\gamma}{q},k,\ell}[f]}(x,y)|x-y|^{\gamma-n}
\,dy\right]^{\frac{p}{q}}\,dx,
\end{align}
where
\begin{align*}
E^{\star}_{\lambda,\frac{\gamma}{q},k,\ell}[f]
:= \left\{(x,y)\in \mathbb{R}^n\times \mathbb{R}^n:\
\left|\Delta^k_{x,y}f\right|>\lambda
\left|y-x\right|^{\ell+\frac{\gamma}{q}}\right\}.
\end{align*}
Choose some nonnegative and radially decreasing function
$\eta \in C^{\infty}_{\rm c}$
satisfying both
${\rm supp}\,(\eta)\subset B({\bf 0},1)$
and $\int_{\mathbb{R}^n}\eta(x)\,dx =1$.
Let $\eta_2 := 2^n \eta(2\cdot)$ and $f:= \eta_2 \ast {\bf 1}_{B({\bf
0},1)}$.
Then $f\in C^{\infty}_{\rm c}$ and
${\bf 1}_{B({\bf 0},\frac{1}{2})}\le f\le {\bf 1}_{B({\bf
0},\frac{3}{2})}$.
Notice that, for any $x\in B({\bf 0}, 2k)^{\complement}$ and
$y\in B({\bf 0},\frac{1}{2})$,
$f(y)=1$ and $f(x+\frac{i}{k}(y-x))=0$ for any $i\in
\mathbb{Z}_{+}\cap [0,k-1]$, and hence $\Delta^k_{x,y}f=1$.
Since $\gamma=-\ell q$, it follows that, for any $\lambda\in (0,1)$,
\begin{align*}
\left[B({\bf 0}, 2k)^{\complement}
\times
B\left({\bf 0},\frac{1}{2}\right)\right]
\subset E^{\star}_{\lambda,\frac{\gamma}{q},k,\ell}[f].
\end{align*}
Using this, \eqref{e-shp-1}, and the assumption
$n(\frac{1}{p}-\frac{1}{q})\ge \ell$ and $\gamma=-\ell q$,
we conclude that
\begin{align*}
&\sup_{\lambda\in (0,\infty)}
\lambda^p
\int_{\mathbb{R}^n}\left[
\int_{{\mathbb{R}^n}}{\bf 1}_{E_{\lambda,\frac{\gamma}{q},k,\ell}[f]}(x,h)
|h|^{\gamma-n}\, dh\right]^{\frac{p}{q}}
\,dx\\
& \quad \gtrsim
\sup_{\lambda\in (\frac{1}{2},1)}
\lambda^p \int_{B({\bf 0},2k)^{\complement}}
\left[\int_{B({\bf 0},\frac{1}{2})}|x-y|^{\gamma-n}
\,dy\right]^{\frac{p}{q}}\,dx\\
& \quad \sim
\int_{B({\bf 0},2k)^{\complement}}
|x|^{\frac{p(\gamma-n)}{q}}\,dx
\sim
\int_{2k}^{\infty}r^{-p[\ell-n(\frac{1}{p}-\frac{1}{q})]-1}\,dr=\infty.
\end{align*}	
This implies
\eqref{e-sharp} and hence finishes the proof of Proposition
\ref{pro-sharp}.
\end{proof}

\subsection{Proof of Corollary \ref{cor-ap}}\label{ssub-3.3}

We turn to show Corollary \ref{cor-ap} in this subsection.
We first establish the following auxiliary
estimate, which plays an essential role in the proof of Corollary
\ref{cor-ap}.

\begin{lemma}\label{lem-lo}
Let $k\in \mathbb{N}$, $p\in [1,\infty)$, $q\in (0,\infty)$, $\beta\in
\mathbb{R}\setminus\{1\}$,
and
$\upsilon\in L^{1}_{\rm loc}(\mathbb{R})$
be nonnegative.
Then there exists a positive constant $C$
such that, for any $f\in L^{1}_{\rm loc}(\mathbb{R})$,
\begin{align}\label{e-x-r}
&\sup_{\lambda\in (0,\infty)}\lambda^p \int_{\mathbb{R}}
\left[\int_{0}^{\infty}{\bf 1}_{E^{\spadesuit}_{\lambda,
\beta,k}[f]}(x,r)r^{q(\beta-1)-1}\,dr
\right]^{\frac{p}{q}}\upsilon(x)\,dx\notag\\
&\quad
\le C
\sup_{\lambda\in (0,\infty)} \lambda^p \sum_{\alpha\in
\{0,\frac{1}{3},
\frac{2}{3}\}}\sum_{Q\in\mathcal{D}^{\alpha}_{\lambda,\beta,k}[f]}
|Q|^{p(\beta-1)}\upsilon(Q),
\end{align}
where
\begin{align*}
E^{\spadesuit}_{\lambda,
\beta,k}[f]:
=\left\{(x,r)\in \mathbb{R}\times (0,\infty):\
\int_{x-r}^{x+r}\int_{x-r}^{x+r}
\left|\Delta^k_{y,z}f\right|\,dy\,dz>\lambda r^{k+\beta+1}\right\}.
\end{align*}
\end{lemma}
\begin{proof}

Fix
$f\in L^{1}_{\rm loc}(\mathbb{R})$.
By \eqref{e-sym} and a change of variables, we find that, for any open
interval $I\subset \mathbb{R}$,
\begin{align}\label{e-BEk}
\int_{I}\int_I \left|\Delta^k_{x,y}f\right|\,dx\,dy
&=
\int_{I}\int_I
\left|\Delta^k_{x,y}\left[f-P_I^{(k-1)}(f)\right]\right|\,dx\,dy\notag\\
&=
\int_{I}\int_I
\left|\sum_{j=0}^{k}(-1)^{k-j}\binom{k}{j}\left[f-P_I^{(k-1)}(f)\right]\left(\frac{(k-j)x+jy}{k}\right)\right|\,dx\,dy\notag\\
&\lesssim
\int_{I}\int_I \left|\left[f-P_I^{(k-1)}(f)\right](z)\right|\,dz\,dy
=
|I|E_k(f,I).
\end{align}
Moreover,
from Lemma \ref{lem-dy-cub}, we deduce that, for any
$(x,r)\in E^{\spadesuit}_{\lambda,
\beta,k}[f]$, there exist
$C_0 \in (1,\infty)$,
$\alpha_{x,r}\in \{0,\frac{1}{3},
\frac{2}{3}\}$,
and $Q_{x,r}\in \mathcal{D}^{\alpha_{x,r}}$
such that $(x-r,x+r)\subset Q_{x,r}
\subset (x-C_0 r, x+C_0 r)$. This, combined with \eqref{e-BEk} and
\eqref{uni-sim}, further implies that, for any
$(x,r)\in E^{\spadesuit}_{\lambda,\beta,k}[f]$,
\begin{align*}
\lambda|Q_{x,r}|^{\beta+k}
&
\sim
\lambda r^{\beta+k}
\le
r^{-1}\int_{x-r}^{x+r}\int_{x-r}^{x+r}
\left|\Delta^k_{y,z}f\right|\,dy\,dz\\
&\lesssim
E_k\left(f,(x-r,x+r)\right)
\sim
\inf_{P\in \mathcal{P}_{k-1} (\mathbb{R})}
\left\|f-P\right\|_{L^1 ((x-r,x+r))}\\
&\le
\inf_{P\in \mathcal{P}_{k-1} (\mathbb{R})}
\left\|f-P\right\|_{L^1 (Q_{x,r})}
\sim
E_k\left(f,Q_{x,r}\right).
\end{align*}
By this and \eqref{df-ab}, we conclude that,
for any
$(x,r)\in E^{\spadesuit}_{\lambda,\beta,k}[f]$,
there exist a positive constant $C_{(k,\beta)}$,
depending only on both $k$ and $\beta$, $\alpha_{x,r}\in
\{0,\frac{1}{3}, \frac{2}{3}\}$, and $Q_{x,r}\in
\mathcal{D}_{C_{(k,\beta)}\lambda,\beta,k}^{\alpha}[f]$
such that
\begin{align*}
(x,r)\in Q\times \left[\frac{|Q_{x,r}|}{2C_0
},\frac{|Q_{x,r}|}{2}\right].
\end{align*}
From this, it follows that
\begin{align*}
E^{\spadesuit}_{\lambda,\beta,k}[f]\subset
\bigcup_{\alpha\in \{0,\frac{1}{3}, \frac{2}{3}\}}
\bigcup_{Q\in \mathcal{D}_{C_{(k,\beta)}\lambda,\beta,k}^{\alpha}[f]}
\left\{Q\times \left[\frac{|Q|}{2C_0},\frac{|Q|}{2}\right]\right\},
\end{align*}
which further implies that, for any $x\in \mathbb{R}$,
\begin{align}\label{e-apctr01}
\int_{0}^{\infty}{\bf 1}_{E^{\spadesuit}_{\lambda,
\beta,k}[f]}(x,r)r^{q(\beta-1)-1}\,dr
\lesssim
\sum_{\alpha\in \{0,\frac{1}{3}, \frac{2}{3}\}}
\sum_{Q\in \mathcal{D}_{C_{(k,\beta)}\lambda,\beta,k}^{\alpha}[f]}
|Q|^{q(\beta-1)}{\bf 1}_{Q}(x).
\end{align}
Furthermore, using Lemma \ref{lem-Qx}, we find that, for any
$\alpha\in \{0,\frac{1}{3},\frac{2}{3}\}$ and $x\in \bigcup_{Q\in
\mathcal{D}^{\alpha}_{C_{(k,\beta)}\lambda,\beta,k}[f]} Q$,
there exists an open interval $Q_x \in
\mathcal{D}^{\alpha}_{C_{(k,\beta)}\lambda,\beta,k}[f]$
containing $x$ such that
\begin{align*}
\sum_{Q\in \mathcal{D}_{C_{(k,\beta)}\lambda,\beta,k}^{\alpha}[f]}
|Q|^{q(\beta-1)}{\bf 1}_{Q}(x)
\lesssim
|Q_x|^{q(\beta-1)}.
\end{align*}
From this, \eqref{e-apctr01}, and the Tonelli theorem, we deduce that
\begin{align*}
&\int_{\mathbb{R}}
\left[\int_{0}^{\infty}{\bf 1}_{E^{\spadesuit}_{\lambda,
\beta,k}[f]}(x,r)r^{q(\beta-1)-1}\,dr
\right]^{\frac{p}{q}}\upsilon(x)\,dx\\
&\quad \lesssim
\int_{\mathbb{R}}
\left[	\sum_{\alpha\in \{0,\frac{1}{3}, \frac{2}{3}\}}
|Q_x|^{q(\beta-1)}
\right]^{\frac{p}{q}}\upsilon(x)\,dx
\sim
\int_{\mathbb{R}}
\sum_{\alpha\in \{0,\frac{1}{3}, \frac{2}{3}\}}
|Q_x|^{p(\beta-1)}
\upsilon(x)\,dx\\
&\quad
\leq
\int_{\mathbb{R}}
\sum_{\alpha\in \{0,\frac{1}{3}, \frac{2}{3}\}}
\sum_{Q\in \mathcal{D}_{C_{(k,\beta)}\lambda,\beta,k}^{\alpha}[f]}
|Q|^{p(\beta-1)}{\bf 1}_{Q}(x)
\upsilon(x)\,dx\\
&\quad=
\sum_{\alpha\in \{0,\frac{1}{3},
\frac{2}{3}\}}\sum_{Q\in\mathcal{D}^{\alpha}_{C_{(k,\beta)}\lambda,\beta,k}[f]}
|Q|^{p(\beta-1)}\upsilon(Q).
\end{align*}
This further implies that \eqref{e-x-r} holds and hence finishes the
proof of Lemma \ref{lem-lo}.
\end{proof}

\begin{proof}[Proof of Corollary \ref{cor-ap}]
By Theorem \ref{esti-omega} with $\ell:=k$, we immediately obtain
(iii) holds if (i) holds.
Now, we prove that (i) implies (ii).
Assume $\upsilon\in A_p(\mathbb{R})$.
Repeating the proof of both Proposition \ref{pro-E0}(ii) and Theorem
\ref{thm-up-ap} with $A_1$, $\ell$,
and $n(\frac{1}{p}-\frac{1}{q})<\ell$ replaced, respectively, by
$A_p (\mathbb{R})$, $k$, and $1-\frac{1}{q}<k$,
we conclude that (ii) holds.

Next, we prove (ii) implies (i). Assume (ii) hold. Let $I_0:=[4,8]$,
$I_1 := [1,2]$, $I_2:= [16,32]$, and $I_3:=[3,9]$.
From this, \eqref{e-shp-1}, and \eqref{eq-cor-ap}, we deduce that,
for any $f\in \dot{W}^{k,p}_{\upsilon} (\mathbb{R})$,
\begin{align}\label{eq-wE}
\sup_{\lambda\in(0,\infty)}
\lambda^p \int_{\mathbb{R}}
\left[\int_{\mathbb{R}}{\bf 1}_{E^{\star}_{\lambda,
\frac{\gamma}{q},k}[f]}(x,y)|x-y|^{\gamma-1}
\,dy\right]^{\frac{p}{q}}\upsilon(x)\,dx
\lesssim
\|f\|_{\dot{W}^{k,p}_{\upsilon} (\mathbb{R})},
\end{align}
where $E^{\star}_{\lambda,
\frac{\gamma}{q},k}[f]$ and
$\Delta^k_{x,y}f$ are the same as in \eqref{e-sym}.
Observe that inequality \eqref{eq-wE} is both dilation and
translation invariant;
that is, for any $\delta\in (0,\infty)$ and
$x_0\in \mathbb{R}$, both the weights $\upsilon (\delta \cdot)$ and
$\upsilon (\cdot-x_0)$ satisfy \eqref{eq-wE}
with the same implicit positive constant.
This, together with Lemma \ref{lem-apwight}(iii), further implies
that, to show $\upsilon\in A_p (\mathbb{R})$, it suffices to prove
that,
for any nonnegative function $g\in L^{1}_{\rm loc}(\mathbb{R})$,
\begin{align}\label{eq-ggg}
\left[\int_{I_0}g(x)\,dx\right]^p
\lesssim
\frac{1}{\upsilon(I_0)}
\int_{I_0}[g(x)]^p\upsilon(x)\,dx.
\end{align}
To do this, we assume that $ g\in C^{\infty}(\mathbb{R})$ is
nonnegative and choose $\eta\in C^{\infty}(\mathbb{R})$ such that
\begin{align}\label{eq-ETA}
{\bf 1}_{I_0}\le \eta \le {\bf 1}_{I_3}.
\end{align}
For any $h\in L^{1}(\mathbb{R})$ and $x\in \mathbb{R}$, let
\begin{align}\label{eq-aaaa}
\mathcal{A}(h)(x):= \int_{-\infty}^{x}h(t)\,dt
\end{align}
and $f:=\mathcal{A}^k(g\eta)$,
where $\mathcal{A}^k$ denotes the $k$-fold iteration of
$\mathcal{A}$. Clearly, $f\in C^{\infty}(\mathbb{R})$,
$f^{(k)}=g\eta$, and $\mathrm{supp\,}(f^{(k)})\subset I_3$.
From \cite[p.\,336, (4.16)]{bs1988}, we infer that,
for any $x,h\in\mathbb{R}$,
\begin{align*}
\Delta^k_h f(x)=\int_{\mathbb{R}}M_{k}(t)
\sum_{\zeta\in \mathbb{Z}_{+}^n,\,|\zeta|=k}\frac{k!}{\zeta!}
\partial^{\zeta}f(x+th)h^\zeta \,dt,
\end{align*}
where $M_1:= {\bf 1}_{(0,1)}$ and, for any $j\in \mathbb{N}$,
$M_{j+1}:=M_{j}\ast M_1$.
By these and a change of variables, we find that, for any $x\in I_1$
and $y\in I_2$,
\begin{align}\label{e-Dkxy}
\Delta^k_{x,y}f
&=\left({y-x}\right)^k
\int_{\mathbb{R}}M_{k}(t)
f^{(k)}\left(x+t\left[\frac{y-x}{k}\right]\right) \,dt\notag\\
&=\left({y-x}\right)^{k-1}
\int_{\mathbb{R}}M_{k}\left(\frac{k[s-x]}{y-x}\right)
g(s)\eta(s) \,ds\notag\\
&\ge \left({y-x}\right)^{k-1}
\left[\inf_{s\in I_3}M_{k}\left(\frac{k[s-x]}{y-x}\right)\right]
\int_{I_0}
g(s) \,ds.
\end{align}
Observe that, for any $x\in I_1$, $y\in I_2 $,
and $s\in I_3$, $\frac{k}{31}\le \frac{k(s-x)}{y-x}\le \frac{4}{7}k$.
This, combined with the definition of $M_k$,
further implies that $\inf_{x\in I_1,\,y\in I_2,\,s\in
I_3}M_{k}(\frac{k[s-x]}{y-x})>0$.
Therefore, for any $x\in I_1$ and $y\in I_2 $,
\begin{align*}
\Delta^k_{x,y}f&\gtrsim
|x-y|^{k+{\frac{\gamma}{q}}}
\int_{I_0}
g(s) \,ds.
\end{align*}
This, together with the symmetry of $E^{\star}_{\lambda,
\frac{\gamma}{q},k}[f]$, further implies that
\begin{align*}
(I_1\times I_2)\cup (I_2\times I_1)
\subset
E^{\star}_{\lambda_{(k,q,\gamma)},
\frac{\gamma}{q},k}[f],
\end{align*}
where
$
\lambda_{(k,q,\gamma)}
:=
C_{(k,q,\gamma)}\int_{I_0}
g(s) \,ds
$
with a implicit
positive
constant $C_{(k,q,\gamma)}$
depending only on
$k$, $q$, and $\gamma$.
From this, \eqref{eq-wE}, and \eqref{eq-ETA}, it follows that
\begin{align}\label{eq-gI3}
&\upsilon(I_1\cup I_2)
\left[\int_{I_0}g(s) \,ds\right]^p\notag\\
&\quad \lesssim
\lambda_{(k,q,\gamma)}^p \int_{I_1\cup I_2}\upsilon(s)\,ds\notag\\
&\quad\lesssim
\lambda_{(k,q,\gamma)}^p \int_{\mathbb{R}}
\left[\int_{\mathbb{R}}{\bf 1}_{E^{\star}_{ \lambda_{(k,q,\gamma)},
\frac{\gamma}{q},k}[f]}(x,y)|x-y|^{\gamma-1}
\,dy\right]^{\frac{p}{q}}\upsilon(x)\,dx
\notag\\
&\quad\lesssim
\|f\|^p_{\dot{W}^{k,p}_{\upsilon} (\mathbb{R})}
\le
\int_{I_3} \left[g(s)\eta(s)\right]^p\upsilon(s)\,ds
\le
\int_{I_3} \left[g(s)\right]^p\upsilon(s)\,ds.
\end{align}
Next, fix a nonnegative $g\in L^{1}_{\rm loc}$ and
choose a nonnegative $\varphi\in C_{\rm c}^{\infty}(\mathbb{R})$
such that
$\int_{\mathbb{R}}\varphi (x)\,dx=1$.
For any $\varepsilon\in (0,\infty)$ and $x\in \mathbb{R}$, let
\begin{align*}
\varphi_{\varepsilon}(x):=
\frac{1}{\varepsilon}\varphi\left(\frac{x}{\varepsilon} \right)
\ \ \text{and}\ \
g_{\varepsilon}(x):= \left(g {\bf 1}_{I_0}\right)\ast
\varphi_{\varepsilon}(x).
\end{align*}
Therefore, $\{g_{\varepsilon}\}_{\varepsilon} \subset
C^{\infty}(\mathbb{R})$ are nonnegative and hence \eqref{eq-gI3} holds
with
$g:=g_{\varepsilon}$ for any $\varepsilon\in (0,\infty)$.
Using \cite[Corollary 2.9]{D2001}, we conclude that
$\lim_{\varepsilon\to 0^{+}}g_{\varepsilon}=g$ almost everywhere in
$I_0$.
Moreover, applying the Young inequality, we find that, for any
$\varepsilon\in (0,\infty)$,
\begin{align*}
\left\|g_{\varepsilon}\right\|_{L^{\infty}(\mathbb{R})}
\le
\left\|\varphi_{\varepsilon}\right\|_{L^{1}(\mathbb{R})}
\left\|g{\bf 1}_{I_0}\right\|_{L^{\infty}(\mathbb{R})}
=\left\|g{\bf 1}_{I_0}\right\|_{L^{\infty}(\mathbb{R})}.
\end{align*}
From this, the Lebesgue dominated convergence theorem, and
\eqref{eq-gI3} with $g:=g_{\varepsilon}$ for any $\varepsilon\in
(0,1)$, we deduce that
\begin{align}\label{eq-a0}
\upsilon(I_1\cup I_2)\left[\fint_{I_0}g(s)\,ds\right]^p
\sim
\upsilon(I_1 \cup I_2)\left[\int_{I_0}g(s)\,ds\right]^p
\lesssim
\int_{I_0}[g(s)]^p\upsilon(s)\,ds.
\end{align}
Applying this and the fact that \eqref{eq-wE} has the dilation
invariance again, we conclude that, for any
nonnegative $g\in L^{1}_{\rm loc}$,
\begin{align*}
\upsilon(I_0 \cup [64,108])\left[\fint_{I_2}g(s)\,ds\right]^p
\lesssim
\int_{I_2}[g(s)]^p\upsilon(s)\,ds.
\end{align*}
Letting $g:={\bf 1}_{I_2}$, we obtain
\begin{align*}
\upsilon(I_0)\le \upsilon(I_0 \cup [64,108])
\lesssim
\upsilon(I_2)
\le
\upsilon(I_1 \cup I_2).
\end{align*}
This, combined with \eqref{eq-a0}, implies \eqref{eq-ggg} and
hence $\upsilon\in A_p (\mathbb{R})$. Thus, we complete the proof that
(ii) implies (i).

Next, we show that (iii) implies (i).
Assume that (iii) holds.
From Lemma \ref{lem-lo} and \eqref{eq-n1-ap},
we infer that, for any $f\in \dot{W}^{k,p}_{\upsilon} (\mathbb{R})$,
\begin{align}\label{e-ap-n101}
\sup_{\lambda\in (0,\infty)}\lambda^p \int_{\mathbb{R}}
\left[\int_{\mathbb{R}}{\bf 1}_{E^{\spadesuit}_{\lambda,
\beta,k}[f]}(x,r)r^{q(\beta-1)-1}\,dr
\right]^{\frac{p}{q}}\upsilon(x)\,dx
\lesssim
\|f\|^p_{\dot{W}^{k,p}_{\upsilon} (\mathbb{R})}.
\end{align}
Assume that $ g\in C^{\infty}(\mathbb{R})$ is nonnegative and $\eta\in
C^{\infty}(\mathbb{R})$ such that
\eqref{eq-ETA} holds.
Let $\mathcal{A}$ be the same as in \eqref{eq-aaaa},
$f:=\mathcal{A}^k(g\eta)$, $x\in I_0$, and $r\in [28,29]$.
Then $I_1 \cup I_2 \subset B(x,r)$.
Using these, \eqref{e-Dkxy}, and the fact that
$\inf_{x\in I_1,\,y\in I_2,\,s\in I_3}M_{k}(\frac{k[s-x]}{y-x})>0$, we
find that there exists a positive constant $C_{(k,\beta)}$,
depending only on $k$ and $\beta$, such that
\begin{align*}
\int_{x-r}^{x+r}\int_{x-r}^{x+r} \left|\Delta^k_{y,z}f\right|\,dy\,dz
&\ge
\int_{I_2}\int_{I_1} \left|\Delta^k_{y,z}f\right|\,dy\,dz\\
&\ge
\int_{I_2}\int_{I_1}\left(\frac{z-y}{k}\right)^{k-1}
\left[\inf_{s\in I_3}M_{k}\left(\frac{k[s-y]}{z-y}\right)\right]
\,dy\,dz\int_{I_0}
g(s) \,ds\\
&\ge
C_{(k,\beta)}r^{\beta+k+1}\int_{I_0}g(s) \,ds.
\end{align*}
Let $\lambda_{(k,\beta)}:=C_{(k,\beta)}\int_{I_0}g(s) \,ds$.
We then have $I_0 \times [28,29]\subset
E^{\spadesuit}_{\lambda_{(k,\beta)},\beta,k}[f]$.
By this and \eqref{e-ap-n101}, we conclude that
\begin{align*}
\upsilon(I_0)\left[\int_{I_0}g(s) \,ds\right]^p
&\lesssim
\lambda_{(k,\beta)}^p \int_{I_0}
\left[\int_{28}^{29}r^{q(\beta-1)-1}\,dr
\right]^{\frac{p}{q}}\upsilon(x)\,dx\\
&\le
\lambda_{(k,\beta)}^p \int_{\mathbb{R}}
\left[\int_{0}^{\infty}{\bf 1}_{E^{\spadesuit}_{\lambda,
\beta,k}[f]}(x,r)r^{q(\beta-1)-1}\,dr
\right]^{\frac{p}{q}}\upsilon(x)\,dx\\
&\lesssim
\|f\|_{\dot{W}^{k,p}_{\upsilon} (\mathbb{R})}
\le
\int_{I_3} \left[g(s)\eta(s)\right]^p\upsilon(s)\,ds
\le
\int_{I_3} \left[g(s)\right]^p\upsilon(s)\,ds.
\end{align*}
Using this and a slight modification of the proof of (i),
we find that \eqref{eq-ggg} also holds.
Therefore, we obtain $\upsilon \in A_p (\mathbb{R})$.
This finishes the proof of (iii) implies (i) and hence
Corollary \ref{cor-ap}.
\end{proof}

\section{Proofs of Theorems \ref{thm-main}, \ref{t-chaWkX}, and
\ref{thm-fsGN}}\label{sec-Formulae}

We present the proofs of Theorems \ref{thm-main}, \ref{t-chaWkX}, and
\ref{thm-fsGN}, respectively, in Subsection \ref{ssub4-1},
\ref{ss-cha}, and \ref{ssec-fract}.

\subsection{Proof of Theorem \ref{thm-main}}\label{ssub4-1}

To prove Theorem \ref{thm-main}, we need the following
upper estimate in {\rm BBF} spaces, which can be obtained by repeating
the proof of \cite[(4.10)]{dlyyz-bvy}
with $E_f(\lambda,q)$, $|\nabla f|$, and Theorem 3.5 therein replaced,
respectively, by $E_{\lambda,\frac{\gamma}{q},k,\ell}[f]$,
$|\nabla^{\ell} f|$,
and Theorem \ref{thm-up-ap} here; we omit the details.

\begin{proposition}\label{pro-up-g}
Let $X$ be a {\rm BBF} space.
Assume that there exists some $p\in [1,\infty)$
such that $X^{\frac{1}{p}}$ is a ball Banach function space
and the Hardy--Littlewood maximal operator
$\mathcal{M}$ is bounded on $(X^{\frac{1}{p}})'$.
Let $k,\ell\in\mathbb{N}$ with $\ell\le k$,
$q\in(0,\infty)$ satisfy $n(\frac1p-\frac1q)<\ell$,
$\Gamma_{p,q}$ be the same as in \eqref{eq-GAMMA}, and
$\gamma\in\Gamma_{p,q}$.
Then there exists a positive constant $C$
such that, for any $f\in C^{\infty}$ with
$|\nabla^{\ell}f|\in C_{\rm c}$,
\begin{align*}
\sup_{\lambda\in(0,\infty)}\lambda
\left\|\left[\int_{\mathbb{R}^n}
{\bf 1}_{E_{\lambda,\frac{\gamma}{q},k,\ell}[f]}(\cdot,h)
|h|^{\gamma-n}\,dh\right]^{\frac1q}\right\|_{X}
\le C
\left\|\nabla^\ell f\right\|_{X}.
\end{align*}
\end{proposition}

Next, we investigate the lower estimate of BSVY formulae.
For this purpose, let $U\subset \mathbb{R}^n$ be an open set.
Recall that a function
$f:\ U\to \mathbb{C}$ is said to be \emph{locally Lipschitz} on $U$
if, for
any compact set $K\subset U$, there exists a positive constant $C_K$,
depending on $K$, such that, for any
$x,y\in K$,
\begin{align*}
\left|f(x)-f(y)\right|\le C_K |x-y|.
\end{align*}
Then we have the following lower estimate.

\begin{proposition}\label{pro-linf}
Let $X$ be a {\rm BBF} space, $k\in \mathbb{N}$,
$q\in (0,\infty)$, $\gamma\in\mathbb{R}\setminus \{0\}$,
and $\lambda \to L_\gamma$ be the same as in \eqref{e-Lgamma}.
\begin{enumerate}[{\rm (i)}]
\item For any
$f\in C^k$ with $|\nabla^k f|\in C_{\rm
c}$,
\begin{align}\label{e-linf-1}
&	\liminf_{\lambda\to L_{\gamma}}\lambda\left\|\left[
\int_{{\mathbb{R}^n}}{\bf
1}_{E_{\lambda,\frac{\gamma}{q},k}[f]}(\cdot,h)
|h|^{\gamma-n}\, dh\right]^{\frac{1}{q}}
\right\|_X \notag\\
&	\quad\geq
|\gamma|^{-\frac{1}{q}}\left\|\left[ \int_{\mathbb{S}^{n-1}}
\left|\sum_{\alpha\in\mathbb{Z}_{+}^n,\,|\alpha|=k}\partial^{\alpha}f(\cdot)\xi^{\alpha}\right|^q\,d\mathcal{H}^{n-1}(\xi)
\right]^{\frac{1}{q}}\right\|_X.
\end{align}

\item For any
bounded open set $\Omega\subset \mathbb{R}^n$
and
$f\in C^k $ satisfying that, for any $\alpha\in
\mathbb{Z}_{+}^n$ with $|\alpha|=k$, $\partial^\alpha f$ is locally
Lipschitz on $\mathbb{R}^n$,
\begin{align}\label{e-linf-U}
&	\liminf_{\lambda\to L_{\gamma}}\lambda\left\|\left[
\int_{{\mathbb{R}^n}}{\bf
1}_{E_{\lambda,\frac{\gamma}{q},k}[f]}(\cdot,h)
|h|^{\gamma-n}\, dh\right]^{\frac{1}{q}}{\bf 1}_{\Omega}\right\|_X \notag\\
&	\quad\geq
|\gamma|^{-\frac{1}{q}}\left\|\left[ \int_{\mathbb{S}^{n-1}}
\left|\sum_{\alpha\in\mathbb{Z}_{+}^n,\,|\alpha|=k}\partial^{\alpha}f(\cdot)\xi^{\alpha}\right|^q\,d\mathcal{H}^{n-1}(\xi)
\right]^{\frac{1}{q}}{\bf 1}_{\Omega}\right\|_X.
\end{align}
\end{enumerate}
\end{proposition}

\begin{proof}
We first show (ii).
Let $\Omega\subset \mathbb{R}^n$ be a
bounded open set and
$f\in C^k$ satisfy that, for any $\alpha\in
\mathbb{Z}_{+}^n$ with $|\alpha|=k$, $\partial^\alpha f$ is locally
Lipschitz on $\mathbb{R}^n$.
From \cite[Proposition 1.4.5]{G14}, we infer that,
for any $h\in \mathbb{R}^n$ and for almost every $x\in \mathbb{R}^n$,
\begin{align}\label{e-hdo}
\Delta^{k}_{h}f(x)
=
\int_{[0,1]^k}
\sum_{\alpha\in\mathbb{Z}_{+}^n,\,|\alpha|=k}\partial^{\alpha}f(x+[s_1+\cdots+s_k]h)h^{\alpha}
\,ds_1\cdots \, ds_k.
\end{align}
This, together with a change of variables, further implies that, for
almost every $x\in \Omega$
and for any $h:=r\xi$ with $r\in (0,\infty)$ and $\xi\in
\mathbb{S}^{n-1}$,
\begin{align}\label{e-linf-ke}
\lim_{r\to 0^+}\frac{|\Delta^k_{r\xi} f(x)|}{r^k}
=
\left|\sum_{\alpha\in\mathbb{Z}_{+}^n,\,|\alpha|=k}\partial^{\alpha}f(x)\xi^{\alpha}\right|.
\end{align}
Applying a change of variables again, we obtain, for any
$\lambda\in (0,\infty)$,
\begin{align}\label{e-Ch-l}
&\lambda\left\|\left[
\int_{{\mathbb{R}^n}}{\bf
1}_{E_{\lambda,\frac{\gamma}{q},k}[f]}(\cdot,h)
|h|^{\gamma-n}\, dh\right]^{\frac{1}{q}}{\bf 1}_{\Omega}\right\|_X
=\left\|\left[
\int_{{\mathbb{R}^n}}{\bf
1}_{{\mathcal{E}}_{\lambda,\frac{\gamma}{q},k}[f]}(\cdot,h)
|h|^{\gamma-n}\, dh\right]^{\frac{1}{q}}\right\|_X,
\end{align}
where
\begin{align*}
{\mathcal{E}}_{\lambda,\frac{\gamma}{q},k}[f]
:=
\left\{(x,h)\in \Omega\times (\mathbb{R}^n\setminus\{{\bf 0}\}):\
\lambda^{\frac{qk}{\gamma}}|h|^{-\frac{\gamma}{q}-k}\left|
\Delta_{\lambda^{-\frac{q}{\gamma}}h}^{k}f (x)\right|>1\right\}.
\end{align*}
Notice that $\lambda^{-\frac{q}{\gamma}}\to 0$ as $\lambda\to
L_\gamma$. From this and \eqref{e-linf-ke},
it follows that
\begin{align*}
\liminf_{\lambda\to L_\gamma}
{\bf 1}_{{\mathcal{E}}_{\lambda,\frac{\gamma}{q},k}[f]}
\ge
{\bf 1}_{\{(x,h)\in \Omega\times (\mathbb{R}^n\setminus \{{\bf 0}\}):\
|h|^{-\frac{\gamma}{q}-k}|\sum_{\alpha\in\mathbb{Z}_{+}^n,\,|\alpha|=k}\partial^{\alpha}f(x)h^{\alpha}|>1
\}}.
\end{align*}
This, combined with \eqref{e-Ch-l}, Remark \ref{rm-bqbf}(iii),
and a change of variables, further implies that
\begin{align*}
&	\liminf_{\lambda\to L_{\gamma}}\lambda\left\|\left[
\int_{{\mathbb{R}^n}}{\bf
1}_{E_{\lambda,\frac{\gamma}{q},k}[f]}(\cdot,h)
|h|^{\gamma-n}\, dh\right]^{\frac{1}{q}}{\bf 1}_{\Omega}\right\|_X\\
&\quad\ge
\left\|\left[\int_{{\mathbb{R}^n}}
{\bf 1}_{\{(x,h)\in \Omega\times (\mathbb{R}^n\setminus \{{\bf 0}\}):\
|h|^{-\frac{\gamma}{q}-k}|\sum_{\alpha\in\mathbb{Z}_{+}^n,\,|\alpha|=k}\partial^{\alpha}f(x)h^{\alpha}|>1\}}(\cdot,h)
|h|^{\gamma-n}\, dh\right]^{\frac{1}{q}}\right\|_X\\
&\quad=
\left\|\left[
\int_{\mathbb{S}^{n-1}}
\int_{0}^{\infty}
{\bf 1}_{\{(x,r\xi)\in \Omega\times (\mathbb{R}^n\setminus \{{\bf
0}\}):\
r^{-\frac{\gamma}{q}}|\sum_{\alpha\in\mathbb{Z}_{+}^n,\,|\alpha|=k}\partial^{\alpha}f(x)\xi^{\alpha}|>1\}}(\cdot,r\xi)
r^{\gamma-1}\, dr
\,d\mathcal{H}^{n-1}(\xi)\right]^{\frac{1}{q}}\right\|_X\\
&	\quad=
|\gamma|^{-\frac{1}{q}}
\left\|\left[ \int_{\mathbb{S}^{n-1}}
\left|\sum_{\alpha\in\mathbb{Z}_{+}^n,\,|\alpha|=k}\partial^{\alpha}f(\cdot)\xi^{\alpha}\right|^q\,d\mathcal{H}^{n-1}(\xi)
\right]^{\frac{1}{q}}{\bf 1}_{\Omega}\right\|_X.
\end{align*}
This finishes the proof of (ii).

Next, we prove (i). Let
$f\in C^k$ and
${\rm supp}(|\nabla^k f|)\subset B({\bf 0},R)$
with $R\in (0,\infty)$.
Clearly, for any $\alpha\in \mathbb{Z}_{+}^n$ with $|\alpha|=k$,
$\partial^\alpha f$ is locally Lipschitz on $\mathbb{R}^n$.
Using this and \eqref{e-linf-U} with $\Omega:=B({\bf 0},R)$,
we immediately obtain
\eqref{e-linf-1} and hence (i).
This then finishes the proof of Proposition \ref{pro-linf}.
\end{proof}

Via borrowing some ideas from the proof of
Frank \cite[Lemma 6]{f22},
we establish the following subtle estimate in the limiting
identity of BSVY formulae,
which is the key point to improve
the existing results about BSVY formulae in BBF spaces.

\begin{proposition}\label{pro-limsup}
Let $X$ be a {\rm BBF} space.
Assume that there exists some $p\in [1,\infty)$
such that $X^{\frac{1}{p}}$ is a ball Banach function space
and the Hardy--Littlewood maximal operator
$\mathcal{M}$ is bounded on $(X^{\frac{1}{p}})'$.
Let $k\in {\mathbb{N}}$, $q\in(0,\infty)$,
and $\gamma\in \mathbb{R}\setminus \{0\}$.
Then the following statements hold.
\begin{enumerate}[{\rm (i)}]
\item If both $n=1$ and $\gamma\in
(-\infty,-\frac{q}{p})\cup(0,\infty)$ or both $n\in \mathbb{N} \cap
[2,\infty)$
and $n(\frac{1}{p}-\frac{1}{q})<k$, then, for any
$f\in C^{k}$ with $|\nabla^{k}f|\in C_{\rm
c}$,
\begin{align}\label{eq-limup}
&\limsup_{\lambda\to L_\gamma}\lambda\left\|\left[
\int_{{\mathbb{R}^n}}{\bf
1}_{E_{\lambda,\frac{\gamma}{q},k}[f]}(\cdot,h)
|h|^{\gamma-n}\, dh\right]^{\frac{1}{q}}\right\|_X\notag\\
&\quad\le
|\gamma|^{-\frac{1}{q}}\left\|\left[ \int_{\mathbb{S}^{n-1}}
\left|\sum_{\alpha\in\mathbb{Z}_{+}^n,\,|\alpha|=k}\partial^{\alpha}f(\cdot)\xi^{\alpha}\right|^q\,d\mathcal{H}^{n-1}(\xi)
\right]^{\frac{1}{q}}\right\|_X,
\end{align}
where $\lambda \to L_\gamma$ is the same as in \eqref{e-Lgamma}.
\item If $n(\frac{1}{p}-\frac{1}{q})<k$, then \eqref{eq-limup} holds
for any
$f\in C_{\rm c}^{k}$.
\end{enumerate}
\end{proposition}

To show the above proposition,
we need
the following inequality;
see, for instance, \cite[p.\,699]{b2002}.

\begin{lemma}\label{lem-thet}
Let $q\in [1,\infty)$. For any $\theta\in(0,1)$, there exists a
positive
constant $C_{(\theta)}$,
depending only on $\theta$, such that, for any $a,b \in (0,\infty)$,
\begin{align*}
(a+b)^q \leq (1+\theta)a^q+C_{(\theta)}b^q.
\end{align*}
\end{lemma}

Recall that
the following extrapolation lemma of {\rm BBF} spaces is exactly \cite[Lemma
4.6]{dlyyz-bvy}.

\begin{lemma}\label{lem-rxgmm}
Let $X$ be a {\rm BBF} space.
Assume that the Hardy--Littlewood maximal operator $\mathcal{M}$
is bounded on $X$ with its operator norm denoted by
$\|{\mathcal{M}}\|_{X\to X}$.
For any $g\in X$ and $x\in{\mathbb{R}^n} $, let
\begin{align}\label{eq-rxg}
R_{X}g(x):=\sum_{m=0}^{\infty}\frac{\mathcal{M}^m g (x)}{2^m
\|\mathcal{M}\|_{X\to
X}^m},
\end{align}
where, for any $m\in {\mathbb{N}}$, $\mathcal{M}^m$ is the
$m$-fold iteration
of $\mathcal{M}$ and
$\mathcal{M}^0 g(x):=|g(x)|$. Then, for any $g\in X$,
\begin{enumerate}[{\rm (i)}]
\item for any $x\in {\mathbb{R}^n}$, $|g(x)|\leq R_{X}g(x);$
\item $R_{X}g\in A_1 $ and $[R_{X}g]_{A_1}\leq
2\|\mathcal{M}\|_{X\to X};$
\item $\|R_{X}g\|_X \leq 2\|g\|_X$.
\end{enumerate}
\end{lemma}

We also need the following technique
lemma from \cite[Lemma 3.3]{ZYY23}.

\begin{lemma}\label{lem-rxp}
Let $X$ be a {\rm BBF} space.
Assume that there exists some $p\in [1,\infty)$
such that $X^{\frac{1}{p}}$ is a ball Banach function space
and the Hardy--Littlewood maximal operator
$\mathcal{M}$ is bounded on $(X^{\frac{1}{p}})'$.
Then, for any $f\in X$,
\begin{align*}
\|f\|_X
\leq
\sup_{\|g\|_{(X^{\frac{1}{p}})'}\leq
1}\left[\int_{{\mathbb{R}^n}}|f(x)|^p R_{(X^{\frac{1}{p}})'}
g(x)\, dx\right]^{\frac{1}{p}}
\leq
2^{\frac{1}{p}}\|f\|_X,
\end{align*}
where $R_{(X^{\frac{1}{p}})'}
g$ is the same as in \eqref{eq-rxg} with $X$ replaced by
$(X^{\frac{1}{p}})'$.
\end{lemma}

\begin{proof}[Proof of Proposition \ref{pro-limsup}]
Let
$f\in C^{k}$ with $|\nabla^{k}f|\in C_{\rm
c}$ and ${\mathop\mathrm{\,supp\,}}(|\nabla^k f|)\subset
B(\textbf{0},R)$ for some $R\in (0,\infty)$.
We now show (i) by considering the following
two cases on $n$ .

\emph{Case 1)} $n=1$. In this case,
repeating the proof of \cite[Theorem 3.25]{ZYY23}
with $E_{\lambda,\frac{\gamma}{q}}[f]$ and $f'$ therein replaced,
respectively, by
$E_{\lambda,\frac{\gamma}{q},k}[f]$ and $f^{(k)}$ here,
we obtain the desired result.

\emph{Case 2)} $n\in \mathbb{N}\cap [2,\infty)$.
In this case, by the assumption
$|\nabla^k f|\in C_{\rm c}$
and \eqref{e-hdo},
we find that there exists a positive constant $B_{(k,f)}$,
depending only on both $k$ and $f$, such that,
for any $x,h\in\mathbb{R}^n$,
\begin{align}\label{eq-Bfx}
\left|\Delta^k_h
f(x)-\sum_{\alpha\in\mathbb{Z}_{+}^n,\,|\alpha|=k}\partial^{\alpha}f(x)h^{\alpha}\right|
\le
B_{(k,f)}|h|^{k+1}.
\end{align}
Using this and a slight modification of
the proof of \cite[Theorem 3.22(i)]{ZYY23},
we obtain the desired result when $\gamma\in (0,\infty)$.
Thus, from now on, we assume $\gamma\in (-\infty,0)$.
For any $x\in\mathbb{R}^n$, $\xi\in \mathbb{S}^{n-1}$,
and $\varepsilon,\lambda\in (0,\infty)$,
let
\begin{equation}\label{eq-Y2.5}
R_\varepsilon (x,\xi,\lambda): =
\min\left\{\varepsilon,\lambda^{-\frac{q}{\gamma}}\left[\left|\sum_{\alpha\in\mathbb{Z}_{+}^n,\,|\alpha|=k}\partial^{\alpha}f(x)\xi^{\alpha}\right|+
B_{(k,f)}\varepsilon\right]^{\frac{q}{\gamma}} \right\}.
\end{equation}
Next, we claim that, for any $\lambda\in (0,\infty)$,
if $(x,r\xi)\in E_{\lambda,\frac{\gamma}{q},k}[f]$ with
$r\in(0,\infty)$ and $\xi\in\mathbb{S}^{n-1}$, then
$
r\in [R_\varepsilon (x,\xi,\lambda),\infty)
$ for any $\varepsilon\in(0,\infty)$.
Indeed, from \eqref{eq-Bfx} and \eqref{e-EE},
we deduce that, for any $\varepsilon,\lambda\in (0,\infty)$,
if $(x,r\xi)\in E_{\lambda,\frac{\gamma}{q},k}[f]\cap \{\mathbb{R}^n
\times B({\bf 0},\varepsilon)\}$ with $r\in(0,\infty)$ and
$\xi\in\mathbb{S}^{n-1}$,
then
$$\lambda < \frac{|\Delta^k_{r\xi}f(x)|}{r^{k+\frac{\gamma}{q}}}
\le
\frac{|\sum_{\alpha\in\mathbb{Z}_{+}^n,\,|\alpha|=k}\partial^{\alpha}f(x)\xi^{\alpha}|+
B_{(k,f)}r}{r^{\frac{\gamma}{q}}};$$
using this and the assumption that $\gamma\in (-\infty,0)$, we find
that
$$\varepsilon> r>
\left\{\lambda^{-1}\left[\left|\sum_{\alpha\in\mathbb{Z}_{+}^n,\,|\alpha|=k}\partial^{\alpha}f(x)\xi^{\alpha}\right|+
B_{(k,f)}r\right] \right\}^{\frac{q}{\gamma}}
>
\left\{\lambda^{-1}\left[\left|\sum_{\alpha\in\mathbb{Z}_{+}^n,\,|\alpha|=k}\partial^{\alpha}f(x)\xi^{\alpha}\right|+
B_{(k,f)}\varepsilon\right] \right\}^{\frac{q}{\gamma}},$$
which, together with \eqref{eq-Y2.5},
further implies that $r\in [R_\varepsilon (x,\xi,\lambda),\infty)$.
On the other hand,
for any $\varepsilon,\lambda\in (0,\infty)$,
if $(x,r\xi)\in E_{\lambda,\frac{\gamma}{q}}[f]
\cap [\mathbb{R}^n \times B({\bf 0},\varepsilon)^{\complement}]$
with $r\in(0,\infty)$ and $\xi\in\mathbb{S}^{n-1}$,
then it is obvious that $r\in [R_\varepsilon (x,\xi,\lambda),\infty)$.
Therefore, the claim holds.
By this claim, a change of variables, Lemma \ref{lem-thet},
and the assumption that $X$ is a {\rm BBF} space, we
conclude that, for any $\eta\in(0,\infty)$,
there exists $C_{(\eta)}\in[1,\infty)$, depending only on $\xi$,
such that, for any $\lambda\in (0,\infty)$,
\begin{align*}
&\lambda \left\|\left[
\int_{{\mathbb{R}^n}}{\bf
1}_{E_{\lambda,\frac{\gamma}{q},k}[f]}(\cdot,h)
|h|^{\gamma-n}\, dh\right]^{\frac{1}{q}}{\bf
1}_{B(\textbf{0},(k+1)R)}\right\|_X\\
&\quad \le \lambda \left\|
\left[\int_{\mathbb{S}^{n-1}}\int_{R_\varepsilon
(\cdot,\xi,\lambda)}^{\infty}\,r^{\gamma-1}dr
\,d\mathcal{H}^{n-1}(\xi)
\right]^{\frac{1}{q}}{\bf 1}_{B(\textbf{0},(k+1)R)} \right\|_X\\
&\quad = (-\gamma)^{-\frac{1}{q}}
\lambda \left\|
\left[\int_{\mathbb{S}^{n-1}}\left[R_\varepsilon
(\cdot,\xi,\lambda)\right]^{\gamma}
\,d\mathcal{H}^{n-1}(\xi)
\right]^{\frac{1}{q}}{\bf 1}_{B(\textbf{0},(k+1)R)} \right\|_X\\
&\quad \le(1+\eta)
(-\gamma)^{-\frac{1}{q}}\left\|\left[ \int_{\mathbb{S}^{n-1}}
\left|\sum_{\alpha\in\mathbb{Z}_{+}^n,\,|\alpha|=k}\partial^{\alpha}
f(\cdot)\xi^{\alpha}\right|^q\,d\mathcal{H}^{n-1}(\xi)
\right]^{\frac{1}{q}}{\bf 1}_{B(\textbf{0},(k+1)R)}\right\|_{X}\\
&\qquad+C_{(\eta)}
(-\gamma)^{-\frac{1}{q}}B_{(k,f)}\varepsilon
\left[ \mathcal{H}^{n-1}
\left(\mathbb{S}^{n-1}\right)\right]^{\frac{1}{q}}\left\|{\bf
1}_{B(\textbf{0},(k+1)R)}\right\|_{X}\\
&\qquad+
(-\gamma)^{-\frac{1}{q}}
\lambda\varepsilon^{\frac\gamma q}\left[ \mathcal{H}^{n-1}
\left(\mathbb{S}^{n-1}\right)\right]^{\frac{1}{q}}
\left\|{\bf 1}_{B({\bf 0},(k+1)R)}\right\|_X.
\end{align*}
Letting $\lambda\to0^{+}$, $\varepsilon\to 0^{+}$, and then
$\eta\to0^+$
and applying Definition \ref{def-X}(iv), we obtain
\begin{align*}
&\limsup_{\lambda\to 0^+}\lambda \left\|\left[
\int_{{\mathbb{R}^n}}{\bf
1}_{E_{\lambda,\frac{\gamma}{q},k}[f]}(\cdot,h)
|h|^{\gamma-n}\, dh\right]^{\frac{1}{q}}{\bf
1}_{B(\textbf{0},(k+1)R)}\right\|_X\\
&\quad \le
|\gamma|^{-\frac{1}{q}}\left\|\left[ \int_{\mathbb{S}^{n-1}}
\left|\sum_{\alpha\in\mathbb{Z}_{+}^n,\,|\alpha|=k}\partial^{\alpha}f(\cdot)\xi^{\alpha}\right|^q\,d\mathcal{H}^{n-1}(\xi)
\right]^{\frac{1}{q}}\right\|_X.
\end{align*}
To show (i), it remains to prove that
\begin{equation}\label{eq-upj}
\limsup_{\lambda\to 0^+}\lambda \left\|\left[
\int_{{\mathbb{R}^n}}{\bf
1}_{E_{\lambda,\frac{\gamma}{q},k}[f]}(\cdot,h)
|h|^{\gamma-n}\, dh\right]^{\frac{1}{q}}{\bf
1}_{B(\textbf{0},(k+1)R)^{\complement}}\right\|_X = 0.
\end{equation}
Notice that, for any given $x\in B({\bf 0},(k+1)R)^{\complement}$
and $h\in \mathbb{R}^n\setminus \{{\bf 0}\}$, if $x+ih\in B({\bf
0},R)^{\complement}$ for any $i\in \mathbb{N}\cap [1,k]$,
then $\Delta^k_h f(x)=0$. Thus,
\begin{align*}
&E_{\lambda,\frac{\gamma}{q},k}[f]
\cap \left[B\left({\bf 0},(k+1)R\right)^\complement \times
\mathbb{R}^n\right]\\
&\quad\subset
\bigcup_{i=1}^{k}
\left\{(x,h)\in \left[B\left({\bf 0},(k+1)R\right)^\complement \times
\mathbb{R}^n\right]:\ x+ih \in B({\bf 0},R)\right\}
=: \bigcup_{i=1}^{k} D_i.\notag
\end{align*}
Also observe that, for any $i \in \mathbb{N}$ and
$(x,h)\in D_i$,
\begin{align}\label{e-ih}
i|h|\ge |x|-|x+ih|>\frac{k}{k+1}|x|.
\end{align}
Therefore, from these and a change of variables,
we deduce that
\begin{align}\label{e-kBR}
&\left\|\left[
\int_{{\mathbb{R}^n}}{\bf
1}_{E_{\lambda,\frac{\gamma}{q},k}[f]}(\cdot,h)
|h|^{\gamma-n}\, dh\right]^{\frac{1}{q}}{\bf
1}_{B(\textbf{0},(k+1)R)^{\complement}}\right\|_X\notag\\
&\quad \le
\sum_{i=1}^k \left\|\left[
\int_{{\mathbb{R}^n}}{\bf 1}_{D_i}(\cdot,h)
|h|^{\gamma-n}\, dh\right]^{\frac{1}{q}}{\bf
1}_{B(\textbf{0},(k+1)R)^{\complement}}\right\|_X\notag\\
&\quad \le
\sum_{i=1}^k \left\|\left[
\int_{{\mathbb{R}^n}}{\bf 1}_{D_i}(\cdot,h)
\,
dh\right]^{\frac{1}{q}}\left[\frac{k}{i(k+1)}|\cdot|\right]^{\frac{\gamma-n}{q}}{\bf
1}_{B(\textbf{0},(k+1)R)^{\complement}}\right\|_X\notag\\
&\quad =
\sum_{i=1}^k \left\|\left[i^{-n}
\int_{B({\bf 0},R)}
\,
dy\right]^{\frac{1}{q}}\left[\frac{k}{i(k+1)}|\cdot|\right]^{\frac{\gamma-n}{q}}{\bf
1}_{B(\textbf{0},(k+1)R)^{\complement}}\right\|_X\notag\\
&\quad \lesssim
\left\||\cdot|^{\frac{\gamma-n}{q}}{\bf
1}_{B(\textbf{0},(k+1)R)^{\complement}}\right\|_X.
\end{align}
By a) of the present proposition and
Lemmas \ref{lem-rxp}, \ref{lem-apwight}(ii),
and \ref{lem-rxgmm}, we find that
\begin{align}\label{e-BX}
& \left\||\cdot|^{\frac{\gamma-n}{q}}{\bf
1}_{(B(\textbf{0},(k+1)R))^{\complement}}\right\|_X\notag\\
&\quad \le
\sup_{\|g\|_{(X^{\frac{1}{p}})'}\le 1}
\left[\int_{(B(\textbf{0},(k+1)R))^{\complement}}|x|^
{\frac{p(\gamma-n)}{q}}R_{(X^{\frac{1}{p}})'}g(x)\,dx\right]^{\frac{1}{p}}\notag\\
&\quad =
\sup_{\|g\|_{(X^{\frac{1}{p}})'}\le 1}
\left[\sum_{j=1}^{\infty}\int_{ B(\textbf{0},2^j(k+1)R) \setminus
B(\textbf{0},2^{j-1}(k+1)R)
}|x|^{\frac{p(\gamma-n)}{q}}R_{(X^{\frac{1}{p}})'}g(x)\,dx\right]^{\frac{1}{p}}\notag\\
&\quad \lesssim
\sup_{\|g\|_{(X^{\frac{1}{p}})'}\le 1}
\left[\sum_{j=1}^{\infty}\left(2^{j-1}R\right)^{\frac{p(\gamma-n)}{q}}\int_{
B(\textbf{0},2^j(k+1)R) }
R_{(X^{\frac{1}{p}})'}g(x)\,dx\right]^{\frac{1}{p}}\notag\\
&\quad \le
\sup_{\|g\|_{(X^{\frac{1}{p}})'}\le 1}
\left[\sum_{j=1}^{\infty}\left(2^{j-1}R\right)^{\frac{p(\gamma-n)}{q}}2^{jn}{[R_{(X^{\frac{1}{p}})'}g]_{A_1
}}
\int_{ B(\textbf{0},(k+1)R) }
R_{(X^{\frac{1}{p}})'}g(x)\,dx\right]^{\frac{1}{p}}\notag\\
&\quad \lesssim
\left\|\mathcal{M}\right\|^{\frac{1}{p}}_{(X^{\frac{1}{p}})' \to
(X^{\frac{1}{p}})'}R^{\frac{\gamma-n}{q}}
\left(\sum_{j=1}^{\infty}2^{j[\frac{p(\gamma-n)}{q}+n]}\right)^{\frac{1}{p}}\left\|{\bf
1}_{B(\textbf{0},(k+1)R)}\right\|_X.
\end{align}
Using the assumption $n(\frac{1}{p}-\frac{1}{q})<k$, we obtain, when
$\gamma\in (-\infty,-kq]$,
\begin{align*}	\frac{p(\gamma-n)}{q}+n=\frac{p\gamma}{q}
+n\left(1-\frac{p}{q}\right)<p\left(\frac{\gamma}{q}+k\right)\le 0,
\end{align*}
which, combined with \eqref{e-kBR} and \eqref{e-BX}, further implies
that \eqref{eq-upj} holds for $\gamma\in (-\infty,-kq]$.

Next, we only need to prove \eqref{eq-upj} holds in the case where
$n\in\mathbb{N}\cap [2,\infty)$ and $\gamma\in (-kq,0)$.
In this case, observe that the set $B(\textbf{0},R)^\complement$ is
connected. Thus, we rewrite $f = g_{f} + P_f$, where
$g_{f}\in C^k_{\rm c}$
with ${\mathop\mathrm{\,supp\,}}(g_f) \subset B(\textbf{0},R)$ and
where $P_f \in \mathcal{P}_{k-1}$.
Clearly,
for any $x,h\in \mathbb{R}^n$,
\begin{align*}
\left|\Delta_h^k f(x)\right|
\le 2^k\left\|g_{f}\right\|_{L^\infty}=:C_{(f)}.
\end{align*}
From this and \eqref{e-EE}, it follows that,
for any $\lambda\in (0,\infty)$ and $(x,h)\in
E_{\lambda,\frac{\gamma}{q},k}[f]$,
we have $|h|\in(0,[\lambda^{-1}
C_{(f)}]^{\frac{1}{k+\frac{\gamma}{q}}}).$
Applying this and \eqref{e-ih},
we conclude that, for any $\lambda\in (0,\infty)$ and $(x,h)\in
E_{\lambda,\frac{\gamma}{q},k}[f]
\cap[B\left({\bf 0},(k+1)R\right)^\complement \times \mathbb{R}^n]$,
$$|x|<(k+1)|h|<(k+1) \left[\lambda^{-1}
C_{(f)}\right]^{\frac{1}{k+\frac{\gamma}{q}}}$$
and hence
\begin{align*}
E_{\lambda,\frac{\gamma}{q},k}[f]
\cap \left[B\left({\bf 0},(k+1)R\right)^\complement \times
\mathbb{R}^n\right]
\subset
E_{\lambda,\frac{\gamma}{q},k}[f]
\cap \left\{A_{(\lambda)}\times \mathbb{R}^n\right\},
\end{align*}
where
\begin{align*}
A_{(\lambda)}:=B\left({\bf 0}, (k+1) \left[\lambda^{-1}
C_{(f)}\right]^{\frac{1}{k+\frac{\gamma}{q}}}\right)\setminus
B\left({\bf 0},(k+1)R\right).
\end{align*}
For any $\lambda \in (0,\infty)$, let
$$J_{\lambda}: = \left\lceil\log_2\left[R^{-1}\left\{\lambda^{-1}
C_{(f)} \right\}^{\frac{1}{k+\frac{\gamma}{q}}}
\right]\right\rceil.$$
Repeating the proofs of \eqref{e-kBR} and \eqref{e-BX} with
$B({\bf 0},(k+1)R)^\complement$ replaced by
$A_{(\lambda)}$,
we find that, for any $\lambda\in (0,\infty)$,
\begin{align*}
\lambda\left\|\left[
\int_{{\mathbb{R}^n}}{\bf
1}_{E_{\lambda,\frac{\gamma}{q},k}[f]}(\cdot,h)
|h|^{\gamma-n}\, dh\right]^{\frac{1}{q}}{\bf
1}_{A_{(\lambda)}}\right\|_X
\lesssim \lambda
\left\{\sum_{j=0}^{J_\lambda}2^{j[\frac{p(\gamma-n)}{q}+n]}\right\}^{\frac{1}{p}}
\left\|{\bf 1}_{B(\textbf{0},(k+1)R)}\right\|_{X}.
\end{align*}
Thus, to obtain \eqref{eq-upj}, we only need to show that
\begin{align}\label{eq-uj}
\limsup_{\lambda\to 0^+}\lambda^p\sum_{j=0}^{J_\lambda}
2^{j[\frac{p(\gamma-n)}{q}+n]} = 0.
\end{align}
Indeed, from the assumptions $\gamma\in (- kq,0)$ and
$n(\frac{1}{p}-\frac{1}{q})<k$,
we deduce that
\begin{align*}
\limsup_{\lambda\to 0^+}\lambda^p\sum_{j=0}^{J_\lambda}
2^{j[\frac{p(\gamma-n)}{q}+n]}
&\lesssim
\limsup_{\lambda\to 0^+}
\left\{\lambda^p+
\lambda^{\frac{p(k+\frac{\gamma}{q})-[\frac{p(\gamma-n)}{q}+n]}{k+\frac{\gamma}{q}}}\right\}\\
&= 	\limsup_{\lambda\to 0^+}
\left\{\lambda^p+\lambda^{\frac{p[k-n(\frac{1}{p}-\frac{1}{q})]}{k+\frac{\gamma}{q}}}\right\}
= 0.
\end{align*}
Therefore, \eqref{eq-uj} holds. This proves \eqref{eq-upj}
and hence (i).

Finally, we show (ii).
Repeating the proof of Case 2) of (i), we easily obtain that
(ii) holds,
which then completes the proof of Proposition \ref{pro-limsup}.
\end{proof}

Finally, we give the following auxiliary conclusion.

\begin{lemma}\label{pro-lim}
Let $X$ be a {\rm BQBF} space, $k\in{\mathbb{N}}$,
$q\in (0,\infty)$, and $\gamma\in\mathbb{R}\setminus\{0\}$.
Then there exist two positive
constants $\kappa_{(n,k,q)}$ and
$C_{(n,k,q)}$, depending only on $n$, $k$,
and $q$, such that, for any
$f\in \dot{W}^{k,1}_{\rm loc}$,
\begin{align}\label{eq-kappa}
\kappa_{(n,k,q)}\left\|\nabla^k f \right\|_X
\leq
\left\|\left[ \int_{\mathbb{S}^{n-1}}
\left|\sum_{\alpha\in\mathbb{Z}_{+}^n,\,|\alpha|=k}\partial^{\alpha}f(\cdot)\xi^{\alpha}\right|^q\,d\mathcal{H}^{n-1}(\xi)
\right]^{\frac{1}{q}}\right\|_X\le C_{(n,k,q)}\left\|\nabla^k f
\right\|_X.
\end{align}
\end{lemma}

\begin{proof}
Let $f\in \dot{W}^{k,1}_{\rm loc}$.
From \cite[Remark 5.1]{FKR15}, we infer that there exists a
positive constant $\kappa_{(n,k,q)}\in (0,1]$,
depending only on $n$, $k$,
and $q$, such that, for any $x\in \mathbb{R}^n$,
\begin{align*}
\kappa_{(n,k,q)}\left|\nabla^k f(x)\right|
\leq
\left[ \int_{\mathbb{S}^{n-1}}
\left|\sum_{\alpha\in\mathbb{Z}_{+}^n,\,|\alpha|=k}\partial^{\alpha}f(x)\xi^{\alpha}\right|^q\,d\mathcal{H}^{n-1}(\xi)
\right]^{\frac{1}{q}}.
\end{align*}
This further implies that the first inequality in \eqref{eq-kappa}
holds.
On the other hand,
applying the compactness of
$\mathbb{S}^{n-1}$,
we find that,
for any $\xi\in\mathbb{S}^{n-1} $ and $\alpha\in\mathbb{Z}_{+}^n$ with
$|\alpha|=k$, $|\xi^{\alpha}|^2 \lesssim 1$. Thus, from this and the
H\"{o}lder inequality, it follows that
\begin{align*}
&\left\|\left[ \int_{\mathbb{S}^{n-1}}
\left|\sum_{\alpha\in\mathbb{Z}_{+}^n,\,|\alpha|=k}\partial^{\alpha}f(\cdot)\xi^{\alpha}\right|^q\,d\mathcal{H}^{n-1}(\xi)
\right]^{\frac{1}{q}}\right\|_X\\
&\quad \le
\left\|\,\left|\nabla^k f(\cdot) \right|\left[ \int_{\mathbb{S}^{n-1}}
\left\{\sum_{\alpha\in\mathbb{Z}_{+}^n,\,|\alpha|=k}\left|\xi^{\alpha}\right|^2\right\}^{\frac{q}{2}}\,d\mathcal{H}^{n-1}(\xi)
\right]^{\frac{1}{q}}\right\|_X
\lesssim
\left\|\nabla^k f \right\|_X,
\end{align*}
which further implies \eqref{eq-kappa} holds. This finishes the proof of
Lemma \ref{pro-lim}.
\end{proof}

We now turn to prove Theorem \ref{thm-main}.

\begin{proof}[Proof of Theorem \ref{thm-main}]
We only show (I) of the present theorem
because (II) can be obtained by repeating
the proof of \cite[Theorem 4.10]{dlyyz-bvy} with
$E_f (\lambda ,q)$, $|\nabla f|$, and Theorem 4.5 therein
replaced, respectively, by $E_{\lambda,\frac{\gamma}{q},k}[f]$,
$|\nabla^k f|$, and (I) here.
Let $\upsilon\in A_1$.
We first claim that, for
any $f\in \dot{W}^{k,p}_{\upsilon}$,
\eqref{eq-main-01} with $X:=L^{p}_{\upsilon}$ holds.
Indeed, from the proof of \cite[Theorem 5.15]{dlyyz-bvy}
and Lemma \ref{lem-apwight}(i),
we infer that $X:=L^{p}_{\upsilon}$ has an absolutely continuous
norm and
${[L^{p}_{\upsilon}]}^{\frac{1}{p}}$
is a {\rm BBF} space
and $\mathcal{M}$ is bounded on
$({[L^{p}_{\upsilon}]}^{\frac{1}{p}})'$.
By this, Proposition \ref{pro-linf}(i),
and Lemma \ref{pro-lim}, we find that \eqref{eq-main-01} with
$X:=L^{p}_{\upsilon}$ holds for any
$f\in C^{\infty}$ with $|\nabla^k f|\in C_{\rm
c}$, which, together with Theorem \ref{thm-e} and a
density argument used in
the proof of \cite[(3.55)]{ZYY23},
further implies the lower estimate of \eqref{eq-main-01} with
$X:=L^{p}_{\upsilon}$.
From this and Proposition \ref{pro-up-g} with $\ell:=k$,
we deduce that the above claim holds.

Next, we show (i).
Fix $f\in \dot{W}^{k,X}$ and
let $Y:=(X^{\frac{1}{p}})'$. By the assumptions of the present
theorem that
$X^{\frac{1}{p}}$ is a {\rm BBF} space
and $\mathcal{M}$ is bounded on $(X^{\frac{1}{p}})'$
 and by Lemmas
\ref{lem-rxgmm} and \ref{lem-rxp},
we find that, for any $g\in Y$ with
$\|g\|_{Y} =1$, $R_{Y}g\in A_1$ and
\begin{align*}
\sup_{\|g\|_{Y} =1}
\left[\int_{\mathbb{R}^n}\left|\nabla^k f(x) \right|^p
R_{Y}g(x)\,dx \right]^{\frac{1}{p}}
\sim
\left\|\nabla^k f\right\|_{X},
\end{align*}
which implies $f\in \dot{W}^{k,p}_{R_{Y}g}$.
From this, the above claim, and Lemma \ref{lem-rxp} again, it follows
that
\begin{align*}
\left\|\nabla^k f\right\|_{X}
&\sim
\sup_{\|g\|_Y = 1}
\sup_{\lambda\in (0,\infty)}\lambda
\left\|\left[
\int_{{\mathbb{R}^n}}{\bf
1}_{E_{\lambda,\frac{\gamma}{q},k}[f]}(\cdot,h)
|h|^{\gamma-n}\,
dh\right]^{\frac{1}{q}}\right\|_{L^p_{R_Yg}}\\
&\sim
\sup_{\lambda\in (0,\infty)}\lambda
\left\|\left[
\int_{{\mathbb{R}^n}}{\bf
1}_{E_{\lambda,\frac{\gamma}{q},k}[f]}(\cdot,h)
|h|^{\gamma-n}\, dh\right]^{\frac{1}{q}}\right\|_{X}.
\end{align*}
Thus, we obtain (i).

Finally, we prove (ii).
To this end,
we assume that $X$ has an absolutely continuous norm.
From this, Lemma \ref{rem-weak}, and
the assumptions that $X^{\frac{1}{p}}$ is a {\rm BBF} space
and $\mathcal{M}$ is bounded on $(X^{\frac{1}{p}})'$,
and Theorem \ref{thm-e}(i), we find that
there exists a sequence $\{f_m\}_{m\in {\mathbb{N}}}\subset
C^{\infty}$
with $|\nabla^k f_m|\in C_{\rm c}$
such that \eqref{eq-extension} holds.
Using this and a slight modification of the proof of
\cite[(4.12)]{dlyyz-bvy}, we easily obtain the lower estimate of
\eqref{eq-main-02}.
Therefore, it remains to show the upper estimate of
\eqref{eq-main-02}. To do this, we consider the following two cases on
$p$ and $n$.

\emph{Case 1)} $p=n=1$.
In this case, by the assumptions $\gamma\in (-\infty,-q)$ and
$n(\frac{1}{p}-\frac{1}{q})<k$, Proposition \ref{pro-limsup} for
$f_m$, Lemma \ref{lem-thet},
and the upper estimate of \eqref{eq-main-01},
we find that, for any $m\in\mathbb{N}$
and $\delta,\eta\in(0,1)$, there exists $C_{(\eta)}\in(0,\infty)$ such
that
\begin{align*}
&\limsup_{\lambda\to L_{\gamma}}
\lambda\left\|\left[\int_{\mathbb{R}^n}
\mathbf{1}_{E_{\lambda,\frac{\gamma}{q},k}[f]}(\cdot,h)
\left|h\right|^{\gamma-n}\,dh\right]^\frac{1}{q}\right\|_X\\
&\quad\leq (1+\eta)\limsup_{\lambda\to L_{\gamma}}
\lambda\left\|\left[\int_{\mathbb{R}^n}
\mathbf{1}_{E_{\delta\lambda,\frac{\gamma}{q},k}[f_m]}(\cdot,h)
\left|h\right|^{\gamma-n}\,dh\right]^\frac{1}{q}\right\|_X\\
&\qquad+
C_{(\eta)}\limsup_{\lambda\to L_{\gamma}}
\lambda\left\|\left[\int_{\mathbb{R}^n}
\mathbf{1}_{E_{(1-\delta)\lambda,\frac{\gamma}{q},k}[f-f_m]}(\cdot,h)
\left|h\right|^{\gamma-n}\,dh\right]^\frac{1}{q}\right\|_X
\\
&\quad\leq
(1+\eta)\delta^{-1}
|\gamma|^{-\frac{1}{q}}
\left\|\left[ \int_{\mathbb{S}^{n-1}}
\left|\sum_{\alpha\in\mathbb{Z}_{+}^n,\,|\alpha|=k}\partial^{\alpha}f_m
(\cdot)\xi^{\alpha}\right|^q\,d\mathcal{H}^{n-1}(\xi)
\right]^{\frac{1}{q}}\right\|_X\\
&\quad\quad
+
C(1-\delta)^{-1}C_{(\eta)}
\left\|\nabla^k(f-f_m)\right\|_{X}
,
\end{align*}
where $C$ is the implicit positive constant in the upper estimate of
\eqref{eq-main-01}.
This, combined with \eqref{eq-kappa} and \eqref{eq-extension},
via first letting $m\to\infty$ and then
letting $\eta\to0$ and $\delta\to1$,
further implies the upper estimate of \eqref{eq-main-02} in this case.

\emph{Case 2)} $p\in (1,\infty)$ or $n\in \mathbb{N}\cap [2,\infty)$.
In this case,
from the assumptions that $X^{\frac{1}{p}}$ is a {\rm BBF} space
and $\mathcal{M}$ is bounded on $(X^{\frac{1}{p}})'$,
Lemma \ref{rem-weak}, the assumption the
$X$ has an absolutely continuous norm, and Theorem
\ref{thm-e}(ii), we
infer that there exists
a sequence $\{g_m\}_{m\in\mathbb{N}}\subset C^{\infty}_{\rm
c}$
such that
\begin{align*}
\lim_{m\to \infty}\left\|\nabla^k(f-g_m)\right\|_X=0.
\end{align*}
Using this and repeating the proof of Case 1) with $f_m$ replaced by
$g_m$, we find that
the upper estimate of \eqref{eq-main-02} holds in this case.
This finishes the proof of (ii) and hence Theorem \ref{thm-main}.
\end{proof}

\subsection{Proof of Theorem \ref{t-chaWkX}}\label{ss-cha}

To show Theorem \ref{t-chaWkX}, we need
an interpolation argument.
To state this, we first introduce a new function space.
For any $p,q\in [1,\infty]$, any $\gamma\in \mathbb{R}$, and
any nonnegative
locally integrable function $\upsilon$, the space
$T^{p,q}_{\gamma,\upsilon}$ is defined to
be the set of all $f\in \mathscr{M}(\mathbb{R}^{n}\times
\mathbb{R}^{n})$
such that
\begin{align*}
\|f\|_{T^{p,q}_{\gamma,\upsilon}(\mathbb{R}^{n}\times \mathbb{R}^{n})}
:=
\left\{\int_{\mathbb{R}^n}\left[\int_{\mathbb{R}^n}
\left|f(x,h)\right|^q|h|^{\gamma-n}\,dh\right]^{\frac{p}{q}}\upsilon(x)\,dx\right\}^{\frac{1}{p}},
\end{align*}
with the usual modifications made when $p=\infty$ or $q=\infty$, is finite.
Then we have the following Marcinkiewicz-type interpolation theorem.
Since its proof is a slight
modification of the proof of \cite[Theorem 7.24]{zwyy2021} with the
mixed-norm Lebesgue space $L^{\vec{p}}$ (see Subsection
\ref{ssec-mxlp}
for its precise definition) replaced by
$T^{p,q}_{\gamma,\upsilon}(\mathbb{R}^{n}\times \mathbb{R}^{n})$, we
omit the details.

\begin{lemma}\label{lem-sw-ww}
Let $p,q\in (1,\infty)$, $\gamma\in \mathbb{R}$, $\upsilon$
be a nonnegative locally integrable function,
$r_1\in (\frac{1}{\min\{p,q\}},1)$, and $r_2\in (1,\infty)$.
Assume that $A$ is a sublinear operator defined on
$T^{r_1 p,r_1 q}_{\gamma,\upsilon}(\mathbb{R}^{n}\times
\mathbb{R}^{n})+T^{r_2 p,r_2 q}_{\gamma,\upsilon}(\mathbb{R}^{n}\times
\mathbb{R}^{n})$
satisfying that there exist positive constants $C_1$ and $C_2$ such
that, for any $i\in \{1,2\}$
and $f\in T^{r_i p,r_i q}_{\gamma,\upsilon}(\mathbb{R}^{n}\times
\mathbb{R}^{n})$,
\begin{align*}
\sup_{\lambda\in (0,\infty)}
\lambda\left\|{\bf 1}_{\{(x,h)\in \mathbb{R}^n\times\mathbb{R}^n:\
|A(f(x,h))|>\lambda\}}\right\|_{T^{p,q}_{\gamma,\upsilon}(\mathbb{R}^{n}\times\mathbb{R}^n)}
\le
C_i \left\|f\right\|_{T^{p,q}_{\gamma,\upsilon}
(\mathbb{R}^{n}\times \mathbb{R}^{n})}.
\end{align*}
Then there exists a positive constant $C$ such that, for any
$f\in \mathscr{M}(\mathbb{R}^{n}\times \mathbb{R}^{n})$,
\begin{align*}
&	\sup_{\lambda\in (0,\infty)}
\lambda\left\|{\bf 1}_{\{(x,h)\in \mathbb{R}^n\times\mathbb{R}^n:\
|A(f(x,h))|>\lambda\}}\right\|_{T^{p,q}_{\gamma,\upsilon}(\mathbb{R}^{n}\times\mathbb{R}^n)}\\
&\quad	\le
C\sup_{\lambda\in (0,\infty)}
\lambda\left\|{\bf 1}_{\{(x,h)\in \mathbb{R}^n\times\mathbb{R}^n:\
|f(x,h)|>\lambda\}}\right\|_{T^{p,q}_{\gamma,\upsilon}(\mathbb{R}^{n}\times\mathbb{R}^n)}.\notag
\end{align*}
\end{lemma}

Furthermore, we need the boundedness of
the convolution operator on $T^{p,q}_{\gamma,\upsilon}
(\mathbb{R}^{n}\times \mathbb{R}^{n})$.
To be precise,
for any $\varphi\in C_{\rm c}^{\infty}$,
$f\in \mathscr{M}(\mathbb{R}^{n}\times \mathbb{R}^{n})$,
and $x,h\in
\mathbb{R}^n$, let
\begin{align*}
\mathcal{A}_{\varphi}(f)(x,h):=\int_{\mathbb{R}^n}f(x-y,h)\varphi(y)\,dy.
\end{align*}
Clearly, for any given $\varphi\in C_{\rm c}^{\infty}$,
$\mathcal{A}_{\varphi}$ is
linear.
Moreover,
the \emph{partial
centered maximal operator} $\mathcal{M}_{\mathbb{R}^n}$ is
defined by setting, for any $x,h\in \mathbb{R}^n$,
\begin{align*}
\mathcal{M}_{\mathbb{R}^n}(f)(x,h)
:=
\sup_{R\in (0,\infty)}
\fint_{B(x,R)}|f(y,h)|\,dy.
\end{align*}
Now, we prove the following boundedness result.

\begin{lemma}\label{lem-ss}
Let $p\in (1,\infty)$, $q\in [1,\infty)$, $\gamma\in \mathbb{R}$, and
$\upsilon\in A_1$.
Then there exists a positive
constant $C$ such that, for any $f\in \mathscr{M}(\mathbb{R}^{n}\times
\mathbb{R}^{n})$ and any nonnegative function $\varphi\in C_{\rm
c}^{\infty}$ with
$\int_{\mathbb{R}^n}\varphi(x)\,dx=1$,
\begin{align}\label{e-young}
\left\|\mathcal{A}_{\varphi}(f)\right\|_{T^{p,q}_{\gamma,\upsilon}(\mathbb{R}^{n}\times
\mathbb{R}^{n})}
\le
C
\left\|f\right\|_{T^{p,q}_{\gamma,\upsilon}(\mathbb{R}^{n}\times
\mathbb{R}^{n})}.
\end{align}
\end{lemma}

\begin{proof}
Let $f\in \mathscr{M}(\mathbb{R}^{n}\times \mathbb{R}^{n})$ and
nonnegative function $\varphi\in C_{\rm c}^{\infty}$
with
$\int_{\mathbb{R}^n}\varphi(x)\,dx=1$.
From \cite[Proposition 2.7]{D2001}, we infer that, for any $x,h\in
\mathbb{R}^n$,
\begin{align}\label{e-am}
\left|
\mathcal{A}_{\varphi}(f)(x,h)
\right|
\le
\mathcal{M}_{\mathbb{R}^n}(f)(x,h).
\end{align}
Fix $p\in (1,\infty)$.
We consider the following three cases on $q$.

\emph{Case 1)} $q\in (1,p]$. In this case, let $r:=\frac{p}{q}\in
[1,\infty)$. Take $g\in L^{r'}_{\upsilon^{1-r'}}$
with $\|g\|_{L^{r'}_{\upsilon^{1-r'}}}=1$. By the
assumption $\upsilon\in A_1$
and Lemma \ref{lem-apwight}(v), we conclude that $\upsilon^{1-r'}\in
A_{r'}$,
which, together with Lemma \ref{lem-apwight}(vi), further implies that
\begin{align}\label{e-mg}
\|\mathcal{M}g\|_{L^{r'}_{\upsilon^{1-r'}}}
\lesssim
1.
\end{align}
Let
\begin{align*}
F(\cdot):
=\int_{\mathbb{R}^n}
\mathcal{M}_{\mathbb{R}^n}(f)(\cdot,h)^q|h|^{\gamma-n}\,dh.
\end{align*}
From this, the Tonelli Theorem, \cite[Theorem 2.16]{D2001}, the
assumption $q>1$, the H\"{o}lder inequality,
and \eqref{e-mg}, we deduce that
\begin{align*}
\left|\int_{\mathbb{R}^n}F(x)
g(x)\,dx\right|
&=
\int_{\mathbb{R}^n}\left(\int_{\mathbb{R}^n}
\mathcal{M}_{\mathbb{R}^n}(f)(x,h)^q
g(x)\,dx \right)|h|^{\gamma-n}\,dh\\
&\lesssim
\int_{\mathbb{R}^n}\left(\int_{\mathbb{R}^n}
\left|f(x,h)\right|^q
\mathcal{M}g(x)\,dx \right)|h|^{\gamma-n}\,dh\\
&=
\int_{\mathbb{R}^n}\left(\int_{\mathbb{R}^n}
\left|f(x,h)\right|^q
|h|^{\gamma-n}\,dh\,\upsilon(x)^{\frac{1}{r}}\right)
\left(\mathcal{M}g(x)\upsilon(x)^{-\frac{1}{r}}\right)
\,dx \\
&\le
\left\|f\right\|_{T^{p,q}_{\gamma,\upsilon}(\mathbb{R}^{n}\times
\mathbb{R}^{n})}^{q}
\left\|\mathcal{M}g\right\|_{L^{r'}_{\upsilon^{1-r'}}}
\lesssim
\left\|f\right\|_{T^{p,q}_{\gamma,\upsilon}(\mathbb{R}^{n}\times
\mathbb{R}^{n})}^{q}.
\end{align*}
Using this, \eqref{e-am}, and Lemma \ref{lem-apwight}(v), we find that
\begin{align*}
\left\|\mathcal{A}_{\varphi}(f)\right\|^q_{T^{p,q}_{\gamma,\upsilon}(\mathbb{R}^{n}\times
\mathbb{R}^{n})}
&\le
\left\|\mathcal{M}_{\mathbb{R}^n}(f)\right\|^q_{T^{p,q}_{\gamma,\upsilon}(\mathbb{R}^{n}\times
\mathbb{R}^{n})}
=
\left\|F\right\|_{L^r_{\upsilon}}\\
&=\sup\left\{\left|\int_{\mathbb{R}^n}F(x)g(x)\, dx\right|:\
\left\|g\right\|_{L^{r'}_{\upsilon^{1-r'}}}=1\right\}
\lesssim
\left\|f\right\|_{T^{p,q}_{\gamma,\upsilon}(\mathbb{R}^{n}\times
\mathbb{R}^{n})}^{q}.
\end{align*}
Thus, we obtain \eqref{e-young} in this case.

\emph{Case 2)} $q=\infty$. In this case,
by \eqref{e-am} and the boundedness
of $\mathcal{M}_{\mathbb{R}^n}$ again, we find that \eqref{e-young}
also holds in this case.

\emph{Case 3)} $q\in (p,\infty)$. In this case,
from the Riesz--Thorin interpolation
theorem
of mixed norm Lebesgue spaces (see \cite[Theorem 8.2]{bp1961}) and the
above two cases, we deduce that
\eqref{e-young} holds in this case.
This then finishes the proof of Lemma \ref{lem-ss}.
\end{proof}

Via Lemmas \ref{lem-sw-ww} and \ref{lem-ss},
we then obtain the following Young inequality
of convolutions in {\rm BBF} spaces, which plays
a key role in the proof of Theorem \ref{t-chaWkX}.

\begin{proposition}\label{p-young}
Let $X$ be a {\rm BBF} space.
Assume that there exists some $p\in (1,\infty)$
such that $X^{\frac{1}{p}}$ is a ball Banach function space
and the Hardy--Littlewood maximal operator
$\mathcal{M}$ is bounded on $(X^{\frac{1}{p}})'$.
Let $k\in \mathbb{N} $, $q\in (1,\infty)$, and $\gamma\in \mathbb{R}$.
Then there exists a positive constant $C$ such that, for
any $f\in X$ and
any nonnegative function
$\varphi\in C_{\rm c}^{\infty}$ with
$\int_{\mathbb{R}^n}\varphi(x)\,dx=1$,
\begin{align*}
&\sup_{\lambda\in (0,\infty)}\lambda\left\|\left[
\int_{{\mathbb{R}^n}}{\bf
1}_{E_{\lambda,\frac{\gamma}{q},k}[\varphi\ast f]}(\cdot,h)
|h|^{\gamma-n}\, dh\right]^{\frac{1}{q}}\right\|_X \\
&\quad\le
C
\sup_{\lambda\in (0,\infty)}\lambda\left\|\left[
\int_{{\mathbb{R}^n}}{\bf
1}_{E_{\lambda,\frac{\gamma}{q},k}[f]}(\cdot,h)
|h|^{\gamma-n}\, dh\right]^{\frac{1}{q}}\right\|_X.
\end{align*}
\end{proposition}

\begin{proof}
Let $\upsilon \in A_1$.
To show the present proposition, we
only need to prove that, for any $f\in L^{1}_{\rm loc}$
and
any nonnegative function $\varphi\in C_{\rm c}^{\infty}$
with
$\int_{\mathbb{R}^n}\varphi(x)\,dx=1$,
\begin{align}\label{p-younge1}
\sup_{\lambda\in (0,\infty)}
\lambda\left\|{\bf 1}_{E_{\lambda,\frac{\gamma}{q},k}[\varphi\ast
f]}\right\|_{T^{p,q}_{\gamma,\upsilon}(\mathbb{R}^{n}\times
\mathbb{R}^{n})}
\lesssim
\sup_{\lambda\in (0,\infty)}
\lambda\left\|{\bf 1}_{E_{\lambda,\frac{\gamma}{q},k}[
f]}\right\|_{T^{p,q}_{\gamma,\upsilon}(\mathbb{R}^{n}\times
\mathbb{R}^{n})}.
\end{align}
Indeed, if we assume \eqref{p-younge1} holds for a momoent,
then, applying the present assumption of $X$ and repreating
the argument that used in the proof of \cite[(4.10)]{dlyyz-bvy},
we conclude that \eqref{p-younge1} also holds with
$L^p_\upsilon$ replaced by $X$.

Now, we turn to show \eqref{p-younge1}.
Fix $f\in L^{1}_{\rm loc}$ and a nonnegative function
$\varphi\in C_{\rm c}^{\infty}$ with
$\int_{\mathbb{R}^n}\varphi(x)\,dx=1$
and, for any $x\in\mathbb{R}^n$ and $h\in \mathbb{R}^n\setminus \{{\bf
0}\}$,
let
\begin{align*}
F(x,h):=|h|^{-(k+\frac{\gamma}{q})}\left|\Delta^k_h f(x)\right|.
\end{align*}
By Lemma \ref{lem-ss}, we have,
for any given $r_1 \in (\frac{1}{\min\{p,q\}},1)$
and $r_2\in (1,\infty)$ and
for any $i\in\{1,2\}$,
\begin{align*}
\sup_{\lambda\in (0,\infty)}
\lambda\left\|{\bf 1}_{\{(x,h)\in \mathbb{R}^n\times\mathbb{R}^n:\
|\mathcal{
A}_{\varphi}(F)(x,h)|>\lambda\}}\right\|_{T^{r_i p,r_i
q}_{\gamma,\upsilon}(\mathbb{R}^{n}\times\mathbb{R}^n)}
\lesssim
\left\|F\right\|_{T^{r_i p,r_i
q}_{\gamma,\upsilon}(\mathbb{R}^{n}\times\mathbb{R}^n)},
\end{align*}
which, combined with Lemma \ref{lem-sw-ww}, further
implies that
\begin{align}\label{e-wwF}
&\sup_{\lambda\in (0,\infty)}
\lambda\left\|{\bf 1}_{\{(x,h)\in \mathbb{R}^n\times\mathbb{R}^n:\
|\mathcal{
A}_{\varphi}(F)(x,h)|>\lambda\}}\right\|_{T^{p,q}_{\gamma,\upsilon}(\mathbb{R}^{n}\times\mathbb{R}^n)}
\notag\\
&\quad\lesssim
\sup_{\lambda\in (0,\infty)}
\lambda\left\|{\bf 1}_{\{(x,h)\in \mathbb{R}^n\times\mathbb{R}^n:\
|F(x,h)|>\lambda\}}\right\|_{T^{p,q}_{\gamma,\upsilon}(\mathbb{R}^{n}\times\mathbb{R}^n)}.
\end{align}
Observe that, for any $x\in\mathbb{R}^n$ and $h\in
\mathbb{R}^n\setminus \{{\bf 0}\}$,
\begin{align*}
|h|^{-(k+\frac{\gamma}{q})}\left|\Delta^k_h (\varphi\ast f)(x)\right|	
\le
\left|\mathcal{A}_{\varphi}(F)(x,h)\right|.
\end{align*}
From this and \eqref{e-wwF}, it follows that
\begin{align*}
\sup_{\lambda\in (0,\infty)}
\lambda\left\|{\bf 1}_{E_{\lambda,\frac{\gamma}{q},k}[\varphi\ast
f]}\right\|_{T^{p,q}_{\gamma,\upsilon}(\mathbb{R}^{n}\times
\mathbb{R}^{n})}
&\le
\sup_{\lambda\in (0,\infty)}
\lambda\left\|{\bf 1}_{\{(x,h)\in \mathbb{R}^n\times\mathbb{R}^n:\
|\mathcal{		
A}_{\varphi}(F)(x,h)|>\lambda\}}\right\|_{T^{p,q}_{\gamma,\upsilon}(\mathbb{R}^{n}\times\mathbb{R}^n)}\\
&\lesssim
\sup_{\lambda\in (0,\infty)}
\lambda\left\|{\bf 1}_{\{(x,h)\in \mathbb{R}^n\times\mathbb{R}^n:\
|F(x,h)|>\lambda\}}\right\|_{T^{p,q}_{\gamma,\upsilon}(\mathbb{R}^{n}\times\mathbb{R}^n)}\\
&=\sup_{\lambda\in (0,\infty)}
\lambda\left\|{\bf 1}_{E_{\lambda,\frac{\gamma}{q},k}[
f]}\right\|_{T^{p,q}_{\gamma,\upsilon}(\mathbb{R}^{n}\times
\mathbb{R}^{n})}.
\end{align*}
This then finishes the proof of \eqref{p-younge1} and
hence Proposition \ref{p-young}.
\end{proof}

\begin{proof}[Proof of Theorem \ref{t-chaWkX}]
By Proposition \ref{pro-wklo} and Theorem \ref{thm-main}(i), we
immediately obtain
the necessity.
Next, we prove the sufficiency. To this end, let
$f\in L^{1}_{\rm loc}$ satisfy \eqref{e-c-wkx} and let
$\varphi\in C_{\rm c}^{\infty}$ be a nonnegative
function with
$\int_{\mathbb{R}^n}\varphi(x)\,dx=1$ and, for any $t\in (0,\infty)$,
$\varphi_t := t^{-n}\varphi(\frac{\cdot}{t})$.
Notice that, for any $t\in (0,\infty)$, $\varphi_t \ast f\in
C^{\infty}$ and, for any
$\alpha\in\mathbb{Z}_{+}^n$ with $|\alpha|=k$,
$\partial^{\alpha}(\varphi_t \ast f)=(\partial^{\alpha}\varphi_t) \ast
f\in L^{\infty}_{\rm loc}$, which further implies that
$\partial^{\alpha}(\varphi_t \ast f)$ is locally Lipschitz on
$\mathbb{R}^n$.
From this, Lemma \ref{pro-lim}, and Propositions \ref{pro-linf}(ii)
and \ref{p-young},
it follows that, for any bounded open set
$\Omega$ and any $t\in (0,\infty)$,
\begin{align*}
&\left\|\nabla^k(\varphi_t \ast f){\bf 1}_{\Omega}\right\|_X\\
&\quad\lesssim
\liminf_{\lambda\to L_{\gamma}}\lambda\left\|\left[
\int_{{\mathbb{R}^n}}{\bf 1}_{E_{\lambda,\frac{\gamma}{q},k}[\varphi_t
\ast f]}(\cdot,h)
|h|^{\gamma-n}\, dh\right]^{\frac{1}{q}}{\bf 1}_{\Omega}\right\|_X\\
&\quad\leq
\sup_{\lambda
\in (0,\infty)}\lambda\left\|\left[
\int_{{\mathbb{R}^n}}{\bf 1}_{E_{\lambda,\frac{\gamma}{q},k}[\varphi_t
\ast f]}(\cdot,h)
|h|^{\gamma-n}\, dh\right]^{\frac{1}{q}}\right\|_X\\
&\quad\lesssim
\sup_{\lambda
\in (0,\infty)}\lambda\left\|\left[
\int_{{\mathbb{R}^n}}{\bf 1}_{E_{\lambda,\frac{\gamma}{q},k}[
f]}(\cdot,h)
|h|^{\gamma-n}\, dh\right]^{\frac{1}{q}}\right\|_X=:{\rm I}<\infty.
\end{align*}
Using this and Definition \ref{def-X}(iii), we obtain, for any $t\in
(0,\infty)$,
\begin{align}\label{e-vtf}
\left\|\nabla^k(\varphi_t \ast f)\right\|_X
\lesssim
{\rm I}.
\end{align}
Since both $X$ and $X'$ have absolutely continuous norms,
from \cite[Lemma 3.6]{zyy2023bbm0}, we deduce that
$X$ is reflexive.
Using this and Remarks \ref{rem-x-a}(iii) and \ref{r-dual}, we find
that
$X$ is separable,
$X=X^{**}$, and $X^{*}=X'$.
These, together with \eqref{e-vtf} and the Banach--Alaoglu theorem
(see, for instance, \cite[Theorem 3.17]{rudin}), further imply that,
for any given $\alpha\in \mathbb{Z}_{+}^n$ with $|\alpha|=k$, there
exist
$\{t_j\}_{j\in \mathbb{N}}$ and $g_{\alpha}\in X$
such that $t_j\to 0^{+}$ as $j\to \infty$ and, for any $\Phi\in X'$,
\begin{align}\label{e-p-integral}
\int_{\mathbb{R}^n}\partial^{\alpha}(\varphi_{t_j} \ast
f)(x)\Phi(x)\,dx
\to
\int_{\mathbb{R}^n}g_{\alpha}(x)\Phi(x)\,dx
\end{align}
as $j\to \infty$.
By this, Definition \ref{def-X'}, and Remark \ref{rem-x-a},
we conclude that
\begin{align}\label{e-piX}
\left\|g_{\alpha}\right\|_X
&=
\sup_{\|\Phi\|_{X'}=1}
\left|\int_{\mathbb{R}^n}g_{\alpha}(x)\Phi(x)\,dx \right|\notag\\
&=
\sup_{\|\Phi\|_{X'}=1}
\lim_{j\to \infty}
\left|\int_{\mathbb{R}^n}\partial^{\alpha}(\varphi_{t_j} \ast
f)(x)\Phi(x)\,dx \right|\notag\\
&\le
\sup_{\|\Phi\|_{X'}=1}
\sup_{j\in \mathbb{N}}
\left|\int_{\mathbb{R}^n}\partial^{\alpha}(\varphi_{t_j} \ast
f)(x)\Phi(x)\,dx \right|\le
\sup_{j\in \mathbb{N}}	\left\|\nabla^k(\varphi_{t_j} \ast f)\right\|_X
\lesssim
{\rm I}.
\end{align}

On the other hand, from Remark \ref{rem-x-a} and Definition
\ref{def-X}(iv), it follows that $C_{\rm
c}^{\infty}\subset X'$.
By \cite[p.\,27, Theorem 1.2.19]{g2014},
we find that $\varphi_t \ast f\to f$ in $L^{1}_{\rm
loc}$ as $t\to 0^{+}$.
Applying this and
\eqref{e-p-integral}, we conclude that, for any $\phi\in C_{\rm
c}^{\infty}$,
\begin{align*}
\int_{\mathbb{R}^n}f(x)\partial^{\alpha}\phi(x)\,dx
&=
\lim_{j\to \infty}
\int_{\mathbb{R}^n}(\varphi_{t_j} \ast
f)(x)\partial^{\alpha}\phi(x)\,dx \\
&=
(-1)^k
\lim_{j\to \infty}
\int_{\mathbb{R}^n}\partial^{\alpha}(\varphi_{t_j} \ast
f)(x)\phi(x)\,dx
=(-1)^k
\int_{\mathbb{R}^n}g_{\alpha}(x)\phi(x)\,dx ,
\end{align*}
This, combined with the definition and the uniqueness of weak
derivatives (see, for instance, \cite[pp.\,143--144]{eg2015}), further
implies that $\partial^{\alpha}f$ exists and
$\partial^{\alpha}f=g_{\alpha}$.
From this
and \eqref{e-piX}, we deduce that
\begin{align*}
\left\|\nabla^k f\right\|_{X}
\sim
\sum_{\alpha\in
\mathbb{Z}_{+}^n,\,|\alpha|=k}\left\|g_{\alpha}\right\|
\lesssim
{\rm I},
\end{align*}
which further implies that $f\in \dot{W}^{k,X}$
and hence completes the proof of Theorem \ref{t-chaWkX}.
\end{proof}

\subsection{Proof of Theorem \ref{thm-fsGN}}\label{ssec-fract}

\begin{proof}[Proof of Theorem \ref{thm-fsGN}]
Applying an argument similar to that used in the proof of Theorem
\ref{thm-e}(i), we find that, to show the present theorem, it suffices
to prove that,
for any $f\in C^{\infty}$
with $|\nabla^{k}f|\in C_{\rm c}$, \eqref{eq-in-kks},
\eqref{eq-in-kks1}, and
\eqref{eq-app-dd} hold
with $p\in [1,\infty)$, $\upsilon\in A_1$,
and $X:=L^p_{\upsilon}$.
Thus, from now on, let $p\in [1,\infty)$ and $\upsilon\in A_1$
and fix $f\in C^{\infty}$
with $|\nabla^{k}f|\in C_{\rm c}$.

We first show \eqref{eq-in-kks}.
Assume ${q_{0}}\in [1,\infty)$.
By the assumption
$\frac{1}{q}=\frac{1-s}{{q_{0}}}+s$, we obtain, for any $x,h\in
{\mathbb{R}^n}$ with $|h|\neq 0$,
\begin{align}\label{e-ESS}
\frac{|\Delta_{h}^{k}f (x)|}{|h|^{\frac{\gamma}{q}+k-1+s}}
=
\left[\frac{|\Delta_{h}^{k}f (x)|}{
|h|^{\frac{\gamma}{{q_{0}}}+k-1}}\right]^{1-s}
\left[\frac{|\Delta_{h}^{k}f (x)|}{
|h|^{\gamma+k}}\right]^{s}.
\end{align}
From this, it then follows that, for any $\lambda\in (0,\infty)$,
\begin{align*}
E_{\lambda,\frac{\gamma}{q}+s-1,k}[f]
&\subset
E_{A^{-s}\lambda,\frac{\gamma}{q_0}-1,k}[f]
\cup
E_{A^{1-s}\lambda, \gamma,k} [f],
\end{align*}
where $A\in (0,\infty)$ is a constant specified later.
By this, we conclude that,
for any $\lambda\in(0,\infty)$,
\begin{align}\label{1514}
&\lambda\left\|\int_{\mathbb{R}^n}
\mathbf{1}_{
E_{\lambda,\frac{\gamma}{q}+s-1,k}[f]}(\cdot,h)
\left|h\right|^{\gamma-n}\,dh\right\|_{{L^p_{\upsilon}}}^{\frac{1}{q}}\notag\\
&\quad\leq\lambda
\left\|\int_{\mathbb{R}^n}
\mathbf{1}_{E_{A^{-s}\lambda,\frac{\gamma}{q_0}-1,k}[f]}(\cdot,h)
\left|h\right|^{\gamma-n}\,dh\right\|_{{L^p_{\upsilon}}}^{\frac{1}{q}}\notag\\
&\qquad+\lambda\left\|\int_{\mathbb{R}^n}
\mathbf{1}_{E_{A^{1-s}\lambda,\gamma,k}[f]}(\cdot,h)
\left|h\right|^{\gamma-n}\,dh\right\|_{{L^p_{\upsilon}}}^{\frac{1}{q}}\notag\\
&\quad\leq\lambda^{1-\frac{{q_{0}}}{q}}
\left(A^s\mathrm{G}\right)^{\frac{{q_{0}}}{q}}
+\lambda^{1-\frac{1}{q}}
\left(A^{s-1}\mathrm{H}\right)^{\frac{1}{q}},
\end{align}
where
$$
\mathrm{G}:=\sup_{\lambda\in(0,\infty)}\lambda
\left\|\int_{\mathbb{R}^n}
\mathbf{1}_{E_{\lambda,\frac{\gamma}{q_0}-1,k}[f]}(\cdot,h)
\left|h\right|^{\gamma-n}\,dh\right\|_{{L^p_{\upsilon}}}^{\frac{1}{{q_{0}}}}
$$
and
\begin{align}\label{eq-HH}
\mathrm{H}:=\sup_{\lambda\in(0,\infty)}\lambda
\left\|\int_{\mathbb{R}^n}
\mathbf{1}_{E_{\lambda,\gamma,k}[f]}(\cdot,h)
\left|h\right|^{\gamma-n}\,dh\right\|_{{L^p_{\upsilon}}}.
\end{align}
Choose $A\in(0,\infty)$ such that
$$
\lambda^{1-\frac{{q_{0}}}{q}}
\left(A^s\mathrm{G}\right)^{\frac{{q_{0}}}{q}}
=\lambda^{1-\frac{1}{q}}
\left(A^{s-1}\mathrm{H}\right)^{\frac{1}{q}}.
$$
From this, \eqref{1514}, and
the assumption
$\frac{1}{q}=\frac{1-s}{{q_{0}}}+s$,
we infer that
\begin{align}\label{eq-app-01}
&\lambda\left\|\int_{\mathbb{R}^n}
\mathbf{1}_{
E_{\lambda,\frac{\gamma}{q}+s-1,k}[f]}(\cdot,h)
\left|h\right|^{\gamma-n}\,dh\right\|_{{L^p_{\upsilon}}}^{\frac{1}{q}}
\lesssim\mathrm{G}^{1-s}\mathrm{H}^s.
\end{align}
Applying the assumption $\gamma\in \Gamma_{p,1}$ and Theorem
\ref{thm-up-ap} with $\ell:=k$ and $q:=1$,
we obtain
\begin{align}\label{e-HHH}
{\rm H}	\lesssim\left\|\nabla^k
f\right\|_{L^p_{\upsilon}}.
\end{align}
We then claim that
\begin{align*}
\mathrm{G}\lesssim \left\|\nabla^{k-1}
f\right\|_{{L^{pq_0}_{\upsilon}}}.
\end{align*}
Indeed,
if $k\in \mathbb{N}\cap [2,\infty)$, then,
by the assumption $\gamma\in \Gamma_{p,1}$ and Theorem \ref{thm-up-ap}
with $\ell:=k-1$, $p:=p q_0$, and $q:=q_0$, we find that the above
claim holds in this case; if $k=1$, then,
from a slight modification of the proof of \cite[(4.7)]{ZYY23},
we also deduce that the above claim holds in this case.
Combining \eqref{eq-app-01}, \eqref{e-HHH}, and the above claim,
we find that \eqref{eq-in-kks}
holds.

We next prove \eqref{eq-in-kks1}. To this end, we assume $q_0=\infty$.
Without loss of generality, we may assume
$\|\nabla^{k-1}f\|_{L^\infty}
\in(0,\infty)$.
By \eqref{e-hdo}, we find that,
for any $x,h\in \mathbb{R}^n$ with $h\neq 0$,
\begin{align*}
\frac{|\Delta_h^{k}f(x)|}{|h|^{k-1}}
\lesssim
\left\|\nabla^{k-1}f\right\|_{L^\infty}.
\end{align*}
Thus, there exists a positive
constant $C$, depending
only on $n$ and $k$, such that,
for any $\lambda\in (0,\infty)$ and
$(x,h)\in E_{\lambda,\gamma s+s-1,k}[f]$,
\begin{align*}
\left[\frac{|\Delta^{k}_h f|}{|h|^{\gamma+k}}\right]^{s}
\ge
\frac{|\Delta^{k}_h f|}{|h|^{\gamma s+k-1+s}}
\left[C\left\|\nabla^{k-1}f\right\|_{L^\infty}\right]^{s-1}
>\lambda\left[C\left\|\nabla^{k-1}f\right\|_{L^\infty}\right]^{s-1},
\end{align*}
which further implies that
\begin{align}\label{eq-app-Ds}
E_{\lambda,\gamma s+s-1,k}[f]
\subset E_{{\lambda^\frac{1}{s}}
{[C\|\nabla^{k-1}f
\|_{L^\infty}]^\frac{s-1}{s}},\gamma,k}[f].
\end{align}
Clearly, ${{q_{0}}}=\infty$ implies $qs=1$.
Using this and \eqref{eq-app-Ds},
we find that
\begin{align*}
&\sup_{\lambda\in(0,\infty)}\lambda
\left\|\left[\int_{\mathbb{R}^n}
\mathbf{1}_{E_{\lambda,\frac{\gamma}{q}+s-1,k}[f]}(\cdot,h)
\left|h\right|^{\gamma-n}\,dh
\right]^\frac{1}{q}\right\|_{X^q}\\
&\quad\leq
\left[C\left\|\nabla^{k-1}f\right\|_{L^\infty}\right]^{1-s}
\left[\sup_{\lambda\in(0,\infty)}\lambda
\left\|\int_{\mathbb{R}^n}
\mathbf{1}_{E_{\lambda,\gamma,k}[f]}(\cdot,h)
\left|h\right|^{\gamma-n}\,dh\right\|_{X}\right]^s.\notag
\end{align*}
This, together with Theorem \ref{thm-up-ap}
with $\ell:=k$, implies that \eqref{eq-in-kks1} holds.

Finally, we show
\eqref{eq-app-dd}. By \eqref{eq-app-ss}, we find that,
for any $x,h\in \mathbb{R}^n$ with $|h|\neq 0$,
\begin{align}\label{e-EEGamma}
\frac{|\Delta_h^k f(x)|}{|h|^{k-1+s+\frac{\gamma}{q}}}
=
\left[\frac{|\Delta_h^k f(x)|}{|h|^{k-1+s_0 +\frac{\gamma}{q_0
}}}\right]^{1-\eta}
\left[\frac{|\Delta_h^k f(x)|}{|h|^{k+\gamma}}\right]^{\eta}.
\end{align}
Repeating the proof of \eqref{eq-in-kks} with \eqref{e-ESS}, $s$, and
$E_{A^{-s}\lambda,\frac{\gamma}{q_0}-1,k}[f]$ therein replaced,
respectively,
by \eqref{e-EEGamma}, $\eta$, and
$E_{A^{-s}\lambda,\frac{\gamma}{q_0}+s_0-1,k}[f]$ here, we obtain
\eqref{eq-app-dd}.
This finishes the proof of Theorem \ref{thm-fsGN}.
\end{proof}

\section{Applications to Specific Function Spaces}\label{ssec-fract-s}

In this section, we aim to apply Theorems \ref{thm-main},
\ref{t-chaWkX}, and
\ref{thm-fsGN}, respectively,
to Lebesgue spaces in Subsection \ref{ss-ls},
weighted Lebesgue spaces in Subsection \ref{ss-wls},
(Bourgain--)Morrey type spaces in Subsection \ref{ss-bmts},
local and global generalized Herz spaces in Subsection \ref{ss-lgghs},
mixed-norm Lebesgue spaces in Subsection \ref{ssec-mxlp},
variable Lebesgue spaces in Subsection \ref{ssec-vari-lp},
Lorentz spaces in Subsection \ref{ss-Lorentz},
Orlicz spaces in Subsection \ref{ss-ORL},
and
Orlicz-slice spaces in Subsection \ref{ss-OSS},

\subsection{Lebesgue Spaces}\label{ss-ls}
Let $X:=L^p$ with $p\in[1,\infty)$.
In this case,
it is clear that
all the assumptions of Theorems \ref{thm-main}, \ref{t-chaWkX},
and
\ref{thm-fsGN} hold.
Applying this and Theorems \ref{thm-main}, \ref{t-chaWkX},
and
\ref{thm-fsGN}, we immediately obtain the following results.

\begin{theorem}\label{thm-lp-m}
Let $k\in {\mathbb{N}}$ and $p\in [1,\infty)$.
\begin{enumerate}[{\rm (i)}]
\item Let $q\in (0,\infty)$ satisfy $n(\frac{1}{p}-\frac{1}{q})<k$
and $\gamma\in \Gamma_{p,q}$.
Then, for any $f\in \dot{W}^{k,p}$,
\begin{align}\label{eq-lp}
\sup_{\lambda\in (0,\infty)}
\lambda
\left\{\int_{\mathbb{R}^n}\left[
\int_{{\mathbb{R}^n}}{\bf 1}_{E_{\lambda,\frac{\gamma}{q},k}[f]}(x,h)
|h|^{\gamma-n}\, dh\right]^{\frac{p}{q}}
\,dx\right\}^{\frac{1}{p}}
\sim
\left(\int_{\mathbb{R}^n}\left|\nabla^k f(x)\right|^p
\,dx\right)^{\frac{1}{p}}
\end{align}
with the positive equivalence constants independent of $f$ and
\begin{align*}
&\lim_{\lambda\to L_\gamma}\lambda
\left\{\int_{\mathbb{R}^n}\left[
\int_{{\mathbb{R}^n}}{\bf 1}_{E_{\lambda,\frac{\gamma}{q},k}[f]}(x,h)
|h|^{\gamma-n}\, dh\right]^{\frac{p}{q}}
\,dx\right\}^{\frac{1}{p}}\\
&\quad=
|\gamma|^{-\frac{1}{q}}\left\{
\int_{\mathbb{R}^n}\left[ \int_{\mathbb{S}^{n-1}}
\left|\sum_{\alpha\in\mathbb{Z}_{+}^n,|\alpha|=k}\partial^{\alpha}f(x)\xi^{\alpha}\right|^q\,d\mathcal{H}^{n-1}(\xi)
\right]^{\frac{p}{q}}\,dx\right\}^{\frac{1}{p}}.
\end{align*}

\item Let $q\in (1,\infty)$ satisfy $n(\frac{1}{p}-\frac{1}{q})<k$,
$\gamma\in \mathbb{R}\setminus \{0\}$,
and $\upsilon\in A_p$.
If $p\in(1,\infty)$,
then $f\in \dot{W}^{k,p}$ if and only if
$f\in L^{1}_{\rm loc}$ and
\begin{align*}
\sup_{\lambda\in (0,\infty)}
\lambda
\left\{\int_{\mathbb{R}^n}\left[
\int_{{\mathbb{R}^n}}{\bf 1}_{E_{\lambda,\frac{\gamma}{q},k}[f]}(x,h)
|h|^{\gamma-n}\, dh\right]^{\frac{p}{q}}
\,dx\right\}^{\frac{1}{p}}
<\infty;
\end{align*}
moreover, \eqref{eq-lp} holds for any $f\in L^{1}_{\rm
loc}$.

\item 	Let $q_{0}\in[1,\infty]$,
$s\in(0,1)$,
$q\in[1,q_{0}]$ satisfy $\frac{1}{q}=\frac{1-s}{q_{0}}+s$,
and
$\gamma\in \Gamma_{p,1}$.
If $q_0\in [1,\infty)$,
then, for any
$f\in\dot{W}^{k,p}$,
\begin{align*}
&\sup_{\lambda\in (0,\infty)}
\lambda
\left\{\int_{\mathbb{R}^n}\left[
\int_{{\mathbb{R}^n}}{\bf 1}_{E_{\lambda,\frac{\gamma}{q},k}[f]}(x,h)
|h|^{\gamma-n}\, dh\right]^{p}
\upsilon(x)\,dx\right\}^{\frac{1}{pq}}\\
&\quad\lesssim
\left(\int_{\mathbb{R}^n}\left|\nabla^{k-1} f(x)\right|^{pq_0}
\,dx\right)^{\frac{1-s}{pq_0}}
\left(\int_{\mathbb{R}^n}\left|\nabla^k f(x)\right|^p
\,dx\right)^{\frac{s}{p}}
\end{align*}
with the implicit positive constant independent of $f$.
If $q_0=\infty$, then, for any $f\in\dot{W}^{k,p}$,
\begin{align*}
&\sup_{\lambda\in (0,\infty)}
\lambda
\left\{\int_{\mathbb{R}^n}\left[
\int_{{\mathbb{R}^n}}{\bf 1}_{E_{\lambda,\frac{\gamma}{q},k}[f]}(x,h)
|h|^{\gamma-n}\, dh\right]^{p}
\,dx\right\}^{\frac{1}{pq}}\\
&\quad \lesssim
\left\|\nabla^{k-1}f\right\|^{1-s}_{L^\infty}
\left(\int_{\mathbb{R}^n}\left|\nabla^k f(x)\right|^p
\,dx\right)^{\frac{s}{p}}
\end{align*}
with the implicit positive constant independent of $f$.

\item Let
$\eta\in(0,1)$,
$
0\leq s_0<s<1<q<q_0<\infty
$
satisfy \eqref{eq-app-ss},
and
$\gamma\in \Gamma_{p,1}$.
Then, for any
$f\in \dot{W}^{k,p}$,
\begin{align*}
&\sup_{\lambda\in (0,\infty)}
\lambda
\left\{\int_{\mathbb{R}^n}\left[
\int_{{\mathbb{R}^n}}{\bf 1}_{E_{\lambda,\frac{\gamma}{q},k}[f]}(x,h)
|h|^{\gamma-n}\, dh\right]^{p}
\,dx\right\}^{\frac{1}{pq}}\\
&\quad\lesssim
\sup_{\lambda\in(0,\infty)}\lambda
\left\{
\int_{\mathbb{R}^n}
\left[\int_{\mathbb{R}^n}
\mathbf{1}_{E_{\lambda,\frac{\gamma}{q_0}+s_0-1,k}[f]}(x,h)
\left|h\right|^{\gamma-n}\,dh\right]^p
\,dx\right\}^{\frac{1-\eta}{pq_0}}
\left(\int_{\mathbb{R}^n}\left|\nabla^k f(x)\right|^p
\,dx\right)^{\frac{\eta}{p}}
\end{align*}
with the implicit positive constant independent of $f$.
\end{enumerate}
\end{theorem}

\begin{remark}
\begin{enumerate}[{\rm (i)}]
\item
Theorem \ref{thm-lp-m}(i) with $X:=L^p$
when $k=1$ improves
\cite[Theorem 4.3]{lyyzz24}
because Theorem \ref{thm-lp-m}(i) removes some extra
assumptions on $\gamma$ in \cite[Theorem 4.3]{lyyzz24}.

\item Theorem \ref{thm-lp-m}(ii) when $k=1$ and $p=q$
exactly coincides with \cite[Theorem 1.3]{bsvy24}
and, in the other cases, is new.

\item Both (iii) and (iv) of Theorem \ref{thm-lp-m} with
$X:=L^p$ when
$k=1$ coincide with \cite[Corollaries 4.5 and 4.6]{lyyzz24}
and when $k\in \mathbb{N}\cap [2,\infty)$ are
new.
\end{enumerate}
\end{remark}

\subsection{Weighted Lebesgue Spaces}\label{ss-wls}

Let $p\in (0,\infty)$ and $\upsilon$ be a nonnegative locally
integrable function on $\mathbb{R}^n$. Recall that
$L^p_{\upsilon}$
denotes the weighted Lebesgue space; see also Definition
\ref{def-Ap}(ii).
When $X:=L^p_\upsilon$, we simply write
$\dot{W}^{k,p}_{\upsilon}:=\dot{W}^{k,X}$.
If $p\in[1,\infty)$ and $\upsilon\in A_p$,
then we define
\begin{align}\label{df-pv}
p_\upsilon:=
\inf\left\{r\in[1,\infty):\
\upsilon\in A_r\right\}.
\end{align}

\begin{theorem}\label{thm-wls}
Let $k\in {\mathbb{N}}$, $p\in [1,\infty)$, and $\upsilon\in
A_{p}$ with $p_{\upsilon}$ defined as in \eqref{df-pv}.
\begin{enumerate}[{\rm (i)}]
\item Let $q\in (0,\infty)$ satisfy
$n(\frac{p_{\upsilon}}{p}-\frac{1}{q})<k$
and $\gamma\in \Gamma_{p,q}$.
Then, for any $f\in \dot{W}^{k,p}_{\upsilon}$,
\begin{align}\label{eq-LP}
\sup_{\lambda\in (0,\infty)}
\lambda
\left\{\int_{\mathbb{R}^n}\left[
\int_{{\mathbb{R}^n}}{\bf 1}_{E_{\lambda,\frac{\gamma}{q},k}[f]}(x,h)
|h|^{\gamma-n}\, dh\right]^{\frac{p}{q}}
\upsilon(x)\,dx\right\}^{\frac{1}{p}}
\sim
\left(\int_{\mathbb{R}^n}\left|\nabla^k f(x)\right|^p
\upsilon(x)\,dx\right)^{\frac{1}{p}}
\end{align}
with the positive equivalence constants independent of $f$ and
\begin{align*}
&\lim_{\lambda\to L_\gamma}\lambda
\left\{\int_{\mathbb{R}^n}\left[
\int_{{\mathbb{R}^n}}{\bf 1}_{E_{\lambda,\frac{\gamma}{q},k}[f]}(x,h)
|h|^{\gamma-n}\, dh\right]^{\frac{p}{q}}
\upsilon(x)\,dx\right\}^{\frac{1}{p}}\\
&\quad=
|\gamma|^{-\frac{1}{q}}\left\{
\int_{\mathbb{R}^n}\left[ \int_{\mathbb{S}^{n-1}}
\left|\sum_{\alpha\in\mathbb{Z}_{+}^n,|\alpha|=k}\partial^{\alpha}f(x)\xi^{\alpha}\right|^q\,d\mathcal{H}^{n-1}(\xi)
\right]^{\frac{p}{q}}\upsilon(x)\,dx\right\}^{\frac{1}{p}}.
\end{align*}

\item Let $q\in (1,\infty)$ satisfy
$n(\frac{p_{\upsilon}}{p}-\frac{1}{q})<k$,
$\gamma\in \mathbb{R}\setminus \{0\}$,
and $\upsilon\in A_p $.
If $p\in(1,\infty)$,
then $f\in \dot{W}^{k,p}_{\upsilon}$ if and only if
$f\in L^{1}_{\rm loc}$ and
\begin{align*}
\sup_{\lambda\in (0,\infty)}
\lambda
\left\{\int_{\mathbb{R}^n}\left[
\int_{{\mathbb{R}^n}}{\bf 1}_{E_{\lambda,\frac{\gamma}{q},k}[f]}(x,h)
|h|^{\gamma-n}\, dh\right]^{\frac{p}{q}}
\upsilon(x)\,dx\right\}^{\frac{1}{p}}
<\infty;
\end{align*}
moreover, \eqref{eq-LP} holds for any $f\in L^{1}_{\rm
loc}$.

\item 	Let $q_{0}\in[1,\infty]$,
$s\in(0,1)$,
$q\in[1,q_{0}]$ satisfy $\frac{1}{q}=\frac{1-s}{q_{0}}+s$,
and
$\gamma\in \Gamma_{p,1}$.
If $q_0\in [1,\infty)$,
then, for any
$f\in\dot{W}^{k,p}_{\upsilon}$,
\begin{align*}
&\sup_{\lambda\in (0,\infty)}
\lambda
\left\{\int_{\mathbb{R}^n}\left[
\int_{{\mathbb{R}^n}}{\bf 1}_{E_{\lambda,\frac{\gamma}{q},k}[f]}(x,h)
|h|^{\gamma-n}\, dh\right]^{p}
\upsilon(x)\,dx\right\}^{\frac{1}{pq}}\\
&\quad\lesssim
\left(\int_{\mathbb{R}^n}\left|\nabla^{k-1} f(x)\right|^{pq_0}
\upsilon(x)\,dx\right)^{\frac{1-s}{pq_0}}
\left(\int_{\mathbb{R}^n}\left|\nabla^k f(x)\right|^p
\upsilon(x)\,dx\right)^{\frac{s}{p}}
\end{align*}
with the implicit positive constant independent of $f$.
If $q_0=\infty$, then, for any
$f\in\dot{W}^{k,p}_{\upsilon}$,
\begin{align*}
&\sup_{\lambda\in (0,\infty)}
\lambda
\left\{\int_{\mathbb{R}^n}\left[
\int_{{\mathbb{R}^n}}{\bf 1}_{E_{\lambda,\frac{\gamma}{q},k}[f]}(x,h)
|h|^{\gamma-n}\, dh\right]^{p}
\upsilon(x)\,dx\right\}^{\frac{1}{pq}}\\
&\quad \lesssim
\left\|\nabla^{k-1}f\right\|^{1-s}_{L^\infty}
\left(\int_{\mathbb{R}^n}\left|\nabla^k f(x)\right|^p
\upsilon(x)\,dx\right)^{\frac{s}{p}}
\end{align*}
with the implicit positive constant independent of $f$.
\item Let
$\eta\in(0,1)$,
$
0\leq s_0<s<1<q<q_0<\infty
$
satisfy \eqref{eq-app-ss},
and
$\gamma\in \Gamma_{p,1}$.
Then, for any
$f\in \dot{W}^{k,p}_{\upsilon}$,
\begin{align*}
&\sup_{\lambda\in (0,\infty)}
\lambda
\left\{\int_{\mathbb{R}^n}\left[
\int_{{\mathbb{R}^n}}{\bf 1}_{E_{\lambda,\frac{\gamma}{q},k}[f]}(x,h)
|h|^{\gamma-n}\, dh\right]^{p}
\upsilon(x)\,dx\right\}^{\frac{1}{pq}}\\
&\quad\lesssim
\sup_{\lambda\in(0,\infty)}\lambda
\left\{
\int_{\mathbb{R}^n}
\left[\int_{\mathbb{R}^n}
\mathbf{1}_{E_{\lambda,\frac{\gamma}{q_0}+s_0-1,k}[f]}(x,h)
\left|h\right|^{\gamma-n}\,dh\right]^p
\upsilon(x)\,dx\right\}^{\frac{1-\eta}{pq_0}}
\left(\int_{\mathbb{R}^n}\left|\nabla^k f(x)\right|^p
\upsilon(x)\,dx\right)^{\frac{\eta}{p}}
\end{align*}
with the implicit positive constant independent of $f$.
\end{enumerate}
\end{theorem}

\begin{proof}
We first show (i). Using
Lemma \ref{lem-apwight}(i),
we conclude that
$L^p_\upsilon$
under consideration satisfies all the assumptions of
Theorem \ref{thm-main}.
From this and Theorem
\ref{thm-main}, we deduce that, for any $f\in
\dot{W}^{k,p}_{\upsilon}$, \eqref{eq-main-01} and
\eqref{eq-main-02} with $X:=L^p_\upsilon$ hold.
This then finishes the proof of (i).

Next, we prove (ii) through (iv). By the proof of \cite[Theorem 5.14]{ZYY23},
we find
that all the assumptions of Theorems \ref{t-chaWkX} and
\ref{thm-fsGN} with $X:=L^p_\upsilon$ are satisfied.
Therefore, by
this and Theorems \ref{t-chaWkX} and
\ref{thm-fsGN} with $X:=L^p_\upsilon$,
we conclude that (ii) through (iv) hold. This then finishes the proof of
Theorem
\ref{thm-wls}.
\end{proof}

\begin{remark}
\begin{enumerate}[{\rm (i)}]
\item Theorem \ref{thm-wls}(i) when $k=1$
improves \cite[Theorem 5.1]{lyyzz24}
via removing some extra assumptions on $\gamma$ in
\cite[Theorem 5.1]{lyyzz24}. Moreover,
Theorem \ref{thm-wls}(i) when $k\in \mathbb{N} \cap [2,\infty)$ is
completely new.

\item To the best of our knowledge, Theorem \ref{thm-wls}(ii) is new.

\item
Both (iii) and (iv) of
Theorem \ref{thm-lp-m} when
$k=1$ coincide with \cite[Theorem 5.2]{lyyzz24} and when $k\in
\mathbb{N}\cap [2,\infty)$ are new.
\end{enumerate}
\end{remark}

\subsection{(Bourgain--)Morrey Type Spaces}\label{ss-bmts}
To study the regularity of the solution of
partial differential equations, Morrey \cite{m1938} introduced
the Morrey
space [see Definition
\ref{df-43}(i)]. We refer to \cite{a2015,
cf1987,hms2017,hms2020,hss2018,hs2017,jw2009,tyy2019,sdh20201,sdh20202,ysy10}
for its developments
and applications.
In 1991, Bourgain \cite{b91} introduced a new function space
which is exactly a special case of Bourgain--Morrey spaces to study
the Bochner--Riesz multiplier
problem in ${\mathbb R}^3$,
After that, to explore some problems on
nonlinear Schr\"{o}dinger equations, Masaki \cite{m16-09234}
introduced the Bourgain--Morrey space for
the full range of exponents [see Definition
\ref{df-43}(ii)]. Later on,
Bourgain--Morrey spaces play important roles
in the study of some linear and nonlinear
partial differential equations (see,
for instance, \cite{BV07,b98,kpv00,ms18,ms18-2,mvv99})
and their several fundamental real-variable
properties were recently investigated by Hatano
et al. \cite{hnsh22}.
Very recently, via combining the structure
of both Besov spaces (or Triebel--Lizorkin spaces)
and Bourgain--Morrey spaces,
Zhao et al. \cite{zstyy23} and Hu et al. \cite{hly23}
introduced the Besov--Bourgain--Morrey space
[see Definition
\ref{df-43}(iii)] (or the
Triebel--Lizorkin--Bourgain--Morrey space
[see Definition
\ref{df-43}(iv)]), respectively.

For any $j\in{\mathbb Z}$ and
$m:=(m_1,\ldots,m_n) \in {\mathbb Z}^n$,
let \begin{equation*}
Q_{j, m}
:=
2^j\left(m+[0,1)^n\right).
\end{equation*}
be the \emph{dyadic cube} of ${\mathbb{R}^n}$.
We now recall the definitions of aforementioned
Bourgain--Morrey-type spaces
as follows (see, for instance, \cite{hnsh22,hly23,zstyy23}).

\begin{definition}\label{df-43}
Let $0<p\le u\le r\le\infty$ and $\tau\in(0,\infty]$.
\begin{enumerate}
\item[{\rm(i)}] The \emph{Morrey space}
$M^u_p$
is defined to be the
set of all $f\in L^p_{\rm loc}$ such that
\begin{equation*}
\|f\|_{M^u_p}:=
\sup_{j\in\mathbb{Z},\,m\in\mathbb{Z}^n}
\left|Q_{j,m}\right|^{\frac 1u-\frac 1p}
\left\|f{\mathbf{1}}_{Q_{j,m}}
\right\|_{L^{p}}<\infty.
\end{equation*}
\item[{\rm(ii)}] The \emph{Bourgain--Morrey space}
$M^u_{p,r}$ is defined to be the
set of all $f\in L^p_{\rm loc}$ such that
\begin{equation*}
\|f\|_{M^u_{p,r}}
:=\left\{\sum_{j\in{\mathbb Z},\,m\in{\mathbb Z}^n}
\left[\left|Q_{j,m}\right|^{\frac{1}{u}-\frac{1}{p}}
\left\|f{\mathbf{1}}_{Q_{j,m}}
\right\|_{L^{p}}\right]^r\right\}^{\frac{1}{r}},
\end{equation*}
with the usual modification made when $r=\infty$, is finite.
\item[{\rm(iii)}] The \emph{Besov--Bourgain--Morrey space
$M\dot{B}_{p,r}^{u,\tau}$} is defined to be the set
of all $f\in L^p_{\rm loc}$ such that
\begin{align*}
\|f\|_{M\dot{B}_{p,r}^{u,\tau}}:=
\left\{\sum_{j\in{\mathbb Z}}
\left[\sum_{m\in{\mathbb Z}^n}
\left\{\left|Q_{j, m}\right|^{\frac{1}{u}-\frac{1}{p}}
\left\|f{\mathbf{1}}_{Q_{j,m}}\right\|_{L^{p}}
\right\}^r \right]^\frac{\tau}{r}\right\}^\frac{1}{\tau},
\end{align*}
with the usual modifications made when $r=\infty$ or $\tau=\infty$,
is finite.
\item[{\rm(iv)}] The \emph{Triebel--Lizorkin--Bourgain--Morrey space
$M\dot{F}_{p,r}^{u,\tau}$} is defined to be the set
of all $f\in L^p_{\rm loc}$ such that
\begin{align*}
\|f\|_{M\dot{F}_{p,r}^{u,\tau}}
:=\left(\int_{\mathbb{R}^n}
\left\{\int_{0}^{\infty}\left[t^{n(\frac1u
-\frac1p-\frac1r)}\left\|f{\bf 1}_{B(y,t)}
\right\|_{L^p}\right]^\tau\,
\frac{dt}{t}\right\}^{\frac r\tau}\,dy\right)^{\frac1r},
\end{align*}
with the usual modifications made when $r=\infty$ or $\tau=\infty$,
is finite.
\end{enumerate}
\end{definition}

\begin{remark}
It is obvious that
$M^{u}_{p,\infty}=M^u_p$
and $M\dot{B}^{u,r}_{p,r}
=M^u_{p,r}$.
Moreover, from \cite[Proposition 3.6(iii)]{hly23},
we deduce that $M\dot{F}^{u,r}_{p,r}
=M^u_{p,r}$.
\end{remark}

When $X:=M\dot{A}_{p,r}^{u,\tau}$
with $A\in\{B,F\}$,
we simply write $\dot{W}^{k}M\dot{A}_{p,r}^{u,\tau}
:=\dot{W}^{k,X}$.
We have the following result for Morrey-type spaces.

\begin{theorem}\label{1524}
Let $k\in \mathbb{N}$, $1\le p<u< r\le\infty$,
$\tau\in(0,\infty]$, and
$A\in\{B,F\}$.
\begin{enumerate}[{\rm (i)}]
\item Let $q\in(0,\infty)$ satisfy $n(\frac{1}{p}-\frac{1}{q})<k$ and
$\gamma\in \Gamma_{p,q}$. Then, for any $f\in \dot{W}^k
M\dot{A}^{u,\tau}_{p,r}$,
\eqref{eq-main-01} holds with
$X:=M\dot{A}_{p,r}^{u,\tau}$.

\item Let $q\in (1,\infty)$ satisfy $n(\frac{1}{p}-\frac{1}{q})<k$ and
$\gamma\in \mathbb{R}\setminus \{0\}$.
If $p,\tau \in (1,\infty)$, then Theorem \ref{t-chaWkX} holds with
$X:=M\dot{A}_{p,r}^{u,\tau}$.

\item 	Let $s\in(0,1)$,
$q_{0}\in[1,\infty]$,
$q\in[1,q_{0}]$ satisfy $\frac{1}{q}=\frac{1-s}{q_{0}}+s$,
and
$\gamma\in \Gamma_{p,1}$.
If $q_0\in [1,\infty)$,
then, for any $f\in\dot{W}^k M\dot{A}^{u,\tau}_{p,r}$,
\eqref{eq-in-kks} holds with
$X:=M\dot{A}_{p,r}^{u,\tau}$.
If $q_0=\infty$, then, for any $f\in \dot{W}^k
M\dot{A}^{u,\tau}_{p,r}$, \eqref{eq-in-kks1}
holds with $X:=M\dot{A}_{p,r}^{u,\tau}$.

\item Let
$\eta\in(0,1)$,
$
0\leq s_0<s<1<q<q_0<\infty
$
satisfy \eqref{eq-app-ss},
and $\gamma\in \Gamma_{p,1}$.
Then
\eqref{eq-app-dd} holds with
$X:=M\dot{A}_{p,r}^{u,\tau}$.
\end{enumerate}
\end{theorem}

\begin{proof}
Repeating the proof of
\cite[Theorem 5.6]{lyyzz24}
with Proposition 3.10 therein replaced by Theorem
\ref{thm-wls} here,
we obtain the desired results.
This finishes the proof of Theorem \ref{1524}.
\end{proof}

\begin{remark}
\begin{enumerate}[{\rm (i)}]
\item Theorem \ref{1524}(i) when $k=1$ coincides with
\cite[Theorems 5.6]{lyyzz24} and when $k\in \mathbb{N}\cap [2,\infty)$
is new.

\item To the best of our knowledge, Theorem \ref{1524}(ii) is new.

\item Both (iii) and (iv) of Theorem \ref{1524} when
$k=1$ coincide with \cite[Theorem 5.9]{lyyzz24}
and when $k\in \mathbb{N}\cap [2,\infty)$ are new.
\end{enumerate}
\end{remark}

\subsection{Local and Global Generalized Herz Spaces}\label{ss-lgghs}

Recall that the classical Herz space was originally
introduced by Herz \cite{herz} to
study the Bernstein theorem on absolutely
convergent Fourier transforms.
Recently, Rafeiro and Samko \cite{rs2020}
introduced the local and the global generalized Herz spaces
(see Definition \ref{gh})
which generalize the classical
Herz spaces and generalized Morrey type spaces.
For more studies on Herz spaces,
we refer to \cite{gly1998,hy1999,hwyy2023,ly1996,
lyh2320,rs2020,zyz2022}.

Let $\mathbb{R}_+:=(0,\infty)$
and $\omega$ be a nonnegative function on $\mathbb{R}_+$.
Then the function $\omega$ is said to be
\emph{almost increasing}
(resp. \emph{almost decreasing})
on $\mathbb{R}_+$ if there exists
a constant $C\in[1,\infty)$ such that,
for any $t,\tau\in\mathbb{R}_+$ satisfying
$t\leq\tau$ (resp. $t\geq\tau$),
$$\omega(t)\leq C\omega(\tau).$$

\begin{definition}
The \emph{function class} $M(\mathbb{R}_+)$
is defined to be the set of all positive functions
$\omega$ on $\mathbb{R}_+$ such that, for any $0<\delta<N<\infty$,
$$0<\inf_{t\in(\delta,N)}\omega(t)\leq\sup_
{t\in(\delta,N)}\omega(t)<\infty$$ and
there exist four constants $\alpha_{0},
\beta_{0},\alpha_{\infty},\beta_{\infty}\in\mathbb{R}$ such that
\begin{enumerate}
\item[(i)] for any $t\in(0,1]$,
$\omega(t)t^{-\alpha_{0}}$ is almost increasing and
$\omega(t)t^{-\beta_{0}}$ is almost decreasing;
\item[(ii)] for any $t\in[1,\infty)$,
$\omega(t)t^{-\alpha_{\infty}}$ is
almost increasing and $\omega(t)t^{-\beta_{\infty}}$ is
almost decreasing.
\end{enumerate}
\end{definition}

We now present the Matuszewska--Orlicz indices
as follows, which were introduced in
\cite{MO,mo65} and characterize the properties of functions at
origin and infinity (see also \cite{lyh2320}).

\begin{definition}
Let $\omega$ be a positive function on $\mathbb{R}_+$. Then
the \emph{Matuszewska--Orlicz indices}
$m_0(\omega)$, $M_0(\omega)$,
$m_\infty(\omega)$, and $M_\infty(\omega)$ of
$\omega$ are defined, respectively, by setting, for
any $h\in(0,\infty)$,
$$m_0(\omega):=\sup_{t\in(0,1)}
\frac{\log\left[\limsup\limits_{h\to0^+}\frac{\omega(ht)}
{\omega(h)}\right]}{\log t},\ M_0(\omega):=\inf_{t\in(0,1)}\frac{\log
\left[
\liminf\limits_{h\to0^+}
\frac{\omega(ht)}{\omega(h)}\right]}{\log t},$$
$$
m_{\infty}(\omega):=\sup_{t\in(1,\infty)}
\frac{\log\left[\liminf\limits_{h\to\infty}
\frac{\omega(ht)}{\omega(h)}\right]}{\log t},\ \mathrm{and}\ M_\infty(\omega)
:=\inf_{t\in(1,\infty)}\frac{\log\left[\limsup\limits_
{h\to\infty}\frac{\omega(ht)}{\omega(h)}\right]}{\log t}.
$$
\end{definition}

The following concept of generalized Herz spaces
was originally introduced by
Rafeiro and Samko in \cite[Definition 2.2]{rs2020}
(see also \cite{lyh2320}).

\begin{definition}\label{gh}
Let $p,r\in(0,\infty]$ and $\omega\in M(\mathbb{R}_+)$.
\begin{enumerate}
\item[{\rm(i)}] Let $\xi\in\mathbb{R}^n$.
The \emph{local generalized Herz space}
$\dot{\mathcal{K}}^{p,r}_{\omega,\xi}$
is defined to be the set of all
$f\in L^p_{\rm loc}(\mathbb{R}^n\setminus\{\xi\})$ such that
\begin{equation*}
\|f\|_{\dot{\mathcal{K}}^{p,r}_{\omega,\xi}
}:=\left\{\sum_{k\in\mathbb{Z}}
\left[\omega(2^{k})\right]^{r}
\left\|f{\bf 1}_{B({\bf 0},2^{k})\setminus B({\bf 0},2^{k-1})}
\right\|_{L^{p}}^{r}\right\}^{\frac{1}{r}}<\infty.
\end{equation*}
\item[{\rm(ii)}] The
\emph{global generalized
Herz space} $\dot{\mathcal{K}}^{p,r}_{\omega}$
is defined to be the set of all
$f\in L^p_{\rm loc}$ such that
$\|f\|_{\dot{\mathcal{K}}^{p,r}_{\omega}}
:=\sup_{\xi\in\mathbb{R}^n}
\|f\|_{\dot{\mathcal{K}}^{p,r}_{\omega,\xi}}<\infty$.
\end{enumerate}
\end{definition}

Let $k\in\mathbb{N}$. When
$X:=\dot{\mathcal{K}}^{p,r}_{\omega,\xi}$
or $X:=\dot{\mathcal{K}}^{p,r}_{\omega}$,
we simply write, respectively,
$$\dot{W}^{k}\dot{\mathcal{K}}^{p,r}_{\omega,\xi}
:=\dot{W}^{k,X}\ \text{or}\
\dot{W}^{k}\dot{\mathcal{K}}^{p,r}_{\omega}
:=\dot{W}^{k,X}.$$

\begin{theorem}\label{herz}
Let $k\in \mathbb{N}$, $p,r\in[1,\infty)$, $q\in(0,\infty)$ satisfy
$n(\frac1p-\frac1q)<k$,
$\gamma\in\Gamma_{p,q}$,
and $\omega\in M(\mathbb{R}_+)$
satisfy
\begin{align}\label{e-MR}
-\frac np<m_0(\omega)\le M_0(\omega)<\frac n{p'}
\ and \
-\frac np<m_\infty(\omega)\le M_\infty(\omega)<\frac n{p'}.
\end{align}
\begin{enumerate}[{\rm(i)}]
\item For any $\xi\in\mathbb{R}^n$
and $f\in \dot{W}^k\dot{\mathcal{K}}
^{p,r}_{\omega,\xi}$,
\eqref{eq-main-01} and \eqref{eq-main-02} hold with
$X:=\dot{\mathcal{K}}^{p,r}_{\omega,\xi}$.

\item For any $f\in \dot{W}^k\dot{\mathcal{K}}
^{p,r}_{\omega}$,
\begin{align*}
\sup_{\lambda\in(0,\infty)}\lambda
\left\|\left[\int_{\mathbb{R}^n}
{\bf 1}_{E_{\lambda,\frac{\gamma}{q},k}[f]}(\cdot,y)
|\cdot-y|^{\gamma-n}\,dy\right]^{\frac1q}\right\|_
{\dot{\mathcal{K}}^{p,r}_{\omega}}
\sim\left\|\nabla^k f\right\|_
{\dot{\mathcal{K}}^{p,r}_{\omega}},
\end{align*}
where
the positive equivalence constants are independent of $f$.
\end{enumerate}
\end{theorem}

\begin{proof}
From the proof of \cite[Theorem 4.15]{zyy2023bbm0},
we infer that $X:=\dot{\mathcal{K}}^{p,r}_{\omega,\xi}$
satisfies all the assumptions of Theorem \ref{thm-main}.
Applying this and Theorem \ref{thm-main}, we obtain (i)
and (ii).
Repeating the proof of \cite[Theorem 5.15]{lyyzz24}
with (1.6) therein replaced by \eqref{eq-main-01} here, we find
that (iii) holds, which then completes the proof of
Theorem \ref{herz}.
\end{proof}

\begin{theorem}\label{tm-c-hrez}
Let $k\in \mathbb{N}$, $p,r\in(1,\infty)$, $q\in (1,\infty)$
satisfy
$n(\frac1p-\frac1q)<k$,
$\gamma\in\mathbb{R}\setminus \{0\}$,
and $\omega\in M(\mathbb{R}_+)$
satisfy \eqref{e-MR}. Then Theorem \ref{t-chaWkX} holds
with $X:=\dot{\mathcal{K}}^{p,r}_{\omega,\xi}$
or $X:=\dot{\mathcal{K}}^{p,r}_{\omega}$.
\end{theorem}

\begin{proof}
Repeating the proof of \cite[Theorem 5.15]{lyyzz24}
with (1.6) therein replaced by Theorem \ref{t-chaWkX} here, we obtain
Theorem \ref{tm-c-hrez}.
\end{proof}

\begin{theorem}\label{SGN-herz}
Let $k\in \mathbb{N}$, $p,r\in[1,\infty)$, $\gamma\in\Gamma_{p,1}$,
$\omega\in M(\mathbb{R}_+)$
satisfy \eqref{e-MR},
$\xi\in\mathbb{R}^n$,
$X\in\{\dot{\mathcal{K}}^{p,r}_{\omega,\xi},
\dot{\mathcal{K}}^{p,r}_{\omega}\}$,
and $f\in\dot{W}^{k,X}$.
\begin{enumerate}
\item[{\rm(i)}] Let $q_0\in[1,\infty]$,
$s\in(0,1)$, and $q\in[1,q_0]$ satisfy
$\frac 1q=\frac{1-s}{q_0}+s$.
If $q_0\in [1,\infty)$, then
\eqref{eq-in-kks} holds.
If $q_0=\infty$, then \eqref{eq-in-kks1}
holds.

\item[{\rm(ii)}] Let $\eta\in(0,1)$ and
$0\leq s_0<s<1<q<q_0<\infty$
satisfy \eqref{eq-app-ss}. Then \eqref{eq-app-dd} holds.
\end{enumerate}
\end{theorem}

\begin{proof}
Repeating the proof of \cite[Theorem 5.15]{lyyzz24} with Theorem
1.4 therein replaced by Theorem \ref{thm-fsGN} here, we conclude
the desired conclusions,
which completes the proof of Theorem \ref{SGN-herz}.
\end{proof}

\begin{remark}
\begin{enumerate}[{\rm (i)}]
\item
Theorem \ref{herz} when $k=1$ improves \cite[Theorem 5.15]{lyyzz24}
via removing some extra assumptions on $\gamma$
and when $k\in \mathbb{N}\cap [2,\infty)$ is new.

\item To the best of our knowledge, Theorem \ref{tm-c-hrez} is
completely new.

\item Theorem \ref{SGN-herz} when $k=1$ coincides with \cite[Theorem
5.16]{lyyzz24}
and
when $k\in \mathbb{N}\cap [2,\infty)$
is new.
\end{enumerate}
\end{remark}

\subsection{Mixed-Norm Lebesgue Spaces}\label{ssec-mxlp}

For a given vector $\vec{r}:=(r_1,\ldots,r_n)
\in(0,\infty]^n$, the \emph{mixed-norm Lebesgue
space $L^{\vec{r}}$} is defined to be the
set of all $f\in\mathscr{M}$ with
the following finite quasi-norm
\begin{equation*}
\|f\|_{L^{\vec{r}}}:=\left\{\int_{\mathbb{R}}
\cdots\left[\int_{\mathbb{R}}\left|f(x_1,\ldots,
x_n)\right|^{r_1}\,dx_1\right]^{\frac{r_2}{r_1}}
\cdots\,dx_n\right\}^{\frac{1}{r_n}},
\end{equation*}
where the usual modifications are made when $r_i=
\infty$ for some $i\in\{1,\ldots,n\}$.
Throughout this subsection, we always let
$r_-:=\min\{r_1, \ldots, r_n\}$ for any vector
$\vec{r}:=(r_1,\ldots,r_n)
\in(0,\infty]^n$.
The study of mixed-norm Lebesgue spaces
can be traced back to H\"ormander \cite{h1960}
and Benedek and Panzone \cite{bp1961}.
For more studies on mixed-norm Lebesgue spaces,
we refer to \cite{cgn2017,hy2021}.
Moreover,
when $\vec{r}\in(0,\infty)^n$,
from the definition of $L^{\vec{r}}$,
we easily deduce that
$L^{\vec{r}}$
is a {\rm BQBF} space.
But $L^{\vec{r}}$ may not be a quasi-Banach function
space
(see, for instance, \cite[Remark 7.20]{zwyy2021}).
When $X:=L^{\vec{r}}$, we simply write
$\dot{W}^{k,\vec{r}}:=\dot{W}^{k,X}$.

\begin{theorem}\label{thm-mnls}
Let $k\in{\mathbb{N}}$, $\vec{r}:=(r_1,\ldots,r_n)\in(1,\infty)^n$,
and $\gamma\in \mathbb{R}\setminus \{0\}$.
\begin{enumerate}[{\rm (i)}]
\item Let $q\in (0,\infty)$ satisfy
$n(\frac{1}{r_{-}}-\frac{1}{q})<k$. Then, for any $f\in
\dot{W}^{k,\vec{r}}$, \eqref{eq-main-01} and
\eqref{eq-main-02} hold with $X:=L^{\vec{r}}$.

\item
Let $q\in (1,\infty)$ satisfy $n(\frac{1}{r_{-}}-\frac{1}{q})<k$. Then
Theorem \ref{t-chaWkX} holds with
$X:=L^{\vec{r}}$.

\item
Let $q_{0}\in[1,\infty]$,
$s\in(0,1)$, and $q\in[1,q_{0}]$ satisfy
$\frac{1}{q}=\frac{1-s}{q_{0}}+s$.
If $q_0\in [1,\infty)$,
then, for any $f\in \dot{W}^{k,\vec{r}}$,
\eqref{eq-in-kks} holds
with $X:=L^{\vec{r}}$.
If $q_0=\infty$, then, for any $f\in
\dot{W}^{k,\vec{r}}$,
\eqref{eq-in-kks1} holds with $X:=L^{\vec{r}}$.

\item Let $\eta\in(0,1)$ and
$0\leq s_0<s<1<q<q_0<\infty$ satisfy \eqref{eq-app-ss}.
Then, for any $f\in \dot{W}^{k,\vec{r}}$,
\eqref{eq-app-dd} holds with $X:=L^{\vec{r}}$.
\end{enumerate}
\end{theorem}

\begin{proof}
From the proof of \cite[Theorem 5.5]{ZYY23}, it follows that
all the assumptions of Theorems \ref{thm-main}, \ref{t-chaWkX}, and
\ref{thm-fsGN} are satisfied. Applying
this and Theorems \ref{thm-main}, \ref{t-chaWkX}, and
\ref{thm-fsGN}, we obtain the desired
conclusions. This finishes the proof of Theorem \ref{thm-mnls}.
\end{proof}

\begin{remark}
\begin{enumerate}[{\rm (i)}]
\item Theorem \ref{thm-mnls}(i) when $k=1$ improves
\cite[Theorem 5.5]{ZYY23} because Theorem \ref{thm-mnls}(i)
removes some extra assumptions on $\gamma$
in \cite[Theorem 5.5]{ZYY23}. Furthermore,
Theorem \ref{thm-mnls}(i) when $k\in \mathbb{N} \cap [2,\infty)$ is
completely new.

\item To the best of our knowledge,
Theorem \ref{thm-mnls}(ii) is new.

\item Both (iii) and (iv) of Theorem \ref{thm-mnls}
when $k=1$ coincide with \cite[Theorems 5.6 and 5.7]{ZYY23} and when
$k\in \mathbb{N} \cap [2,\infty)$
are new.
\end{enumerate}
\end{remark}

\subsection{Variable Lebesgue Spaces}\label{ssec-vari-lp}

Let $r:\ \mathbb{R}^n\to(0,\infty)$ be a nonnegative
measurable function,
\begin{equation*}
\widetilde{r}_-:=\underset{x\in\mathbb{R}^n}{
\mathop\mathrm{\,ess\,inf\,}}\,r(x),
\ \ \text{and}\ \
\widetilde{r}_+:=\underset{x\in\mathbb{R}^n}{
\mathop\mathrm{\,ess\,sup\,}}\,r(x).
\end{equation*}
Recall that the \emph{variable Lebesgue space
$L^{r(\cdot)}$} associated with the function
$r:\ \mathbb{R}^n\to(0,\infty)$ is defined to be the set
of all $f\in\mathscr{M}$ with the following finite
quasi-norm
\begin{equation*}
\|f\|_{L^{r(\cdot)}}:=\inf\left\{\lambda
\in(0,\infty):\ \int_{\mathbb{R}^n}\left[\frac{|f(x)|}
{\lambda}\right]^{r(x)}\,dx\le1\right\}.
\end{equation*}
By the definition of $L^{r(\cdot)}$,
we can show that $L^{r(\cdot)}$
is a {\rm BQBF} space
(see, for instance, \cite[Section 7.8]{shyy2017}).
In particular,
when $1\leq\widetilde r_-\le \widetilde r_+<\infty$,
$L^{r(\cdot)}$ is
a Banach function space
in the terminology of Bennett and Sharpley \cite{bs1988}
and hence also a {\rm BBF} space.
For more related results on variable Lebesgue spaces,
we refer to
\cite{cf2013,dhr2009,kr1991,ns2012}.
When $X:=L^{{r}(\cdot)}$,
we simply write
$\dot{W}^{k,r(\cdot)}:=\dot{W}^{k,X}$.

A function $r:\ \mathbb{R}^n\to(0,\infty)$ is said to be
\emph{globally
{\rm log}-H\"older continuous} if there exists
$r_{\infty}\in\mathbb{R}$ and a positive constant $C$ such that, for
any
$x,y\in\mathbb{R}^n$,
\begin{equation*}
|r(x)-r(y)|\le \frac{C}{\log(e+\frac{1}{|x-y|})}
\ \ \text{and}\ \
|r(x)-r_\infty|\le \frac{C}{\log(e+|x|)}.
\end{equation*}

\begin{theorem}\label{thm-vls}
Let $k\in{\mathbb{N}}$, $r$ be globally
log-H\"older continuous, and $1\leq\widetilde{r}_-
\leq\widetilde{r}_+<\infty$.
\begin{enumerate}[{\rm (i)}]
\item Let
$q\in (0,\infty)$ satisfy $n(\frac{1}{\widetilde{r}_-}-\frac{1}{q})<k$
and $\gamma\in \Gamma_{\widetilde{r}_-,q}$.
Then, for any $f\in \dot{W}^{k,r(\cdot)}$,
\eqref{eq-main-01} and \eqref{eq-main-02} hold with
$X:=L^{{r}(\cdot)}$.

\item Let $q\in (1,\infty)$ satisfy
$n(\frac{1}{\widetilde{r}_-}-\frac{1}{q})<k$ and $\gamma\in
\mathbb{R}\setminus \{0\}$. If $\widetilde{r}_- >1$, then Theorem
\ref{t-chaWkX} holds with
$X:=L^{r(\cdot)}$.

\item Let
$q_{0}\in[1,\infty]$,
$s\in(0,1)$,
$q\in[1,q_{0}]$ satisfy $\frac{1}{q}=\frac{1-s}{q_{0}}+s$,
and
$\gamma\in \Gamma_{\widetilde{r}_-,1}$.
If $q_0\in [1,\infty)$, then,
for any $f\in \dot{W}^{k,r(\cdot)}$,
\eqref{eq-in-kks} holds with $X:=L^{{r}(\cdot)}$.
If $q_0=\infty$, then, for any $f\in
\dot{W}^{k,r(\cdot)}$, \eqref{eq-in-kks1}
holds with $X:=L^{{r}(\cdot)}$.

\item Let $\eta\in(0,1)$,
$
0\leq s_0<s<1<q<q_0<\infty
$
satisfy \eqref{eq-app-ss},
and $\gamma\in \Gamma_{\widetilde{r}_-,1}$.
Then, for any $f\in \dot{W}^{k,r(\cdot)}$,
\eqref{eq-app-dd} holds with $X:=L^{{r}(\cdot)}$.
\end{enumerate}
\end{theorem}

\begin{proof}
Applying the proof of \cite[Theorem 5.10]{ZYY23}, we conclude that all
the assumptions of Theorems \ref{thm-main}, \ref{t-chaWkX}, and
\ref{thm-fsGN} hold for the variable Lebesgue space
$L^{{r}(\cdot)}$ under
consideration. By this and Theorems \ref{thm-main}, \ref{t-chaWkX},
and \ref{thm-fsGN}
with $X:=L^{{r}(\cdot)}$,
we obtain the desired results. This then finishes the proof of Theorem
\ref{thm-vls}.
\end{proof}

\begin{remark}
\begin{enumerate}[{\rm (i)}]
\item Theorem \ref{thm-vls}(i) when $k=1$ improves \cite[Theorem
5.18(i)]{lyyzz24}
via removing some extra assumptions on $\gamma$ in \cite[Theorem
5.18(i)]{lyyzz24}	
and when $k\in \mathbb{N} \cap [2,\infty)$ is new.

\item To the best of our knowledge, Theorem \ref{thm-vls}(ii) is
completely new.

\item Both (iii) and (iv)
of Theorem \ref{thm-vls} when $k=1$ coincide with,
respectively,
\cite[(iii) and (ii) of Theorem 5.18]{lyyzz24}
and
when $k\in \mathbb{N} \cap [2,\infty)$ are new.
\end{enumerate}
\end{remark}

\subsection{Lorentz Spaces}\label{ss-Lorentz}

Recall that, for any $r,\tau\in(0,\infty)$,
the \emph{Lorentz space $L^{r,\tau}$}
is defined to be the set of all
$f\in\mathscr{M}$ such that
\begin{equation*}
\|f\|_{L^{r,\tau}}
:=\left\{\int_0^{\infty}
\left[t^{\frac{1}{r}}f^*(t)\right]^\tau
\frac{\,dt}{t}\right\}^{\frac{1}{\tau}}
<\infty,
\end{equation*}
where $f^*$ denotes the \emph{decreasing rearrangement of $f$},
defined by setting, for any $t\in[0,\infty)$,
\begin{equation*}
f^*(t):=\inf\left\{s\in(0,\infty):\ \left|
\left\{x\in\mathbb{R}^n:\ |f(x)|>s\right\}\right|\leq t\right\}.
\end{equation*}
We adopt the convention $\inf \emptyset = \infty$.
When $r,\tau\in(0,\infty)$,
$L^{r,\tau}$ is a quasi-Banach function space
and hence a {\rm BQBF} space
(see, for instance, \cite[Theorem 1.4.11]{g2014});
when $r,\tau\in(1,\infty)$,
the Lorentz space $L^{r,\tau}$ is a Banach function
space
and hence a {\rm BBF} space
(see, for instance, \cite[p.\,87]{shyy2017} and \cite[p.\,74]{g2014}).
When $X:=L^{r,\tau}$,
we denote $\dot{W}^{k,X}$
simply by $\dot{W}^{k,L^{r,\tau}}$.
\begin{theorem}\label{thm-lorentzs}
Let $k\in {\mathbb{N}}$,
$r,\tau\in(1,\infty)$, and $\gamma\in\mathbb{R}\setminus\{0\}$.
\begin{enumerate}[{\rm (i)}]
\item Let $q\in (0,\infty)$ satisfy
$n(\frac{1}{\min\{r,\tau\}}-\frac{1}{q})<k$. Then,
for any $f\in \dot{W}^{k,L^{r,\tau}}$,
\eqref{eq-main-01} and \eqref{eq-main-02} hold with
$X:=L^{r,\tau}$.

\item Let $q\in (1,\infty)$ satisfy
$n(\frac{1}{\min\{r,\tau\}}-\frac{1}{q})<k$. Then Theorem
\ref{t-chaWkX} holds with
$X:=L^{r,\tau}$.

\item Let $s\in(0,1)$,
$q_0 \in[1,\infty]$, and
$q\in[1,q_0 ]$ satisfy $\frac{1}{q}=\frac{1-s}{q_0 }+s$.
If $q_0\in [1,\infty)$,
then, for any $f\in \dot{W}^{k,L^{r,\tau}}$,
\eqref{eq-in-kks} holds with $X:=L^{r,\tau}$.
If $q_0=\infty$, then, for any $f\in
\dot{W}^{k,L^{r,\tau}}$, \eqref{eq-in-kks1}
holds with $X:=L^{r,\tau}$.

\item Let $\eta\in(0,1)$ and
$
0\leq s_0<s<1<q<q_0<\infty
$
satisfy \eqref{eq-app-ss}.
Then, for any $f\in \dot{W}^{k,L^{r,\tau}}$,
\eqref{eq-app-dd} with $X:=L^{r,\tau}$.
\end{enumerate}
\end{theorem}

\begin{proof}
Using the proof of \cite[Theorem 5.18]{ZYY23}, we easily find that all
the assumptions of Theorems \ref{thm-main}, \ref{t-chaWkX}, and
\ref{thm-fsGN} with $X:=L^{r,\tau}$ are satisfied. This,
combined with Theorems \ref{thm-main}, \ref{t-chaWkX}, and
\ref{thm-fsGN}, further implies the desired
conclusions, which completes the proof of Theorem \ref{thm-lorentzs}.
\end{proof}

\begin{remark}
\begin{enumerate}[{\rm (i)}]
\item Theorem \ref{thm-lorentzs}(i) when $k=1$ improves \cite[Theorem
5.18]{ZYY23}
via removing some extra assumptions on $\gamma$
and when $k\in \mathbb{N} \cap [2,\infty)$ is new.

\item To the best of our knowledge, Theorem \ref{thm-lorentzs}(ii) is
new.

\item Both
(iii) and (iv) of Theorem \ref{thm-lorentzs} when $k=1$ coincide with,
respectively,
\cite[Theorems 5.19 and 5.20]{ZYY23}
and when $k\in \mathbb{N} \cap [2,\infty)$ are new.
\end{enumerate}
\end{remark}

\subsection{Orlicz Spaces}\label{ss-ORL}

Recall that
a non-decreasing function $\Phi:\ [0,\infty)
\ \to\ [0,\infty)$ is called an \emph{Orlicz function}
if $\Phi$ satisfies that
$\Phi(0)= 0$,
$\Phi(t)\in(0,\infty)$ for any $t\in(0,\infty)$,
and
$\lim_{t\to\infty}\Phi(t)=\infty$.
An Orlicz function $\Phi$ is said to be of \emph{lower}
(resp. \emph{upper}) \emph{type} $r$ for some
$r\in\mathbb{R}$ if there exists a positive constant
$C_{(r)}$ such that,
for any $t\in[0,\infty)$ and
$s\in(0,1)$ [resp. $s\in[1,\infty)$],
\begin{equation*}
\Phi(st)\le C_{(r)} s^r\Phi(t).
\end{equation*}
In the remainder of this subsection, we always assume that
$\Phi$
is an Orlicz function with both positive lower
type $r_{\Phi}^-$ and positive upper type $r_{\Phi}^+$.
The \emph{Orlicz space $L^\Phi$}
is defined to be the set of all $f\in\mathscr{M}$
with the following finite quasi-norm
\begin{equation*}
\|f\|_{L^\Phi}:=\inf\left\{\lambda\in
(0,\infty):\ \int_{\mathbb{R}^n}\Phi\left(\frac{|f(x)|}
{\lambda}\right)\,dx\le1\right\}.
\end{equation*}
It is easy to show that $L^\Phi$
is a quasi-Banach function space
(see \cite[Section 7.6]{shyy2017}).
For more related results on Orlicz spaces,
we refer to \cite{dfmn2021,ns2014,rr2002}.
When $X:=L^{\Phi}$,
we denote $\dot{W}^{k,\Phi}$
simply by
$\dot{W}^{k,X}$. Moreover,
for any Orlicz function $\Phi$ and $q,t\in(0,\infty)$,
let
\begin{align}\label{oq}
\Phi_q(t):=\Phi\left(t^q\right).
\end{align}

\begin{theorem}\label{thm-orlicz}
Let $\Phi$ be an Orlicz function with both
positive lower type $r^-_{\Phi}$
and positive upper type $r^+_\Phi$, $1\leq r^-_{\Phi}\leq
r^+_{\Phi}<\infty$, and
$k\in{\mathbb{N}}$.
\begin{enumerate}[{\rm (i)}]
\item Let
$q\in (0,\infty)$ satisfy $n(\frac{1}{r^-_{\Phi}}-\frac{1}{q})<k$
and $\gamma\in \Gamma_{r^-_{\Phi},q}$.
Then, for any $f\in \dot{W}^{k,\Phi}$,
\eqref{eq-main-01} and \eqref{eq-main-02} hold with
$X:=L^\Phi$.

\item Let
$q\in (1,\infty)$ satisfy $n(\frac{1}{r^-_{\Phi}}-\frac{1}{q})<k$
and $\gamma\in \mathbb{R}\setminus \{0\}$.
If $r^-_{\Phi}>1$, then Theorem
\ref{t-chaWkX} holds with $X:=L^{\Phi}$.

\item Let $s\in(0,1)$,
$q_0 \in[1,\infty]$,
$q\in[1,q_0 ]$ satisfy $\frac{1}{q}=\frac{1-s}{q_0}+s$,
and
$\gamma\in \Gamma_{r^-_{\Phi},1}$.
If $q_0\in [1,\infty)$,
then, for any $f\in \dot{W}^{k,\Phi}$,
\eqref{eq-in-kks} holds with $X:=L^{\Phi}$.
If $q_0=\infty$, then, for any $f\in \dot{W}^{k,\Phi}$,
\eqref{eq-in-kks1} holds with $X:=L^{\Phi}$.

\item Let $\eta\in(0,1)$,
$
0\leq s_0<s<1<q<q_0<\infty
$
satisfy \eqref{eq-app-ss},
and $\gamma\in \Gamma_{r^-_{\Phi},1}$.
Then, for any $f\in \dot{W}^{k,\Phi}$,
\eqref{eq-app-dd} holds with $X:=L^{\Phi}$.
\end{enumerate}
\end{theorem}

\begin{proof}
From the proof of \cite[Theorem 5.23]{ZYY23}, it follows that the
Orlicz space $L^{\Phi}$ under con-
sideration satisfies all the assumptions of Theorems \ref{thm-main},
\ref{t-chaWkX}, and \ref{thm-fsGN}. By this and Theorems
\ref{thm-main}, \ref{t-chaWkX}, and \ref{thm-fsGN},
we obtain the
desired conclusions, which completes the proof of Theorem
\ref{thm-orlicz}.
\end{proof}

\begin{remark}
\begin{enumerate}[{\rm (i)}]
\item Theorem \ref{thm-orlicz}(i) when $k=1$ improves \cite[Theorem
5.20(i)]{lyyzz24}
because Theorem \ref{thm-orlicz}(i) removes some extra assumptions
on $\gamma$
in \cite[Theorem 5.20(i)]{lyyzz24}. Moreover,
Theorem \ref{thm-orlicz}(i) when $k\in \mathbb{N} \cap [2,\infty)$ is
completely new.

\item To the best of our knowledge, Theorem \ref{thm-orlicz}(ii) is
new.

\item Both
(iii) and (iv) of Theorem \ref{thm-orlicz} when $k=1$ coincide with,
respectively,
\cite[Theorem 5.20(iii) and (ii)]{lyyzz24}
and when $k\in \mathbb{N} \cap [2,\infty)$ are new.
\end{enumerate}
\end{remark}

\subsection{Orlicz-Slice Spaces}\label{ss-OSS}
Recall that, for any given $t,r\in(0,\infty)$,
the \emph{Orlicz-slice space}
$(E_\Phi^r)_t$ is defined to be
the set of all $f\in\mathscr{M}$ with the
following finite quasi-norm
\begin{equation*}
\|f\|_{(E_\Phi^r)_t}:=\left\{\int_{\mathbb{R}^n}
\left[\frac{\|f\mathbf{1}_{B(x,t)}\|_{L^\Phi}}
{\|\mathbf{1}_{B(x,t)}\|_{L^\Phi}}\right]
^r\,dx\right\}^{\frac{1}{r}},
\end{equation*}
where $\Phi$
is an Orlicz function with both positive
lower type $r_{\Phi}^-$ and positive upper
type $r_{\Phi}^+$.
The Orlicz-slice space was introduced in
\cite{zyyw2019} as a generalization of both
the slice space of Auscher and Mourgoglou
\cite{am2019,ap2017} and the Wiener amalgam space
in \cite{h2019,h1975,kntyy2007}. From
\cite[Lemma 2.28]{zyyw2019} and \cite[Remark 7.41(i)]{zwyy2021},
it follows that
the Orlicz-slice space $(E_\Phi^r)_t$ is a
{\rm BBF} space, but in general is not a
Banach function space.
When $X:=(E_\Phi^r)_t$,
we simply write $\dot{W}^{k,(E_\Phi^r)_t}
:=\dot{W}^{k,X}$.
\begin{theorem}\label{thm-oss}
Let $\Phi$ be an Orlicz function with both
positive lower type $r^-_{\Phi}$
and positive upper type $r^+_\Phi$, $1\leq r^-_{\Phi}\leq
r^+_{\Phi}<\infty$,
$k\in \mathbb{N}$, $t\in(0,\infty)$, and $r\in[1,\infty)$.
\begin{enumerate}[{\rm (i)}]
\item 	Let
$q\in (0,\infty)$ satisfy $n(\frac{1}{r^-_{\Phi}}-\frac{1}{q})<k$
and $\gamma\in \Gamma_{r^-_{\Phi},q}$.
Then, for any $f\in \dot{W}^{k,(E_\Phi^r)_t}
$, \eqref{eq-main-01}
and \eqref{eq-main-02} hold with $X:=(E_\Phi^r)_t$.

\item Let $q\in(1,\infty)$ satisfy
$n(\frac{1}{r^-_{\Phi}}-\frac{1}{q})<k$
and $\gamma\in \mathbb{R}\setminus \{0\}$.
If $r^-_{\Phi}>1$,
then
Theorem \ref{t-chaWkX} holds with $X:=L^{\Phi}$.

\item Let $s\in(0,1)$,
$q_0 \in[1,\infty]$,
$q\in[1,q_0 ]$ satisfy $\frac{1}{q}=\frac{1-s}{q_0}+s$,
and
$\gamma\in \Gamma_{r^-_{\Phi},1}$.
If $q_0\in [1,\infty)$,
then, for any $f\in \dot{W}^{k,(E_\Phi^r)_t}
$,
\eqref{eq-in-kks} holds with $X:=(E_\Phi^r)_t$.
If $q_0=\infty$, then, for any $f\in \dot{W}^{k,(E_\Phi^r)_t}
$,
\eqref{eq-in-kks1} holds with $X:=(E_\Phi^r)_t$.
\item Let $\eta\in(0,1)$,
$
0\leq s_0<s<1<q<q_0<\infty
$
satisfy \eqref{eq-app-ss},
and $\gamma\in \Gamma_{r^-_{\Phi},1}$.
Then, for any $f\in \dot{W}^{k,(E_\Phi^r)_t}
$,
\eqref{eq-app-dd} holds with $X:=(E_\Phi^r)_t$.
\end{enumerate}
\end{theorem}

\begin{proof}
From the proof of \cite[Theorem 5.28]{ZYY23}, we deduce that the
Orlicz-slice space $(E_\Phi^r)_t$ under consideration
satisfies all the assumptions of
Theorems \ref{thm-main}, \ref{t-chaWkX}, and \ref{thm-fsGN}. By this
and Theorems \ref{thm-main}, \ref{t-chaWkX}, and \ref{thm-fsGN},
we obtain the
desired results. This finishes the proof of Theorem \ref{thm-oss}.
\end{proof}

\begin{remark}
\begin{enumerate}[{\rm (i)}]
\item Theorem \ref{thm-oss}(i) when $k=1$ improves \cite[Theorem
5.22(i)]{lyyzz24}
via removing some extra assumptions on $\gamma$
and when $k\in \mathbb{N} \cap [2,\infty)$ is new.

\item Theorem \ref{thm-oss}(ii) is new.

\item Both
(iii) and (iv) of Theorem \ref{thm-oss} when $k=1$ coincide with,
respectively,
\cite[Theorem 5.22(iii) and (ii)]{lyyzz24}
and when $k\in \mathbb{N} \cap [2,\infty)$ are new.
\end{enumerate}	
\end{remark}

\section*{Declaration of competing interest}

There is no competing interest.

\section*{Data availability}

No data was used for the research described in the article.

\bigskip

\noindent Pingxu Hu, Yinqin Li (Corresponding author), Dachun Yang,
Wen Yuan and Yangyang Zhang.

\medskip

\noindent Laboratory of Mathematics and Complex Systems
(Ministry of Education of China),
School of Mathematical Sciences, Beijing Normal University,
Beijing 100875, The People's Republic of China

\smallskip

\noindent {\it E-mails}: \texttt{pingxuhu@mail.bnu.edu.cn} (P. Hu)

\noindent\phantom{\it E-mails }
\texttt{yinqli@mail.bnu.edu.cn} (Y. Li)

\noindent\phantom{\it E-mails }
\texttt{dcyang@bnu.edu.cn} (D. Yang)

\noindent\phantom {\it E-mails }
\texttt{wenyuan@bnu.edu.cn} (W. Yuan)

\noindent\phantom{\it E-mails:} \texttt{yangyzhang@bnu.edu.cn} (Y.
Zhang)

\end{document}